\documentclass[11pt]{amsart}
\usepackage{mathrsfs}
\usepackage{amsfonts}
\usepackage{amsmath}
\usepackage{amssymb}
\usepackage{amsthm}
\usepackage{enumerate}
\usepackage{color}
\usepackage{geometry}
\usepackage{hyperref}
\usepackage{float}
\usepackage{graphicx}
\usepackage{multirow}
\usepackage[numbers,sort&compress]{natbib}
\usepackage{color}

\allowdisplaybreaks
\numberwithin{equation}{section}

\theoremstyle{plain}
\newtheorem{prop}{Proposition}[section]
\newtheorem{coro}[prop]{Corollary}

\newtheorem{lemm}[prop]{Lemma}

\newtheorem{theorem}[prop]{Theorem}

\theoremstyle{definition}
\newtheorem{defi}[prop]{Definition}

\newtheorem{rema}[prop]{Remark}

\renewcommand\aa{a}
\newcommand\bb{b}

\newcommand\cc{c}
\newcommand\dd{d}
\newcommand\ff{f}
\renewcommand\gg{g}
\newcommand\hh{h}
\newcommand\HS[1]{\leavevmode\null\hspace{#1mm}}

\newcommand\Irr{\mathsf{Irr}}
\newcommand\kk{k}
\newcommand\KK{k}
\newcommand\MU{\mu}
\newcommand\MLbs{\mathsf{MLb}_\mathsf{s}}
\newcommand\MLb{\mathsf{MLb}}
\newcommand\MLie{\mathsf{MLie}}
\newcommand\NF{\mathsf{NF}}
\newcommand\pdots{\mathrel{\HS{0.2}{\cdot}{\cdot}{\cdot}\HS{0.2}}}
\renewcommand\ss{s}
\renewcommand\SS{S}
\newcommand\uu{u}
\newcommand\vv{v}

\newcommand\xx{x}
\newcommand\yy{y}
\newcommand\zz{z}
\newcommand\wdots{, ..., }
\newcommand\XX{X}
\newcommand\Lbs{\mathsf{Lb}_{\mathsf{s}}}
\newcommand\Lb{\mathsf{Lb}}
\newcommand{\chart}{\mathsf{char}}
\newcommand\D[1]{|#1|}
\newcommand{\gsb}{Gr\"{o}bner--Shirshov basis}
\newcommand\Id{\mathsf{Id}}
\newcounter{ITEM}
\newcommand\ITEM[1]{\setcounter{ITEM}{#1}\leavevmode\hbox{\rm(\roman{ITEM})}}
\newcommand\mA{\mathcal{A}}
\newcommand\ov[1]{\overline{#1}}
\newcommand{\sm}{\setminus}
\newcommand{\supp}{\mathsf{supp}}

\title{Gr\"{o}bner--Shirshov bases theory and extensions of Leibniz superalgebras}

\author{YuXiu Bai$^*$}
\address{School of Mathematical Sciences, South China Normal University, Guangzhou 510631, P. R. China}
\email{\small 550291315@qq.com}

\author{Yuqun Chen$^{\sharp}$}
\address{School of Mathematical Sciences, South China Normal University, Guangzhou 510631, P. R. China}
\email{yqchen@scnu.edu.cn}

\thanks{Supported by the NNSF of China (11571121), the NSF of Guangdong Province (2017A030313002) and the Science and Technology Program of Guangzhou (201707010137)}

\thanks{${}^*$Supported by the Innovation Project of Graduate School of South China Normal University}

\thanks{${}^{\sharp}$Corresponding author}


\begin{document}
\renewcommand{\thefootnote}{\fnsymbol{footnote}}
\renewcommand{\thefootnote}{\arabic{footnote}}
\begin{abstract}
In this paper, we elaborate  Gr\"{o}bner--Shirshov bases method for Leibniz (super)algebras. We show that there is a unique reduced Gr\"{o}bner--Shirshov basis for every (graded) ideal of a free Leibniz (super)algebra. As applications, we construct linear bases of free metabelian Leibniz superalgebras and  new linear bases of free metabelian Lie algebras.
We present a complete characterization of extensions of a Leibniz (super)algebra by another Leibniz (super)algebra, where the former is presented by generators and relations.
\end{abstract}

\maketitle

\section{Introduction}\label{Intro}
We recall that a superalgebra over a field~$\kk$ is a vector space~$\mA$ with a direct sum decomposition
$\mA=\mA_0\oplus \mA_1$
together with a bilinear multiplication $(- -)$: $\mA \times \mA\rightarrow\mA$ such that~$(\mA_i  \mA_j)\subseteq \mA_{i+j}  $,
where the subscripts are elements of~$\mathbb{Z}_2$. The \emph{parity}~$\D\xx$ of every element~$\xx$ in~$\mA_0$ is~0 and the parity~$\D\xx$ of every nonzero element~$\xx$ in~$\mA_1$ is 1. A nonempty subset~$S$ of~$\mA$ is called homogeneous if~$S\subseteq \mA_0\cup\mA_1$.
 If a superalgebra~$\mA$ satisfies the Leibniz superidentity
$$(x(yz)) = ((xy)z)-(-1)^{|z||y|}((xz)y),$$
for all elements~$\xx,\yy,\zz$ in~$\mathcal{A}_0\cup \mathcal{A}_1$,
then~$\mA$ is called a \emph{Leibniz superalgebra}~\cite{supernew}. In particular, if a Leibniz superalgebra~$\mA$ equals its even part, i.e., $\mA=\mA_0$, then~$\mA$ is an ordinary \emph{Leibniz algebra}, which was introduced by Bloh~\cite{9} and rediscovered by Loday~\cite{24}.
In mathematics and mathematical physics, Leibniz algebras have interactions and applications in many areas such as: classical or noncommutative differential geometry, vertex operator algebras, structure theory, and representation theory, etc. For additional explanations and motivations we refer to \cite{6,16, 20,des,26,belian,superdes}.

Gr\"{o}bner bases and Gr\"{o}bner--Shirshov bases were invented independently by A.I. Shirshov for ideals of free (commutative, anti-commutative) non-associative algebras \cite{D24, D25}, free Lie algebras \cite{D25} and implicitly free associative algebras \cite{D25} (see also \cite{D2, D3}), by H. Hironaka \cite{D18} for ideals of the power series algebras (both formal and convergent), and by B. Buchberger \cite{D12} for ideals of the polynomial algebras. Gr\"{o}bner bases and Gr\"{o}bner--Shirshov bases theories have been proved to be very useful in different branches of mathematics, including commutative algebra and combinatorial algebra. They are powerful tools concerning the following classical problems: normal form; word problem; conjugacy problem; rewriting system; automaton; embedding theorem; PBW theorem; extension; homology; growth function; Dehn function; complexity; etc. See, for instance, the books~\cite{D1,D11,book20,D13} and the surveys~\cite{D4, D7,D10}.

Let $\mathfrak{A},\mathfrak{B},\mathfrak{E}$ be Leibniz (super)algebras over a  field $k$ and a short exact sequence
\begin{equation}\label{ext}
0\rightarrow \mathfrak{A} \rightarrow  \mathfrak{E}  \rightarrow \mathfrak{B} \rightarrow 0.
\end{equation}
Then $\mathfrak{E}$ is called an \emph{extension} of $\mathfrak{B}$ by $\mathfrak{A}$.

The classification is up to an isomorphism of Leibniz algebras that stabilizes $\mathfrak{A}$ and costabilizes $\mathfrak{B}$ and we denote by $\mathcal{E} (\mathfrak{B},\mathfrak{A})$ the isomorphism classes of all extensions of $\mathfrak{B}$
by $\mathfrak{A}$ up to this equivalence relation.
If $\mathfrak{A}^2=0$, i.e., $\mathfrak{A}$ is abelian, then $\mathcal{E} (\mathfrak{B},\mathfrak{A})\cong HL^2(\mathfrak{B},\mathfrak{A})$, where
$HL^2(\mathfrak{B},\mathfrak{A})$ is the second cohomology group \cite{loday}. Our Corollaries \ref{th2.8} and \ref{th2.8'} give a characterization of $ \mathcal{E} (\mathfrak{B},\mathfrak{A})$.

Extension is significant for cohomology  of Leibniz (super)algebra and there are many results, see \cite{coho1,coho2,Dh,belian}.
As for the theory of group extensions, their interpretation in terms
of cohomology is well known.

In group theory, the Hall's problem (\cite{hall},  p. 228) ``it is difficult to determine the identities leading to
conditions for an extension, where the former group is presented by generators and
relations." has been solved by Yuqun Chen \cite{chen08} by using Gr\"{o}bner-Shirshov bases method.

Inspired by Hall's problem above, we want to find all extensions of $\mathfrak{B}$ by $\mathfrak{A}$, where  in (\ref{ext}) $\mathfrak{B}$ is presented
by generators and relations.

There are many results on  extensions of algebras, in particular, on extensions of associative algebras, Lie (super)algebras, Leibniz (super)algebras, etc, see, for example, \cite{mi,mi05,ch48,fp14,mo53,cq,Dong}. Mostly, they deal with some special cases for extensions. There is no result to determine
conditions for an extension of Leibniz (super)algebras when the former one is presented by generators and relations.

We first establish a Composition-Diamond lemma for Leibniz (super)algebras, see Theorem \ref{cd-lemma for Lb}.
By using Theorem \ref{cd-lemma for Lb}, we construct linear bases of free metabelian Leibniz (super)algebras and those for free metabelian Lie algebras,
 and we give a complete characterization of the extensions of  Leibniz (super)algebras of $\mathfrak{B}$ by $\mathfrak{A}$, where $\mathfrak{B}$ is presented by generators and relations, see Theorem \ref{th2.5}.
Theorems \ref{cd-lemma for Lb} and \ref{th2.5} are the main results in this paper.

Linear bases of free metabelian Leibniz algebras generated by finite sets are obtained in~\cite{mlei}. A linear basis of a free metabelian Lie algebra is constructed in~\cite{mlie}.

The paper is organized as follows. In Section 2, we present a construction of a linear basis $\NF(\XX)$ of the free Leibniz (super)algebra $\Lbs(X)$ generated by $X=\XX_0\cup\XX_1$.
In Section 3, we first introduce a well order on $\NF(\XX)$ and investigate the elements of the graded ideal~$\Id(S)$ of~$\Lbs(X)$ generated by~$S$. We elaborate a theory of Gr\"obner--Shirshov bases for Leibniz (super)algebras and show that there is a unique reduced Gr\"{o}bner--Shirshov basis for every (graded) ideal of a free Leibniz (super)algebra.
 In Section 4, we construct linear bases of free metabelian Leibniz (super)algebras generated by any sets and those for free metabelian Lie algebras.
 In Section 5, we present a complete characterization of extensions of a Leibniz (super)algebra by another Leibniz (super)algebra, where the former is presented by generators and relations. As an application, we give a construction of elements in  $\mathcal{E} (\mathfrak{B},\mathfrak{A})$ when  $\mathfrak{A}^2=0$.
\section{Linear bases for free Leibniz superalgebras}\label{linear-basis-order}

Our aim in this section is to construct linear bases for free Leibniz (super)algebras by applying the Gr\"obner--Shirshov bases method for non-associative algebras. For completeness, we first recall the Composition-Diamond lemma for non-associative algebras, see~\cite{D24}  for more details.

\section{Linear bases of free Leibniz superalgebras}\label{linear-basis-order}

Our aim in this section is to construct linear bases of free Leibniz (super)algebras by applying the Gr\"obner--Shirshov bases method for non-associative algebras. For completeness, we first recall the Composition-Diamond lemma for non-associative algebras, see~\cite{D24}  for more details.

\subsection{Gr\"{o}bner--Shirshov bases method for non-associative algebras}
In the whole paper, we assume that~$(\XX,\prec)$ is a fixed well-ordered set.
We recall that \emph{terms} over~$\XX$ are ``words with brackets''; they are
 defined by the following induction  on length:

\ITEM1 Every element~$\aa$ of~$\XX$ is a term over~$\XX$ of~\emph{length} 1, denoted by~$\ell(\aa)=1$.

\ITEM2 If~$\mu=(\mu_1 \mu_2)$, where~$\mu_1$ and~$\mu_2$ are terms over~$\XX$ of length~$m$ and~$n$ respectively, then~$\mu$ is a term over~$\XX$ of length~$m+n$, and the length of $\mu$ is~$\ell(\mu)=m+n$.

 We denote by~$\XX^{**}$ the set of all terms over~$\XX$.

Let~$\star$ be a letter not in~$\XX$. Then a term~$\mu$ over~$\XX\cup \{\star\}$ is called a \emph{star term} if the star~$\star$ occurs exactly one time in~$\mu$. For instance, for~$\aa,\bb$ in~$\XX$,
 the term~$(((ab)\star)(bb))$ is a star term while $((aa)(\star\star))$ and~$((ab)a)$ are not.

 Let~$\kk\XX^{**}$ be the free non-associative algebra generated by~$\XX$ over a field $\kk$ ($\kk\XX^{**}$ is the vector space over $k$ with the $k$-basis ~$\XX^{**}$).
 Then for each star term $\mu$, and for each~$\ff$ in~$\kk\XX^{**}$, the notation~$\mu_{\star\mapsto \ff}$ means the resulted polynomial by replacing the star in~$\mu$ with~$\ff$. For instance,
 $(((ab)\star)(bb))_{\star\mapsto \ff}=(((ab)f)(bb)).$

A well order~$<$ on~$\XX^{**}$ is called a \emph{monomial order} if for all~$\mu,\nu,\eta$ in~$\XX^{**}$ such that~${\mu<\nu}$, we have~$(\mu\eta)<(\nu\eta)$ and ~$(\eta\mu)<(\eta\nu)$.

Let $<$ be a monomial order on $\XX^{**}$.  Then for each nonzero polynomial
$\ff=\alpha \mu+\sum \alpha_{i}\mu_{i}$
in~$\kk\XX^{**}$,
where $\mu_i<\mu$, $0\neq \alpha, \alpha_i\in \kk$, $\mu, \mu_i\in \XX^{**}$, we call~$\mu$ the \emph{leading monomial} of~$\ff$, denoted  by~$\bar\ff$. Moreover, we call~$\ff$ a \emph{monic} polynomial if~$\alpha=1$. A nonempty subset $S\subseteq \KK\XX^{**}$ is called a \emph{monic} set, if every element in $S$ is monic.

We are now ready to recall  a  method of Gr\"{o}bner--Shirshov bases for non-associative algebras.
Let~$S\subseteq \KK\XX^{**}$ be a monic set. For all~$\ff, \gg$ in~$\SS$, if there is some star term~$\mu$ such that~$\bar\ff=\overline{\mu_{\star\mapsto g}}$, then
$ (\ff, \gg)_{\bar{\ff}}:=\ff-\mu_{\star\mapsto g}$
is called an \emph{inclusion  composition} of~$\SS$.

For any $f,g\in\KK\XX^{**},\ \mu\in\XX^{**}$, we define
$$
f\equiv g\ \ \mod(S,\mu),\ \ \mbox{if } \ f- g=\sum \alpha_{i}(\mu_{i})_{\star\mapsto \ss_i},
$$
where each~$\alpha_{i}$ lies in~$k$, each~$\mu_{i}$ is a star term and each~$s_{i}$ lies in~$S$ such that~$\overline{(\mu_{i})_{\star\mapsto \ss_i}}<\mu$.

A monic set~$S$ is called a \emph{Gr\"{o}bner--Shirshov basis} in $\KK\XX^{**}$ if for every inclusion composition~$(\ff, \gg)_{\bar{\ff}}$ of~$\SS$ we have~$(\ff, \gg)_{\bar{\ff}}\equiv 0 \mod(S,\bar\ff)$.

\ \

{\it\noindent{\bf Composition-Diamond lemma for non-associative algebras} \emph{\cite{D24}}\ \
Let $<$ be a monomial order on $\XX^{**}$, $S\subseteq \KK\XX^{**}$  a monic set and $\operatorname{\Id(S)}$  the ideal of $\KK\XX^{**}$ generated by $S$. Then the set~$S$ is a Gr\"{o}bner--Shirshov basis in $\KK\XX^{**}$ if and only if
the set~$\{\tau\in X^{**}\mid \tau\neq \overline{\mu_{\star\mapsto \ss}}$  for any star term $\mu$ and~$s\in S\}$  is a linear basis of the
quotient algebra~$\KK\XX^{**}/\Id(S)$.
}

\subsection{Linear bases of free Leibniz superalgebras}
In this subsection, we consider free Leibniz (super)algebras as quotient algebras of free non-associative algebras.

Some definitions given below are preliminary to go further.
 Let~$\mA=\mA_0\oplus\mA_1$ be a superalgebra and let~$\mathcal{B}$ be a subspace of~$\mA$. Then $\mathcal{B}$ is called an \emph{ideal} of~$\mA$ if~$(xy),\ (yx)$ lie in $\mathcal{B}$ for all~$x\in\mA,\ y\in \mathcal{B}$.
 A subalgebra (resp. ideal)~$\mathcal{B}$ of~$\mA$  is called \emph{graded} if~$\mathcal{B}$ is a subalgebra (resp. ideal) of~$\mA$ and it contains the homogeneous components of all its elements, i.e., $\mathcal{B}=(\mathcal{B}\cap\mA_0)\oplus (\mathcal{B}\cap\mA_1)$. So if $\mathcal{B}=\mathcal{B}_0\oplus\mathcal{B}_1$ is a graded ideal of~$\mA$, then $\mA/\mathcal{B}=\mA_0/\mathcal{B}_0\oplus\mA_1 /\mathcal{B}_1$ is the quotient superalgebra of~$\mA$ by~$\mathcal{B}$.
Let~$\mA'=\mA'_0\oplus\mA'_1$ be a superalgebra.   
 A map~$\varphi:\ \mA\rightarrow\mA'$ is said to  \emph{preserve parity}, if~$\varphi(\mA_0)\subseteq\mA'_0,\ \varphi(\mA_1)\subseteq\mA'_1$.
Moreover, $\varphi$ is called a \emph{homomorphism}  if~$\varphi$ preserves parity and is a linear map with~$\varphi((xy))=(\varphi(x)\varphi(y))$ for all~$x,y\in\mA$.
Let $S$ be a homogeneous subset of $\mA$, i.e., $S\subseteq \mA_0\cup\mA_1$. Then it is clear that the ideal $\Id(S)$ of $\mA$ generated by $S$ is a graded ideal.

Let~$X=\XX_0\cup \XX_1$ be a nonempty set, where each element of~$\XX_0$ is of parity~$0$ and each element of~$X_1$ is of parity~1.
The parity $|\mu|$ of~$\mu=(\mu_1\mu_2)\in X^{**}$ is defined to be~$|\mu_1|+|\mu_2|$ modulo by~2, where~$|\mu_1|$ and~$|\mu_2|$ are defined by induction on length. Thus, each element in $X^{**}$ has parity 0 or 1. A polynomial
$\ff=\sum \alpha_{i}\mu_{i}$ in ~$kX^{**}$, where each $\alpha_{i}\in k$ and $ \mu_{i}\in X^{**}$, is called {\it homogeneous} if each $\mu_{i}$ has the same parity.
Define
$$
S=\{(\mu(\nu\tau))-((\mu\nu)\tau)+(-1)^{|\nu||\tau|}((\mu\tau)\nu )\mid \mu, \nu, \tau \in \XX^{**}\}.
$$
Then $S$ is a monic homogeneous  subset of~$kX^{**}$. It is clear that the quotient algebra of~$kX^{**}$ by the graded ideal~$\mathsf{Id}(S)$ of~$kX^{**}$ generated by~$S$ is exactly the free Leibniz superalgebra $\Lbs(\XX)$ generated by~$X$, i.e., the Leibniz superalgebra $kX^{**}/\mathsf{Id}(S)$ is a solution of the following universal property: for any Leibniz superalgebra $\mA=\mA_0\oplus\mA_1$, for any map $\phi:\ X=\XX_0\cup \XX_1\rightarrow \mA=\mA_0\oplus\mA_1$ with $\phi(X_0)\subseteq\mA_0$ and $\phi(X_1)\subseteq\mA_1$, there exists a unique homomorphism $\Phi:\ kX^{**}/\mathsf{Id}(S)\rightarrow \mA$ such that $\Phi i=\phi$, where $i:\ X\rightarrow kX^{**}/\mathsf{Id}(S),\ x\mapsto x+\mathsf{Id}(S)$. Thus,
$$
\Lbs(\XX)=kX^{**}/\mathsf{Id}(S).
$$

By applying Composition-Diamond lemma for non-associative algebras, we shall get a linear basis of~$\Lbs(\XX)$.

It is clear that every term~$\mu\in X^{**}$ is of the unique form:
$$
\mu=((\cdots((\aa\mu_1)\mu_2)\cdots)\mu_n),
$$
where $\aa$ lies in~$\XX$ and each~$\mu_i$ is a term of length less than $\ell(\mu)$.  Therefore, we are able to introduce the following order on~$\XX^{\ast\ast}$.

For all~$\mu=((\cdots((\aa\mu_1)\mu_2)\cdots)\mu_n)$ and~$\nu=((\cdots((\bb\nu_1)\nu_2)\cdots)\nu_m)\in \XX^{**}$, define
\begin{align}\label{equa0.0}
\mu<'\nu\ \mbox{ if and only if } (\ell(\mu), \mu_n,\dots, \mu_1, \aa )<(\ell(\nu), \nu_m,\dots, \nu_1,  \bb )\ \ \mbox{lexicogaphically,}
\end{align}
where $\mu_{i},\nu_{j}$ are compared by induction on length.

 It is easy to show that the order~$<'$ is a monomial order on~$\XX^{\ast\ast}$.

Let $(\mA, (--))$ be a Leibniz superalgebra. For all elements~$\ff_1\wdots\ff_n$ in~$\mA$, define
$$
[\ff_1\cdots\ff_n]_{_L}:=((\cdots((\ff_1\ff_2)\ff_3)\cdots)\ff_n)  \ \  \ \mbox{(left-normed bracketing)},
$$
where $[\ff_1\cdots\ff_n]_{_L}:=f_1$ if $n=1$.

For example, let $a_1,a_2,a_3,a_4, f$ lie in $ \mA$.  We  have $[a_1a_2a_3a_4]_{_L}=(((a_1a_2)a_3)a_4), $ $[a_1(a_2a_3)a_4]_{_L}=((a_1(a_2a_3))a_4),$
$ [[a_1a_2a_3]_{_L}a_4]_{_L}=[a_1a_2a_3a_4]_{_L}$,
$ [a_1(a_2(a_3a_4))]_{_L}=(a_1(a_2(a_3a_4)))$, and $ [a_1a_2fa_3a_4]_{_L}=((((a_1a_2)f)a_3)a_4)$.

  Now we are ready to show  that the following set $\NF(\XX)$ is  a linear basis of the free Leibniz superalgebra $\Lbs(X)$ generated by~$X=\XX_0\cup\XX_1$. In particular, if~$X=X_0$, then  $\NF(\XX)$  is exactly  the linear basis of the free Leibniz algebra generated by $X$  constructed in~\cite{loday}.

\begin{theorem}\label{basisnf}
Let $S=\{(\mu(\nu\tau))-((\mu\nu)\tau)+(-1)^{|\nu||\tau|}((\mu\tau)\nu )\mid \mu, \nu, \tau \in \XX^{**}\}$.
Then we have the following claims.

\ITEM1 With the order (\ref{equa0.0}) on $\XX^{**}$, the set $S$ is a Gr\"{o}bner--Shirshov basis in $\KK\XX^{**}$.

\ITEM2  The following set forms a linear basis of~$\Lbs(\XX)$
$$
\NF(\XX):=\{ [\aa_1\cdots\aa_n]_{_L}\mid \aa_{i}\in \XX, 1\leq i\leq n, n\geq 1\}.
$$

\end{theorem}

 \begin{proof}
\ITEM1  We just show that every inclusion composition of the form~$(\ff, \gg)_{\bar{\ff}}$ in $S$ is trivial modulo~$(S, \bar{f})$, where
\begin{align*}
\ff&=(\mu'(\mu(\nu\tau)))-((\mu'\mu)(\nu\tau))+(-1)^{|(\nu\tau)||\mu|}((\mu'(\nu\tau))\mu),\\  \gg&=(\mu(\nu\tau))-((\mu\nu)\tau)+(-1)^{|\tau||\nu|}((\mu\tau)\nu)
\end{align*}
and~$\mu', \mu, \nu, \tau\in \XX^{**}$. By direct calculation, we obtain
\begin{align*}
 (\ff, \gg)_{\bar{\ff}}
=& \ff-(\mu'\gg)& \\
=& -((\mu'\mu)(\nu\tau))+(-1)^{|(\nu\tau)||\mu|}((\mu'(\nu\tau))\mu)
+(\mu'((\mu\nu)\tau))
-(-1)^{|\tau||\nu|}(\mu'((\mu\tau)\nu)) &\\
\equiv& -(((\mu'\mu)\nu)\tau)+(-1)^{|\tau||\nu|}(((\mu'\mu)\tau)\nu)
+(-1)^{|(\nu\tau)||\mu|}(((\mu'\nu)\tau)\mu)&\\
&-(-1)^{|(\nu\tau)||\mu|+|\tau||\nu|}(((\mu'\tau)\nu)\mu)
 +(((\mu'\mu)\nu)\tau)-(-1)^{|\nu||\mu|}(((\mu'\nu)\mu)\tau)&\\
 &-(-1)^{|\tau||(\mu\nu)|}(((\mu'\tau)\mu)\nu))
 +(-1)^{|\tau||(\mu\nu)|+|\mu||\nu|}(((\mu'\tau)\nu)\mu))
 -(-1)^{|\tau||\nu|}(((\mu'\mu)\tau)\nu)&\\
 &+(-1)^{|\tau||\nu|+|\tau||\mu|}
(((\mu'\tau)\mu)\nu)
+(-1)^{|\tau||\nu|+|\nu||(\mu\tau)|}(((\mu'\nu)\mu)\tau)&\\
&-(-1)^{|\tau||\nu|+|\nu||(\mu\tau)|+|\mu||\tau|}(((\mu'\nu)\tau)\mu)&\\
\equiv& 0\mod(S, \bar{\ff}).&
\end{align*}
Similarly, the other inclusion compositions of~$\SS$ are trivial. This shows that $S$ is a Gr\"{o}bner--Shirshov basis in $\KK\XX^{**}$.

Point \ITEM2  follows from Composition-Diamond lemma for non-associative algebras.
 \end{proof}


\section{Composition-Diamond lemma for Leibniz superalgebras}\label{CDlemma-Lei}

In this section we elaborate a method of Gr\"{o}bner-Shirshov bases for Leibniz (super)algebras.  From now on, by Theorem \ref{basisnf}, $\NF(\XX)$ is a linear basis of the free Leibniz superalgebra $\Lbs(\XX)$ generated by $\XX=\XX_0\cup\XX_1$.

Now we introduce an order on~$\NF(\XX)$. Let $\deg:X\rightarrow \mathbb{Z}_{>0}$ be a map.
For each $\mu=[\aa_1\aa_2\cdots\aa_n]_{_L}\in \NF(\XX)$, define $$\deg(\mu)=\sum_{i=1}^n \deg(\aa_i)\ \ \mbox{and }\
 \mathsf{wt}(\mu)=(\deg(\mu), \ell(\mu), \aa_1, \aa_2,\dots, \aa_n).
 $$
For instance, for~$\aa_1\wdots \aa_7\in\XX$,~${\mu=[\aa_1\cdots \aa_7]_{_L}\in\NF(\XX)}$, we have~$\deg(\mu)=\sum_{i=1}^{7}\deg(\aa_i)$, $\ell(\mu)=7 $ and
$\mathsf{wt}(\mu)=(\deg(\mu), 7, \aa_1, \aa_2,\dots, \aa_7)$. We define an order on~$\NF(\XX)$ as follows.
\begin{defi}\label{order}
For all $\mu, \nu\in \NF(\XX)$,
$$\mu<\nu\ \mbox{ if and only if } \mathsf{wt}(\mu)<\mathsf{wt}(\nu)\ \ \mbox{lexicogaphically.}$$
Such an order is called a \emph{deg-length-lex} order (degree-length-lexicographic order).
\end{defi}

In what follows, we use this order on $\NF(\XX)$. It is clear that, for each~$\aa\in \XX$ and for each~$\nu\in \NF(\XX)$, if~$\nu<\aa$ holds, then we have
either~$\deg(\nu)<\deg(\aa) $ or $\deg(\aa)=\deg(\nu)$ with~$ \nu\in \XX$ and~$ \nu<\aa$.
 Moreover, it is not difficult to prove that a deg-length-lex order is a well order on~$\NF(\XX)$.
For each nonzero polynomial~$\ff=\sum_{i=1}^{n} \alpha_{i}\mu_i $ in~$\Lbs(\XX)$, where each~$0\neq \alpha_i$ lies in~$k$, each~$\mu_i$ lies in~$\NF(X)$ and~$\mu_n<\cdots<\mu_2<\mu_1$,
we call~$\mu_1$ the \emph{leading monomial} of~$\ff$, denoted by~$\bar\ff$, and call  $\alpha_1$ the coefficient of~$\bar\ff$, denoted by~$lc(f)$. We call~$\ff$ a \emph{monic} polynomial if~$lc(f)=1$. A nonempty subset~$S$ of~$\Lbs(X)$ is called a \emph{monic} set, if every element in $S$ is monic. In addition, we denote
$$\supp(f):=\{\mu_1,\dots,\mu_n\}.$$
 For convenience, we define~$\bar{0}=0$, $\deg(0)=0$, $\ell(0)=0$. Then $0<\mu$ for each $\mu\in \NF(X)$.

Recall that a polynomial $f=\sum\alpha_i \mu_i\in\Lbs(\XX)$ is  called \emph{homogeneous},  if  $|\mu_i|=0$ for all $i$, or $|\mu_i|=1$ for all $i$. If this is the case, then we denote~$|f|=0$ or~$|f|=1$.
A nonempty set $S\subseteq\Lbs(\XX)$ is called \emph{homogeneous} if every~$f$ in~$\SS$ is homogeneous.

From now on, we always assume that~$S$ is a monic homogeneous set.

The following claim does not offer an explicit formula for the product of two elements in~$\Lbs(\XX)$, but
it describes the content of the resulted product.

\begin{lemm}\label{multiplication}
(i) Assume that~$\gg$ is a polynomial and~$\ff_1,...,\ff_n$ are homogeneous polynomials in $\Lbs(\XX)$.
Then we  have $$(\gg[\ff_1\cdots\ff_n]_{_L})=\sum \alpha_{i}[\gg\xx_{i_1}\cdots\xx_{i_{n}}]_{_L},$$  where each $\alpha_i\in \kk$ and each~$(\xx_{i_1},\ldots,\xx_{i_{n}})$ is a permutation of~$(\ff_1,\ldots,\ff_n)$.
In particular, if~$\bb,\aa_1,...,\aa_n\in\XX$, then
$(\bb[\aa_1\cdots\aa_n]_{_L})=\sum\alpha_{i}[\bb\xx_{i_1}\cdots\xx_{i_{n}}]_{_L}$, where each ~$\alpha_i\in \kk$ and each~$(\xx_{i_1},\ldots,\xx_{i_{n}})$ is a permutation of~$(\aa_1,\ldots,\aa_n)$.

(ii) Let $\mu,\nu$ lie in $\NF(\XX)$. If $(\mu\nu)$ is not zero, then for each $\tau\in\supp((\mu\nu))$, we have $\deg(\tau)=\deg(\mu)+\deg(\nu)$ and $\ell(\tau)=\ell(\mu)+\ell(\nu)$.
\end{lemm}
\begin{proof}
We prove (i) by using induction on~$n$. If~$n=1$, then there is nothing to prove. If~$n>1$, then
we get
$$(\gg[\ff_1\cdots\ff_n]_{_L})
= ((\gg[\ff_1\cdots\ff_{n-1}]_{_L})\ff_n)
-(-1)^{|\ff_n||[\ff_1\cdots\ff_{n-1}]_{_L}|}
((\gg\ff_n)[\ff_1\cdots\ff_{n-1}]_{_L}).$$
By induction hypothesis, we have
$$(\gg[\ff_1\cdots\ff_{n-1}]_{_L})=\sum \alpha'_{i}[\gg\xx'_{i_1}\cdots\xx'_{i_{n-1}}]_{_L},\ ((\gg\ff_n)[\ff_1\cdots\ff_{n-1}]_{_L})=\sum \alpha'_{i}[\gg\ff_n\xx'_{i_1}\cdots\xx'_{i_{n-1}}]_{_L},$$
 where each $\alpha'_i\in \kk,$  and each~$(\xx'_{i_1},\ldots,\xx'_{i_{n-1}})$ is a permutation of $(\ff_1,\ldots,\ff_{n-1})$.
Therefore, the result holds.

(ii)  follows by (i).
\end{proof}

\begin{lemm}\label{remark}
Let~$\mu, \nu, \tau$ lie in $\NF(\XX)$ and $\nu<\mu$. Then

\item[(i)]  $\overline{(\nu\tau)}<\overline{(\mu\tau)}$, if $(\mu\tau)$ is not zero;
\item[(ii)]  $\ov{(\tau\nu)}<\ov{(\tau\mu)}$, if $\mu$ lies in $\XX$;
\item[(iii)]   for each nonzero polynomial~$\ff$ in $\Lbs(\XX)$,~$\overline{[\ff\aa_1\cdots\aa_n]_{_L}}=[\bar{\ff}\aa_1\cdots\aa_n]_{_L}$, where~$\aa_1,...,\aa_n\in\XX$.

\end{lemm}

 \begin{proof}
\ITEM1  Suppose $\mu=[b_1\cdots b_m]_{_L}, \nu=[c_1\cdots c_p]_{_L},\tau=[d_1\cdots d_t]_{_L}$, where each ${b_i, c_j, d_l\in X}$, and $(\mu\tau)\neq 0$.
If $(\nu\tau)=0$, there is nothing to prove.
Suppose that $(\nu\tau)$ is not zero.
 If $\deg(\nu)<\deg(\mu)$, then by Lemma \ref{multiplication}, we get~$\deg({\overline{(\nu\tau)}})<\deg({\overline{(\mu\tau)}})$.
 If $\deg(\mu)=\deg(\nu)$ and~$p<m$, then by Lemma \ref{multiplication}, $\deg({\overline{(\nu\tau)}})=\deg({\overline{(\mu\tau)}})$, and~$\ell(\overline{(\nu\tau)})<\ell(\overline{(\mu\tau)})$.
 If $\deg(\mu)=\deg(\nu)$ and $m=p$, then there exists $i\in\{1,\dots,m\}$ such that $ b_j=c_j$ for all~$j<i$ and~$ c_i <b_i$. Thus according to Lemma~\ref{multiplication}, we get~$\overline{(\mu\tau)}=[b_1\cdots b_mx_{i_1}\cdots x_{i_t}]_{_L}$ and $\overline{(\nu\tau)}= [c_1\cdots c_mx_{i_1}\cdots x_{i_t}]_{_L}$, where~ $(x_{i_1},...,x_{i_t})$ is a permutation of $(d_1,...,d_t)$. This shows that $\overline{(\nu\tau)}<\overline{(\mu\tau)}$.

\ITEM2  If $(\tau\nu)=0$, then the result holds immediately.
Assume that $(\tau\nu)\neq 0$. For $\mu\in\XX$, if~$\deg(\nu) <\deg(\mu)$, then we have $$\deg(\ov{(\tau\nu)})=\deg(\tau)+\deg(\nu)<\deg(\tau)+\deg(\mu)=\deg(\ov{(\tau\mu)}).$$ Otherwise~$ \nu$ lies in~$\XX$ and~$ \nu<\mu$. Then we have $\ov{(\tau\nu)}=(\tau\nu)<(\tau\mu)=\ov{(\tau\mu)}$.

\ITEM3 This part follows from (i).
 \end{proof}

Note that, in general, for $\mu, \nu, \tau \in \NF(\XX)$ and $\nu<\mu$, we can not get~$\overline{(\tau\nu)}<\overline{(\tau\mu)}$ even if $(\tau\mu)\neq0$. For instance, assume~$\tau=\aa_1,$ $\mu=(\aa_2\aa_1),$ $\nu=(\aa_1\aa_3),$ ${\deg(\aa_1)=\deg(\aa_2)=\deg(\aa_3)}$ and $\aa_1<\aa_2<\aa_3$. Then we get~$\nu<\mu$ and $\overline{(\tau\mu)}=((\aa_1\aa_2)\aa_1)<\overline{(\tau\nu)}=((\aa_1\aa_3)\aa_1).$

\begin{defi}
Let  $S$ be a monic homogeneous subset of $\Lbs(\XX)$.  We define \emph{normal $S$-polynomials} (with respect to $<$) inductively as follows: for each~$s\in S,\ \aa_1\wdots\aa_m\in \XX$, $m\geq 0$,
\begin{enumerate}

\item[(i)] the polynomial~$[s\aa_1\cdots\aa_m]_{_L} $ is a normal~$\SS$-polynomial;

\item[(ii)] the polynomial~$[\aa_1\cdots\aa_t s\aa_{t+1}\cdots\aa_m]_{_L}$ is a normal~$\SS$-polynomial if~$\ell(\bar s)=1$ and $1\leq t\leq m$.
\end{enumerate}
\end{defi}
  We usually use the notation~$h_s$ for a normal~$S$-polynomial.
\begin{rema}\label{hsa}
  For each normal $S$-polynomial $h_s$, $[h_s\bb_1\cdots\bb_n]_{_L}$ is also a normal $S$-polynomial, where~$\bb_1,\ldots,\bb_n\in\XX, n\geq 0$.
\end{rema}
The leading monomial of a normal~$\SS$-polynomial is always obvious, by Lemma \ref{remark}.
\begin{lemm}\label{hshs}
For each normal $S$-polynomial $h_{s}$, $s\in S$, $\aa_1, \dots, \aa_m\in X$, $m\geq 0$,

\ITEM1 if~$h_{s}=[s\aa_1\cdots\aa_m]_{_L}$, then we have~$\overline{h_{s}}=[\bar{s}\aa_1\cdots\aa_m]_{_L}$;

\ITEM2  if~$h_{s}=[\aa_1\cdots\aa_t s\aa_{t+1}\cdots\aa_m]_{_L}$ where~$\ell(\bar s)=1$, then we get~$\overline{h_{s}}=[\aa_1\cdots\aa_t \bar{s}\aa_{t+1}\cdots\aa_m]_{_L}$;

\ITEM3  $\ov{s}\leq \ov{h_{s}}$ holds.
\end{lemm}

The following lemma  shows that the set
$$
\Irr(S):=\{\mu\in \NF(\XX)\mid \mu\neq \ov{\hh_s} \mbox{ for any normal } S\mbox{-polynomial } \hh_s\},
$$
is a linear generating set of the presented Leibniz superalgebra~$\Lbs(\XX|\SS):=\Lbs(\XX)/\operatorname{\Id}(S)$.
\begin{lemm}\label{f=Irr+n-s-polynomials}
Let $\SS$ be a monic homogeneous subset of~$\Lbs(\XX)$. Then we have
$$\ff=\sum\alpha_i\mu_i+
\sum\beta_jh_{j,s_j},\ \ \mbox{for each }\ff\in \Lbs(\XX),
$$
where each $\alpha_i,\beta_j\in k$, $ \mu_i\in \Irr(\SS)$ with $\mu_i\leq \bar f$ and $\hh_{j,s_j}$ is
a normal $\SS$-polynomial with $\ov{h_{j,s_j}}\leq\bar \ff$.  In particular, it follows that   $\Irr(\SS)$ is a set of linear generators of the Leibniz
superalgebra $ \Lbs(\XX|\SS)  $.
\end{lemm}
\begin{proof} The result follows by induction on $\bar \ff $.
\end{proof}

We are now able to introduce a series of conditions that under which the set~$\Irr(\SS)$ is a linear basis of~$\Lbs(\XX|\SS)$. We shall need  the notation of  composition.  Before going there, we introduce notation below:
for any $f,g\in\Lbs(\XX),\ \mu\in\NF(X)$, $n\in \mathbb{Z}_{>0}$, we define
$$f\equiv g\mod(S,\mu)\ \ \ (\mbox{resp.}\ f \equiv g  \mod(S,n)),$$ if $f- g$ can be written as a linear combination of normal~$S$-polynomials such that their leading monomials $<\mu$ (resp. the degrees of their leading monomials $\leq n$).
Furthermore, a polynomial~$h$ in~$\Lbs(X)$ is said to be \emph{trivial modulo $(S,\mu)$ (resp. $(S,n)$)}, if
$$\hh \equiv0  \mod(S,\mu)\ \ \ (\mbox{resp.}\ \hh \equiv0  \mod(S,n)). $$

\begin{defi}
Let $S$ be a monic homogeneous subset of~$ \Lbs(\XX)$. For all $f,g\in S$, $f\neq g$, we define \emph{compositions} as follows:
\begin{enumerate}
\item[(i)] If $\bar{\ff}=\overline{h_{\gg}}$ for some  normal $S$-polynomial~$h_\gg$, then we  call
 $$(\ff, \gg)_{\bar{\ff}}=\ff-h_\gg$$
 an \emph{inclusion composition} of $\SS$.

\item[(ii)] For each $\mu\in \NF(\XX)$, if $\ell(\bar{f})>1$ and $(\mu\ff)\neq 0$, then we call $(\mu\ff)$ a \emph{left multiplication composition} of $\SS$.
\end{enumerate}

The set $S$ is said to be \emph{closed} under the left multiplication compositions if every left multiplication composition $(\mu\ff)$ of $S$ is trivial modulo $(S,\deg(\mu)+\deg(\bar{\ff}))$.

 The set $S$ is called a \emph{\gsb}  in $\Lbs(\XX)$ if every left multiplication composition~$(\mu\ff)$ of~$S$ is trivial modulo~$(S,\deg(\mu)+\deg(\bar{\ff}))$ and every inclusion composition~$(\ff, \gg)_{\bar{\ff}}$ of~$S$ is trivial modulo $(S,\bar\ff)$.
\end{defi}

We shall see that, if a monic homogeneous set~$S$ is closed under left multiplication compositions, then  the elements of the graded ideal~$\Id(S)$ of~$\Lbs(\XX)$ can be written as linear combinations of normal $S$-polynomials.

\begin{lemm}\label{sum-of-normal-s-polynomial}
Let $S$  be a  monic homogeneous subset of $\Lbs(\XX)$ that is closed under left multiplication compositions.
Then
$[\bb_1\cdots\bb_m [\aa_1\cdots\aa_t s\aa_{t+1}\cdots\aa_n]_{_L}\cc_1\cdots\cc_r]_{_L}=\sum \alpha_i h_{i,s_i}$, where each $\alpha_i\in\kk$; $s\in S$; $\aa_1, \dots, \aa_n, \bb_1, \dots, \bb_m, \cc_1, \dots, \cc_r\in \XX$; $m,n,r\in \mathbb{Z}_{\geq0}$;~$ 0\leq t\leq n$ and each $ h_{i,s_i}$ is a normal $S$-polynomial with~$\deg(\overline{h_{i,s_i}})\leq \deg(\overline{s})+ \sum_{p=1}^n\deg(\aa_p)+\sum_{i=1}^m\deg(\bb_i)+\sum_{j=1}^r\deg(\cc_j)$.
\end{lemm}

\begin{proof}
 It is enough to show that $[\bb_1\cdots\bb_m [\aa_1\cdots\aa_t s \aa_{t+1}\cdots\aa_n]_{_L}]_{_L}$ can be written as a linear combination of normal $S$-polynomials with the desired property.

By Lemma \ref{multiplication} we have
 $$[\bb_1\cdots\bb_m [\aa_1\cdots\aa_t s \aa_{t+1}\cdots\aa_n]_{_L}]_{_L}=\sum \beta_j [\bb_1\cdots\bb_m \xx_{j_1}\cdots\xx_{j_{n+1}}]_{_L},$$ where each~$\beta_j\in\kk$ and each~$ (\xx_{j_1},\ldots,\xx_{j_{n+1}})$ is a permutation of $(\aa_1,\ldots,\aa_n, s)$.

   If $\ell(\bar{s})=1$, then the polynomial~$[\bb_1\cdots\bb_m \xx_{j_1}\cdots\xx_{j_{n+1}}]_{_L}$ is already a normal~$S$-polynomial with the desired property.

If $\ell(\bar{s})>1$, then the claim follows immediately because $S$ is closed under the left multiplication compositions.
\end{proof}
Before going further, we introduce some more notation. For a fixed set~$\XX$, we use~$\XX^*$ for the free monoid and~$\XX^+$ for the free semigroup generated by~$\XX$, that is~${\XX^* = \XX^+ \cup \{\varepsilon\}}$, where $\varepsilon$ is the empty word. For~$\uu_1=a_{1}\cdots a_{m},\ \uu_2=b_{1}\cdots b_{n}$ in~$\XX^*,\ f\in \Lbs(X)$,  where each~$a_{i}, b_j\in\XX$, we recall that
$$[u_1fu_2]_{_L}:
=[a_{1}\cdots a_{m} fb_1\cdots b_{n}]_{_L},\ \ \ \ \ \NF(X)=\{[u]_{_L}\mid u\in X^+\}.$$
For convenience, we denote that $[u_1fu_2]_{_L}:
=[ fu_2]_{_L}$ if ~$u_1=\varepsilon$; $[u_1fu_2]_{_L}:
=[ u_1f]_{_L}$ if ~$u_2=\varepsilon$.

\begin{lemm}\label{key-lemma}
Let $S$ be a Gr\"{o}bner--Shirshov basis in $\Lbs(\XX)$ and $h_{1,s_1}, h_{2,s_2}$ be two normal $S$-polynomials. If~$\overline{h_{1,s_1}}=\overline{h_{2,s_2}}$, then we get
$h_{1,s_1}-h_{2,s_2}\equiv0 \mod (S,\overline{h_{1,s_1}}).$
\end{lemm}

\begin{proof}
Suppose $s_1=\ov{s_1}+\sum\beta_j [\uu_j]_{_L}$, $s_2=\ov{s_2}+\sum\gamma_{_l} [\vv_{_l}]_{_L}$, where each $[\uu_j]_{_L}, [\vv_{_l}]_{_L}\in \NF(X)$.

Note that $\overline{h_{1,s_1}}=[\overline{s_1}\uu ]_{_L}$, where $\uu \in\XX^*$, or $\overline{h_{1,s_1}}=[\vv\overline{s_1}\uu ]_{_L}$, where $\ell(\overline{s_1})=1,\ \uu,\vv \in\XX^*$.
Since $\overline{h_{1,s_1}}=\overline{h_{2,s_2}}$,
the following four situations should be considered.

\ITEM1 $h_{1,s_1}=[s_1\uu \ov{s_2}\vv ]_{_L}$, $h_{2,s_2}=[\ov{s_1}\uu s_2\vv ]_{_L}$, where $\ell(\overline{s_2})=1$, $\uu , \vv \in \XX^*$. Then  we get
\begin{align*}[s_1\uu \ov{s_2}\vv ]_{_L}-[\ov{s_1}\uu s_2\vv ]_{_L}
&=[s_1\uu (\ov{s_2}-s_2)\vv ]_{_L}+[(s_1-\ov{s_1})\uu s_2\vv ]_{_L}\\
&=-\sum\gamma_l [s_1\uu [\vv _{_l}]_{_L} \vv ]_{_L}+\sum\beta_j[[\uu_j]_{_L}\uu s_2\vv ]_{_L},
\end{align*}
where each $[[\uu_j]_{_L}\uu s_2\vv ]_{_L}=[\uu_j\uu s_2\vv ]_{_L}$ is a normal $S$-polynomial and it is straightforward to see that~$\overline{[\uu_j\uu s_2\vv ]_{_L}}<\overline{[\ov{s_1}\uu s_2\vv ]_{_L}}=[\ov{s_1}\uu \ov{s_2}\vv ]_{_L}=\overline{h_{1,s_1}}$.
Moreover, by Lemma~\ref{multiplication}, each ${[s_1\uu [\vv_{_l}]_{_L} \vv ]_{_L}=[[[s_1\uu]_{_L} [\vv_{_l}]_{_L}]_{_L} \vv ]_{_L}}$ can be written as $\sum\eta_p[s_1\uu \vv _{_{lp}}\vv ]_{_L}$, where each $\vv _{_{lp}}=x_{lp_{_1}}\cdots x_{lp_{_{n_l}}}\in\XX^+$ and $(x_{lp_{_1}},\dots, x_{lp_{_{n_l}}})$ is a permutation of $(x_{l_1},\dots, x_{l_{n_l}})$, if $[\vv_{_l}]_{_L}=[x_{l_1}\cdots x_{l_{n_l}}]_{_L}$. It is clear that each $[s_1\uu \vv _{_{lp}}\vv ]_{_L}$ is a normal $S$-polynomial. By Lemma \ref{remark}, we have $\overline{[s_1\uu \vv _{_{lp}}\vv ]_{_L}}=[\ov{s_1}\uu \vv _{_{lp}}\vv ]_{_L}<[\ov{s_1}\uu \ov{s_2}\vv ]_{_L}=\overline{h_{1,s_1}}$. This shows that $h_{1,s_1}-h_{2,s_2}\equiv0 \mod (S,\overline{h_{1,s_1}}).$

\ITEM2 $h_{1,s_1}=[\uu s_1\vv \ov{s_2}w]_{_L}$, $h_{2,s_2}=[\uu \ov{s_1}\vv s_2w]_{_L}$,
where $\ell(\ov{s_1})=\ell(\ov{s_2})=1$, $\uu , \vv , w\in \XX^*$.
Then the reasoning is similar to that of \ITEM1, and we get \begin{multline*}
h_{1,s_1}-h_{2,s_2}=[\uu s_{1}\vv \ov{s_2}w]_{_L}-[\uu \ov{s_1}\vv s_{2}w]_{_L}
=[\uu s_{1}\vv (\ov{s_2}-s_2)w]_{_L}+[\uu (s_1-\ov{s_1})\vv s_{2}w]_{_L}\\
=-\sum\gamma_l [\uu s_{1}\vv[\vv_{_l}]_{_L}w]_{_L}+\sum\beta_j[\uu [\uu _j]_{_L}\vv s_{2}w]_{_L}
\equiv 0\mod(S,\ov{h_{1,s_1}}).
\end{multline*}

\ITEM3 $h_{1,s_1}=[s_1\vv ]_{_L}$ and~$h_{2,s_2}=[s_2\uu \vv ]_{_L}$, or
 $h_{1,s_1}=[s_1\vv ]_{_L}$,~$h_{2,s_2}=[ws_2\uu \vv ]_{_L}$ and~$\ell(\overline{s_2})=1$, where $w, \uu , \vv$ belong to $\XX^*$.
 We only consider the case~$h_{2,s_2}=[s_2\uu \vv ]_{_L}$, since the other is in the similar situation.
 Then using the fact that $S$ is a Gr\"{o}bner--Shirshov basis, we may suppose that~$(s_1,s_2)_{\ov{s_1}}=s_1-[s_2\uu]_{_L}=\sum\alpha_i h_{i,s_i'}'$, where each $\alpha_i\in k,\ \overline{h_{i,s_i'}'}<\overline{s_1}$ and~$h_{i,s_i'}'$ is a normal $S$-polynomial.
    Then we get $$h_{1,s_1}-h_{2,s_2}=[(s_1,s_2)_{\ov{s_1}}\vv ]_{_L}=\sum\alpha_i [h_{i,s_i'}'\vv ]_{_L},$$
where each $[h_{i,s_i'}'\vv ]_{_L}$ is clearly a normal $S$-polynomial. Applying Lemma \ref{remark}, we get $\overline{[h_{i,s_i'}'\vv ]_{_L}}=[\overline{h_{i,s_i'}'}\vv ]_{_L}<[\overline{s_1}\vv ]_{_L}=\overline{h_{1,s_1}}$. Therefore, $h_{1,s_1}-h_{2,s_2}\equiv0 \mod (S,\overline{h_{1,s_1}}).$

\ITEM4 $h_{1,s_1}=[\uu s_1\vv ]_{_L}$, $h_{2,s_2}=[\uu s_2\vv ]_{_L}$, where $\ell(\overline{s_1})=\ell(\overline{s_2})=1$, $\uu , \vv \in \XX^*$. Then we assume
that~$(s_1,s_2)_{\ov{s_1}}=s_1-s_2=\sum\alpha_i h'_{i,s_i'}$, where each $\alpha_i\in k,\ \overline{h_{i,s_i'}'}<\overline{s_1}$ and~$h_{i,s_i'}'$ is a normal $S$-polynomial.
So we get $[\uu s_1\vv ]_{_L}-[\uu s_2\vv ]_{_L}=[\uu (s_1-s_2)\vv ]_{_L}=\sum\alpha_i [\uu h_{i,s_i'}'\vv ]_{_L}$.
If $\ell(\overline{h_{i,s_i'}'})=1$, then ~$\overline{h_{i,s_i'}'}=\ov{s_i'},\ \ell(\ov{s_i'})=1$ and $[\uu h_{i,s_i'}'\vv ]_{_L}$ is a normal $S$-polynomial such that~$\overline{[\uu h_{i,s_i'}'\vv ]_{_L}}=[\uu \overline{h_{i,s_i'}'}\vv ]_{_L}<[\uu \overline{s_1}\vv ]_{_L}=
\overline {h_{1,s_1}}$.
If $\ell(\overline{h_{i,s_i'}'})\neq1$, then noting~$\overline{h_{i,s_i'}'}<\ov{s_1}$ and $\ell(\overline{s_1})=1$, we get ~$\deg(\overline{h_{i,s_i'}'})<\deg(\ov{s_1})$. By Lemma \ref{sum-of-normal-s-polynomial}, we get~$[\uu h_{i,s_i'}'\vv ]_{_L}=\sum\beta_j h''_{j,s_j''}$, where each~$h''_{j,s_j''}$ is a normal $S$-polynomial and
$$\deg(\overline{h''_{j,s_j''}})\leq \deg([\uu ]_{_L})+\deg(\overline{h_{i,s_i'}'})+\deg([\vv ]_{_L})<\deg([\uu ]_{_L})+\deg(\overline{s_1})+\deg([\vv ]_{_L})=\deg(\overline {h_{1,s_1}}).$$ Therefore $\overline{h''_{j,s_j''}}<\overline {h_{1,s_1}}$ and so $h_{1,s_1}-h_{2,s_2}\equiv0 \mod (S,\overline{h_{1,s_1}})$.
\end{proof}

Now we are ready to prove the main result in this paper.

\begin{theorem}\label{cd-lemma for Lb}  (Composition-Diamond lemma for Leibniz superalgebras)
Let $\Lbs(\XX)$ be the free Leibniz superalgebra generated by a well-ordered set $X=X_0\cup X_1$ with the linear basis $\NF(\XX)$, $<$  a deg-length-lex order  on $\NF(\XX)$, $S\subseteq \Lbs(\XX)$  a nonempty monic homogeneous set and $\operatorname{\Id(S)}$  the graded ideal of $\Lbs(\XX)$ generated by $S$. Then the following statements are equivalent.

\ITEM1 The set~$S$ is a Gr\"{o}bner--Shirshov basis in $\Lbs(\XX)$.

\ITEM2 If $0\neq f\in \operatorname{\Id(S)}$, then $\overline {f}=\overline{h_s}$ for some normal~$S$-polynomial~$h_s$.

\ITEM3 The set $\Irr(S)=\{\tau\in \NF(\XX)\mid \tau\neq \overline{h_s} \mbox{ for any normal } S\mbox{-polynomial } h_s\}$ forms  a linear basis of the Leibniz superalgebra $\Lbs(\XX | S):=\Lbs(\XX)/\operatorname{\Id(S)}$.

\end{theorem}

\begin{proof}
 \ITEM1 $\Rightarrow$ \ITEM2 By Lemma~\ref{sum-of-normal-s-polynomial}, for each  nonzero element~$f$ in~$\operatorname{\Id}(S)$, we may assume that~$ \ff=\sum_{i=1}^n\alpha_i\hh_{i,s_i}$,
  where each
$0\neq\alpha_i\in \kk$ and $\hh_{i,s_i}$ is a normal $\SS$-polynomial.
Define~$
\mu_i= \overline{\hh_{i,s_i}}$ and assume that~$\pdots\leq\MU_{l+1}<\MU_{_l}=\pdots=\MU_2=\MU_1.
$
 Use induction on~$\MU_1$.
 If~$\MU_1=\ov{\ff}$, then there is nothing to prove.
 If~$\ov{\ff}<\MU_1$, then $\sum_{i=1}^l\alpha_{i}=0$ and by Lemma \ref{key-lemma}
\begin{align*} \ff&=\sum_{i=1}^l\alpha_i\hh_{i,s_i}+\sum_{ i=l+1}^n\alpha_i\hh_{i,s_i}\\
  &=(\sum_{ i=1}^l\alpha_i)\hh_{1,s_1}
  -\sum_{ i=2}^l\alpha_i(\hh_{1,s_1}
  -   \hh_{i,s_i})+\sum_{ i=l+1}^n\alpha_i\hh_{i,s_i}
  =\sum_{j }\beta_{j}\hh_{j,s_j'}'+\sum_{ i=l+1}^n\alpha_i\hh_{i,s_i},
\end{align*}
 where each $\ov{\hh_{j,s_j'}'}<  \MU_1$.
 Claim \ITEM2 follows by induction hypothesis.

\ITEM2 $\Rightarrow$ \ITEM3  By Lemma \ref{f=Irr+n-s-polynomials},
$\Irr(\SS)$ generates
$\Lbs(\XX|\SS)$ as a vector space. Suppose that~$\sum_{i}\alpha_i\MU_i=0$ in $\Lbs(\XX|\SS)$,
where each $ \alpha_i\in \kk$, $\alpha_i\neq 0$, $\MU_i\in  \Irr(\SS) $ and $ \pdots<\MU_2<\MU_1$.
Then we have~$0\neq \sum_{i}\alpha_i\MU_i\in \operatorname{\Id}(\SS) $. Thus~$\overline{\sum_{i}\alpha_i\MU_i}=\MU_1\in \Irr(\SS) $,
 which contradicts with~\ITEM2.

\ITEM3 $ \Rightarrow$ \ITEM1  Note that every composition in $S$
   is an element of $\operatorname{\Id}(\SS)$.
   It is obvious that every inclusion composition $(f, g)_{\ov{f}}\equiv0 \mod(S,\bar{\ff})$.
   By Lemma~\ref{f=Irr+n-s-polynomials} and \ITEM3,
     for each $\mu\in\NF(X), f\in S$, we get~$(\mu f)=\sum\beta_jh_{j,s_j}$, where each $\beta_j\in\kk,\ h_{j,s_j}$ is a normal $S$-polynomial with $\ov{h_{j,s_j}}\leq\ov{(\mu \ff)}$. Note that~$\ov{h_{j,s_j}}\leq\ov{(\mu \ff)}$ and $(\mu\ff)\neq 0$ implies~$\deg(\ov{h_{j,s_j}})\leq\deg(\mu)+\deg(\bar{\ff})$. Then~$(\mu f)\equiv0  \ \mod(S,\deg(\mu)+\deg(\bar{\ff}))$. Thus \ITEM1 holds.
 \end{proof}

\begin{rema}
 Consider ordinary Leibniz algebra, which is a special case of Leibniz superalgebra, that is, Leibniz superalgebra generated by elements of parity~$0$.
We use the notation~$\mathsf{Lb}(X)$ for the free Leibniz algebra generated by~$X=X_0$.
If this is the case, we have $\Lbs(X)=\mathsf{Lb}(X)$. Thus every polynomial in $\mathsf{Lb}(X)$ is homogeneous, and every ideal in $\mathsf{Lb}(X)$ is a graded ideal.
Therefore, Theorem~\ref{cd-lemma for Lb} is also the Composition-Diamond lemma for (ordinary) Leibniz algebras.
\end{rema}
We now turn to the question on how to recognize whether two graded ideals of~$\Lbs(X)$ is the same or not.
We begin with  the notation of a minimal (resp. reduced) Gr\"{o}bner-Shirshov basis.

\begin{defi}
A Gr\"{o}bner-Shirshov basis $S$ in $\Lbs(X)$ is \emph{minimal (resp. reduced)} if for each $s\in S$, we have~$\overline{s}\in \Irr(S \backslash \{s\})$\ (resp. $\supp(s)\subseteq \Irr(S \backslash \{s\})$).

Suppose $I$ is a graded ideal of $\Lbs(X)$ and $I=\Id(S)$, where $S$ is a homogeneous set. If $S$ is a minimal (resp. reduced) Gr\"{o}bner-Shirshov basis in $\Lbs(X)$, then we also call $S$  a minimal (resp. reduced) Gr\"{o}bner-Shirshov basis for the graded ideal $I$ or for the quotient Leibniz superalgebra $\Lbs(X)/I=\Lbs(X|S)$.
\end{defi}

\begin{lemm}\label{lIrr}
Let $R$ and~$S$ be homogeneous monic subsets of~$\Lbs(\XX)$ such that~$\Irr(S)=\Irr(R)$. Then the following statements hold.
\begin{enumerate}
\item[(i)]
If~$\Id(S)=\Id(R)$, then $S$ is a \gsb\ in~$\Lbs(\XX)$ if and only if~$R$ is a \gsb\ in~$\Lbs(\XX)$.

\item[(ii)]If~$R\subseteq S$ and~$S$ is a \gsb\ in $\Lbs(\XX)$, then~$R$ is also a \gsb\ for~$\Id(S)$.
\end{enumerate}

\end{lemm}
\begin{proof}
\ITEM1 This part follows by Theorem \ref{cd-lemma for Lb}.

\ITEM2
For each~$f\in \Id(S)$, since $\Irr(R)=\Irr(S)$ and $S$ is a Gr\"{o}bner-Shirshov basis for~$\Id(S)$, we have, by Theorem \ref{cd-lemma for Lb}, $ \overline{f}=\overline{h_s}=\overline{g_r}$ for some normal $S$-polynomial~$h_s$ and for some normal~$R$-polynomial~$g_r$.  So we get~$f_1=f-lc(f)g_r\in \Id(S)$ and $\overline{f_1}<\overline{f}$. By induction on~$\overline{f}$, we  deduce  that~$f$ is a linear combination of normal $R$-polynomials, i.e., $f\in \Id(R)$. This shows that $\Id(S)= \Id(R)$. Now the result follows by \ITEM1.
\end{proof}

For  associative algebras and polynomial algebras, it is known that every ideal has a unique reduced Gr\"{o}bner-Shirshov basis.
For Leibniz (super)algebras, we have  a similar result.

For each subset~$S$ of~$\Lbs(X)$, we define
$$
\overline{S}:=\{\overline{s}\mid s\in S\}.
$$

\begin{theorem}\label{reduced}
There is a unique reduced Gr\"{o}bner-Shirshov basis for every  (graded) ideal of the free Leibniz (super)algebra~$\Lbs(\XX)$ generated by $X=\XX_0\cup\XX_1$.
\end{theorem}

\begin{proof} Let~$I=I_0\oplus I_1$ be a graded ideal of~$\Lbs(\XX)$. We first show the existence. It is straightforward that~${\{lc(f)^{-1}f\mid 0\neq f\in I_0\cup I_1\}}$ is a Gr\"{o}bner-Shirshov basis for~$I$.
Now assume that~$\SS$ is a Gr\"{o}bner-Shirshov basis for~$I$. For each~$\mu\in \overline{S}$, we choose a polynomial $f_{\mu}$ in $S$ such that~${\overline{f_{\mu}}=\mu}$. Define
$
S_0=\{f_{\mu}\mid \mu\in \ov{S}\}.
$
Then it is clear that $ S\supseteq S_0$ and~$\Irr(S_0)=\Irr(S)$.
By Lemma \ref{lIrr}, $S_0$ is a Gr\"{o}bner--Shirshov basis for $I$.
For each~$g\in S_0$, we set
$
\bigtriangleup_g =\{f\in S_0 \mid f \neq g, \ov{f}\notin \Irr(\{ g\})\}
$. Let
$S_1 = S_0\setminus\bigcup_{g\in S_0}\bigtriangleup_g.
$
By Lemma \ref{lIrr}, $S_1$ is a minimal Gr\"{o}bner--Shirshov basis for $I$.
For each~$s\in S_1$, by Lemma \ref{f=Irr+n-s-polynomials}
we have~$s=s'+s''$, where $\supp(s')\in \Irr(\SS\setminus\{s\})$ and $s''\in\Id(\SS\setminus\{s\})$. Let
$
S_2=\{s'\mid s\in S_1\}.
$
It is easy to prove that $S_2$ is a reduced \gsb\ for $I$.

We now turn to the uniqueness.
Suppose that $R$ and~$S$ are two arbitrary reduced Gr\"{o}bner--Shirshov bases for $I$. Let~$s_0$ and~$  r_0 $ be the elements of~$S$ and~$R$ respectively such that~$\overline{s_0}=\min (\overline{S})$ and
$\overline{r_0}=\min (\overline{R})$. Since~$R$ is a \gsb, we get~$\overline{s_0}=\overline{h_{r'}}\geq \overline{r'}\geq \overline{r_0}$ for some normal~$R$-polynomial~$h_{r'}$.
Similarly, we get~$\overline{r_0}\geq \overline{s_0}$ and thus~$\overline{r_0}=\overline{s_0}$. We claim that $r_0=s_0$.
Otherwise, we get~$0\neq r_0-s_0\in I$. Similarly to the above reasoning,  we get~$\overline{r_0}> \overline{r_0-s_0}=\overline{h_{r''}}\geq \overline{r''} \geq\overline{r_0}$ for some $r''\in R$, a contradiction.
Therefore, we get~$s_0=r_0$. Supposing that for each~$\nu<\mu$, we have
$$
S^{\nu}:=\{s\in\SS\mid \ov{s}\leq \nu\}=\{r\in R\mid \ov{r}\leq \nu\}=:R^{\nu}.
$$
We shall show that~$S^{\mu}=R^{\mu}$. By symmetry, it is enough to show that~$S^{\mu}\subseteq R^{\mu}$.
Suppose~$s$ belongs to $S^{\mu}$. If ~$\ov\ss<\mu$ holds, then we have $s\in S^{\ov{s}}=R^{\ov{s}}\subseteq R^{\mu}$.
If~$\ov\ss=\mu$, then
$\overline{s}=\overline{h_{r}}\geq \overline{r}$ for some normal~$R$-polynomial~$h_{r}$.
We  claim that~$\overline{s}=\overline{r}$. If we get~$\mu=\overline{s}>\overline{r}$,
then $r$ lies in $R^{\ov{r}}=S^{\ov{r}}$ and thus~$r\in S\sm\{\ss\}$, which contradicts  the fact that $S$ is a reduced Gr\"{o}bner--Shirshov basis.
We now claim that  $s=r$. If $s\neq r$, then we have $0\neq s-r\in I$. Moreover, since~$S$ and~$R$ are reduced Gr\"obner--Shirshov bases, we get~$\overline{s-r}\in \Irr(S)\cup \Irr(R)$ which is in contradiction with Theorem~\ref{cd-lemma for Lb} \ITEM2.
\end{proof}

\begin{rema}
For a graded ideal $I$ of $\Lbs(X)$, Theorem \ref{reduced} gives an algorithm to find the reduced Gr\"obner--Shirshov basis for $I$.
\end{rema}
\begin{coro}
Let $I$ and $J$ be two graded ideals of $\Lbs(X)$,  $S_I$ and $S_J$ be the reduced Gr\"obner--Shirshov bases for $I$ and $J$ respectively. Then $I=J$ if and only if $S_I=S_J$.
\end{coro}

According to Theorem \ref{reduced}, for any Leibniz (super)algebra~$\mA$, $\mA$ has a representation $\mA\cong\Lbs(X|R)$, where $X$ is generators and $R$ is a unique reduced \gsb\ in $\Lbs(X)$.
  The following proposition shows that $\mA$  has another representation $\mA\cong\Lbs(X'|R')$, where $X'$ is generators and $R'$ is a unique reduced \gsb\ in $\Lbs(X')$ with the length of the leading monomial of each polynomial  in $R'$ is greater than 1. If this is the case, then there are only left  multiplication compositions in $R'$.   This result will be useful when we characterize extensions of Leibniz (super)algebras.

\begin{prop}\label{greater1}
Let $R$ be a reduced \gsb\ in $\Lbs(X)$, $R=R'\cup S$, where $S=\{g\in R\mid \ell(\ov{g})=1\}$ and $R'\cap S=\emptyset$. Let $X'=X\setminus\ov{S}$. Then $\Lbs(X|R)$ is isomorphic to $\Lbs(X'|R')$, $R'$ is a reduced \gsb\ in $\Lbs(X')$, and the length of the leading monomial of each polynomial  in $R'$ is greater than 1.
\end{prop}

\begin{proof}
 Note that $R$ is a homogeneous set.
Assume that every element $g$ in $S$ has the form: $g=\ov{g}+r_g$.
Then we have two epimorphisms:
 \begin{center}
 $\varphi_1: \Lbs(X) \rightarrow \Lbs(X'|R'),\  x\mapsto x+\Id(R')$ if $x\in X'$; $\ov{g}\mapsto -r_g+\Id(R')$ if $\ov{g}\in\ov{S}$,
 \end{center}
 \begin{center}
 $\varphi_2: \Lbs(X) \rightarrow \Lbs(X|R)$, $ x\mapsto x+\Id(R)$ for all $x\in X$.
 \end{center}

Now we show that $\ker\varphi_1 = \ker\varphi_2$. According to the fact that $R$ is a reduced \gsb\ in $\Lbs(X)$, we have the following two statements.
(i) If $f\in S$, then we have  $\varphi_1(f)=\varphi_1(\overline{f})+\varphi_1(r_f)=-r_f+r_f+\Id(R')=0+\Id(R')$.
(ii) If $f\in R'$, then $\varphi_1(f)=f+\Id(R')=0$ in $\Lbs(X'|R')$. This shows that $\Id(R)=\ker\varphi_2\subseteq \ker \varphi_1$.

For every $f$ in $\Lbs(\XX)$, by Lemma \ref{f=Irr+n-s-polynomials}, we have $\ff=\sum\alpha_i\mu_i+\sum\beta_jh_{j,s_j}$ in $\Lbs(\XX)$,
where each $\alpha_i,\beta_j\in k$, $ \mu_i\in \Irr(R)$ and $\hh_{j,s_j}$ is
a normal $R$-polynomial in $\Lbs(\XX)$. Then $\varphi_1(f)=\sum\alpha_i\mu_i+\Id(R')$ in $\Lbs(\XX'|R')$ and $\varphi_2(f)=\sum\alpha_i\mu_i+\Id(R)$ in $\Lbs(\XX|R)$.
Assume $\varphi_2(f)\neq 0$. Because the subset $\Id(R')$ of $\Lbs(X')$ is contained in the subset $\Id(R)$ of $\Lbs(X)$, we have $\varphi_1(f)\neq 0$. Thus $ \ker\varphi_1 \subseteq \ker\varphi_2$.

Therefore $\Lbs(X|R)\cong\Lbs(X'|R')$. Furthermore,
\begin{align*}
\Irr(R)&=\{\mu\in \NF(\XX)\mid \mu\neq \ov{\hh_r} \mbox{ for any normal } R\mbox{-polynomial } \hh_r \mbox{ in } \Lbs(X)\}\\
&=\{\mu\in \NF(\XX')\mid \mu\neq \ov{\hh_r} \mbox{ for any normal } R\mbox{-polynomial } \hh_r \mbox{ in } \Lbs(X)\}\\
&=\{\mu\in \NF(\XX')\mid \mu\neq \ov{\hh_r} \mbox{ for any normal } R'\mbox{-polynomial } \hh_r \mbox{ in } \Lbs(X')\}\\
&=\Irr(R')
\end{align*}
By Theorem \ref{cd-lemma for Lb} $\Irr(R)+\Id(R):=\{\mu+\Id(R)\mid\mu\in\Irr(R)\}$ is a linear basis of the algebra $\Lbs(\XX | R)$. Thus $\Irr(R')+\Id(R')=\sigma(\Irr(R)+\Id(R))$ is a linear basis of the algebra $\Lbs(\XX' | R')$. According to Theorem \ref{cd-lemma for Lb}, we know that $R'$ is a \gsb\ in $\Lbs(X')$. Since $R$ is a reduced \gsb\ in $\Lbs(X)$, $R'\ (\subseteq R)$ is a reduced \gsb\ in $\Lbs(X')$.

\end{proof}

\section{Applications}\label{linear-basis-metabelian}
Our aim in this section is to construct linear bases of free metabelian Leibniz (super)algebras and those of free metabelian Lie algebras  by using Gr\"{o}bner--Shirshov bases theories of Leibniz (super)algebras.  In this section, we assume that the degree of every element in $\XX$ is 1, that is $\deg(\aa)=1$, for all $\aa\in\XX$.

\subsection{Linear bases of free metabelian Leibniz superalgebras}
\begin{defi}\cite{belian}
A Leibniz (super)algebra $(\mathcal{A},(--))$ is called \emph{abelian} if~$(xy) = 0$ for all~$  x, y$  in~$\mA$, i.e., $\mA^2=0$.
A Leibniz (super)algebra is called \emph{metabelian} if it is an extension of an abelian
Leibniz (super)algebra by another abelian Leibniz (super)algebra. Denote by~$\MLbs(X)$ the free metabelian Leibniz superalgebra
generated by~$X=X_0\cup X_1$.
\end{defi}

 Clearly, the free metabelian Leibniz superalgebra generated by~$X=X_0\cup X_1$ can be presented by generators and relations obviously.

 \begin{lemm}\label{T-metabelian}
 Let $
T=\{([\cc_1\cdots\cc_p]_{_L}[\dd_1\cdots\dd_t]_{_L})\mid \cc_1,\dots,\cc_p,\dd_1,\dots,\dd_t\in\XX,\ p,t\geq2\}$.  Then we have
$\MLbs(X)=\Lbs(\XX|T)$.
\end{lemm}

To offer a \gsb\ for~$\MLbs(X)$, the set~$T$ is not a good choice because it is not easy to detect the leading monomials of elements in~$T$. So we need to get another presentation for~$\MLbs(X)$.
\begin{lemm}\label{st}
Let~$S=\{[\cc_1\cdots\cc_p\aa_1\aa_2]_{_L}-(-1)^{|\aa_1||\aa_2|}[\cc_1\cdots\cc_p\aa_2\aa_1]_{_L}\mid \aa_1,\aa_2,\cc_1,\dots,\cc_p\in\XX,\ p\geq2\}$.
Then we get~$\Id(T)=\Id(S)$ and thus~$\MLbs(X)=\Lbs(\XX|S)$.
\end{lemm}

\begin{proof}
First we show $\Id(S)\subseteq \Id(T)$. For $[\cc_1\cdots\cc_p]_{_L}\in \NF(\XX),\ \aa_1,\aa_2\in\XX,\ p\geq2$, we have
$$(([\cc_1\cdots\cc_p]_{_L}\aa_1)\aa_2)
-(-1)^{|\aa_1||\aa_2|}(([\cc_1\cdots\cc_p]_{_L}\aa_2)\aa_1)=([\cc_1\cdots\cc_p]_{_L}(\aa_1\aa_2)) \in \Id(T).$$

In order to show that~$([\cc_1\cdots\cc_p]_{_L}[\dd_1\cdots\dd_t]_{_L})\in \Id(S)$, where $\cc_1,\dots,\cc_p,\dd_1,\dots,\dd_t\in\XX$ and $p, t\geq2$, we use induction on $t$. If $t=2$, then we have
$$([\cc_1\cdots\cc_p]_{_L}[\dd_1\dd_2]_{_L})=[\cc_1\cdots\cc_p\dd_1\dd_2]_{_L}
-(-1)^{|\dd_1||\dd_2|}[\cc_1\cdots\cc_p\dd_2\dd_1]_{_L}\in \Id(S).$$
For~$t>2$, we obtain \begin{multline*}([\cc_1\cdots\cc_p]_{_L}[\dd_1\cdots\dd_t]_{_L})\\
=(([\cc_1\cdots\cc_p]_{_L}[\dd_1\cdots\dd_{t-1}]_{_L})\dd_t)
-(-1)^{|[\dd_1\cdots\dd_{t-1}]_{_L}||\dd_t|}(([\cc_1\cdots\cc_p]_{_L}\dd_t)[\dd_1\cdots\dd_{t-1}]_{_L}).
\end{multline*}
By induction hypothesis, both~$([\cc_1\cdots\cc_p\dd_t]_{_L}[\dd_1\cdots\dd_{t-1}]_{_L})$ and~$([\cc_1\cdots\cc_p]_{_L}[\dd_1\cdots\dd_{t-1}]_{_L})$ lie in~$\Id(S)$.  Thus $\Id(T)\subseteq \Id(S)$.
\end{proof}

Denote by~$\chart(k)$ the characteristic of the field~$\kk$. We shall see that fields of different characteristics lead  to   different linear bases of free metabelian Leibniz superalgebras.
Define
\begin{equation}\label{s1=}
{S_1=\{[\cc_1\cdots\cc_p\aa_1\aa_2]_{_L}-(-1)^{|\aa_1||\aa_2|}[\cc_1\cdots\cc_p\aa_2\aa_1]_{_L}\mid \cc_1,\dots,\cc_p, \aa_1,\aa_2\in\XX,\  p\geq2,\ \aa_2<\aa_1\}},
\end{equation}
$$S_2=\{[\cc_1\cdots\cc_p\aa\aa]_{_L}\mid \aa, \cc_1,\dots,\cc_p\in\XX,\ p\geq 2,\ |\aa|=1\},$$
\begin{equation}\label{s'=}
S'=S_1\cup S_2.
\end{equation}
It is clear that in $\Lbs(\XX)$, we have $\Id(S)=\Id(S_1)$ if $\chart(\kk)=2$, and $\Id(S)=\Id(S')$ if $\chart(\kk)\neq2$, where $S$ is defined as in Lemma \ref{st}.

Before constructing a linear basis of $\MLbs(X)$,  we first prove the following lemma, which is helpful for calculating a \gsb\ for $\MLbs(X)$.

\begin{lemm}\label{ell2}
  Let $S_1$ and $S'$ be defined as (\ref{s1=}) and (\ref{s'=}). If~$\nu, [\aa_1\cdots\aa_p]_{_L}$ lie in $\NF(\XX),$ and $\ell(\nu), p\geq2 $, then in $\Lbs(\XX)$ we have
  $ (\nu[\aa_1\cdots\aa_p]_{_L})=0$, or
 $
 (\nu[\aa_1\cdots\aa_p]_{_L})=\sum \beta_j h_{j,s_j},
 $
  where each~$\beta_j\in\kk$ and each~$h_{j,s_j}$ is a normal $S_1$-polynomial if~$\chart(\kk)=2$; each~$h_{j,s_j}$ is a normal $S'$-polynomial if~$\chart(\kk)\neq 2$. Moreover, for each~$j$ satisfying~$\beta_j\neq 0$, we have~$\deg(\overline{h_{j,s_j}})\leq\deg(\nu)+\deg([\aa_1\cdots\aa_p]_{_L})$.
\end{lemm}

\begin{proof}
   Assume $(\nu[\aa_1\cdots\aa_p]_{_L})\neq0$.
   We use induction on $p$. Assume $p=2$. If~$\aa_{1}\neq\aa_2$, then there is nothing to prove. If~$\aa_{1}=\aa_2$, then we need to consider whether $|\aa_1|$  equals $0$ or $1$.

   \ITEM1 Supposing~$|\aa_1|=0$, we have~ $(\nu[\aa_1\aa_2]_{_L})=[\nu\aa_1\aa_2]_{_L}-(-1)^{|\aa_1||\aa_2|}[\nu\aa_2\aa_1]_{_L}=0$.

   \ITEM2 Assuming~$|\aa_1|=1$, we get~$(\nu[\aa_1\aa_2]_{_L})=2[\nu\aa_1\aa_1]_{_L}$. If~$\chart(\kk)=2$, then~$[\nu\aa_1\aa_1]_{_L}=0$. Otherwise~$[\nu\aa_1\aa_1]_{_L}$ is a normal $S'$-polynomial.

    If $p>2$ holds, then by induction hypothesis, we get
\begin{align*}
  (\nu[\aa_1\cdots\aa_p]_{_L})
&=((\nu[\aa_1\cdots\aa_{p-1}]_{_L})\aa_p)-(-1)^{|\aa_p||[\aa_1\cdots\aa_{p-1}]_{_L}|}((\nu\aa_p)[\aa_1\cdots\aa_{p-1}]_{_L})\\
&=\sum\beta_i' (h'_{i,s_i'}\aa_p)-(-1)^{|\aa_p||[\aa_1\cdots\aa_{p-1}]_{_L}|} \sum \beta''_j h''_{j,s_j''},
\end{align*}
where $\beta_i',\beta''_j\in\kk$, and~$(h'_{i,s_i'}\aa_p)$, $h''_{j,s_j''}$ are normal $S_1$-polynomials if~$\chart(\kk)=2$; $(h'_{i,s_i'}\aa_p)$, $h''_{j,s_j''}$ are normal $S'$-polynomials if~$\chart(\kk)\neq2$.
Furthermore, $\deg(\overline{(h'_{i,s_i'}\aa_p)})$ and $\deg(\overline{h''_{j,s_j''}})$ are smaller than or equal to $\deg(\nu)+\deg([\aa_1\cdots\aa_p]_{_L})$.
\end{proof}

\begin{theorem}\label{th4.7}
Let $\MLbs(X)$ be the free metabelian Leibniz superalgebra generated by a well-ordered set $X=X_0\cup X_1$, and let $S_1$ and $S'$ be defined as (\ref{s1=}) and (\ref{s'=}).

(1) If~$\chart(\kk)=2$, then the set~$S_1$ is a \gsb\ in~$\Lbs(X)$, and the following set forms a linear basis of~$\MLbs(X)$:
$$\{[\aa_1\aa_2\cdots\aa_n]_{_L}\mid\aa_1,\dots,\aa_n\in\XX,\ n\geq1,\ \aa_3\leq\aa_4\leq\cdots\leq\aa_n\}.$$

(2) If~$\chart(\kk)\neq2$, then the set~$S'$ is a \gsb\ in~$\Lbs(X)$, and the following set forms a linear basis of~$\MLbs(X)$:
$$\{[\aa_1\aa_2\cdots\aa_n]_{_L}\mid\aa_1,\ldots,\aa_n\in\XX,\ n\geq1,\ \aa_3\leq\cdots\leq\aa_n;\ \mbox{for all}\ j\geq3, \mbox{if}\ |\aa_j|=1, \mbox{then}\ \aa_j\neq\aa_{j+1}\}.$$
\end{theorem}

\begin{proof}
First we consider the left multiplication composition
$$
h:=(\mu([\cc_1\cdots\cc_p\aa_1\aa_2]_{_L}-(-1)^{|\aa_1||\aa_2|}[\cc_1\cdots\cc_p\aa_2\aa_1]_{_L})),
$$
where~$\mu\in\NF(\XX)$, ${\cc_1,\dots,\cc_p,\aa_1, \aa_2\in \XX}$, $\aa_2\leq\aa_1$ and $p\geq2$. Assume that $ h\neq 0$.

If~$\ell(\mu)\geq 2$ holds, then by Lemma \ref{ell2},  we have
$$ h\equiv 0 \mod(S_1,\deg(\mu)+\deg([\cc_1\cdots\cc_p\aa_1\aa_2]_{_L}))\ \mbox{if }\ \chart(\kk)=2$$ and
$$ h\equiv 0 \mod(S',\deg(\mu)+\deg([\cc_1\cdots\cc_p\aa_1\aa_2]_{_L}))\ \mbox{if }\ \chart(\kk)\neq 2.$$

Now we suppose~$\ell(\mu)=1$ and $\aa_2<\aa_1$. Then  we deduce
\begin{multline*}
  h
=((\mu([\cc_1\cdots\cc_p]_{_L}\aa_1))\aa_2)
-(-1)^{|\aa_2||([\cc_1\cdots\cc_p]_{_L}\aa_1)|}((\mu\aa_2)([\cc_1\cdots\cc_p]_{_L}\aa_1))\\
-(-1)^{|\aa_1||\aa_2|}((\mu([\cc_1\cdots\cc_p]_{_L}\aa_2))\aa_1)
+(-1)^{|\aa_1||\aa_2|+|\aa_1||([\cc_1\cdots\cc_p]_{_L}\aa_2)|}((\mu\aa_1)([\cc_1\cdots\cc_p]_{_L}\aa_2))\\
=[\mu[\cc_1\cdots\cc_p]_{_L}\aa_1\aa_2]_{_L}
-(-1)^{|\aa_1||[\cc_1\cdots\cc_p]_{_L}|}[\mu\aa_1[\cc_1\cdots\cc_p]_{_L}\aa_2]_{_L} \\ -(-1)^{|\aa_1||\aa_2|}[\mu[\cc_1\cdots\cc_p]_{_L}\aa_2\aa_1]_{_L}
+(-1)^{|\aa_1||\aa_2|+|\aa_2||[\cc_1\cdots\cc_p]_{_L}|}[\mu\aa_2[\cc_1\cdots\cc_p]_{_L}\aa_1]_{_L}\\
-(-1)^{|\aa_2||([\cc_1\cdots\cc_p]_{_L}\aa_1)|}((\mu\aa_2)([\cc_1\cdots\cc_p]_{_L}\aa_1))+(-1)^{|\aa_1||\aa_2|+|\aa_1||([\cc_1\cdots\cc_p]_{_L}\aa_2)|}((\mu\aa_1)([\cc_1\cdots\cc_p]_{_L}\aa_2)).
\end{multline*}

By Lemma \ref{multiplication},
we obtain \begin{align*}
  &[\mu[\cc_1\cdots\cc_p]_{_L}\aa_1\aa_2]_{_L}-(-1)^{|\aa_1||\aa_2|}[\mu[\cc_1\cdots\cc_p]_{_L}\aa_2\aa_1]_{_L}\\
  =&\sum\alpha_i[[\mu\xx_{i_1}\cdots\xx_{i_p}]_{_L}\aa_1\aa_2]_{_L}
  -(-1)^{|\aa_1||\aa_2|}\sum\alpha_i[[\mu\xx_{i_1}\cdots\xx_{i_p}]_{_L}\aa_2\aa_1]_{_L}\\
  \equiv&\ 0 \mod(S_1,\deg(\mu)+\deg([\cc_1\cdots\cc_p\aa_1\aa_2]_{_L})),
\end{align*}
where each $\alpha_i\in \kk$ and each~$(\xx_{i_1},\ldots,\xx_{i_{p}})$ is a permutation of~$(\cc_1,\ldots,\cc_p)$.
Note that, by Lemma \ref{ell2} and Remark \ref{hsa} the polynomials~$[\mu\aa_1[\cc_1\cdots\cc_p]_{_L}\aa_2]_{_L}$, $[\mu\aa_2[\cc_1\cdots\cc_p]_{_L}\aa_1]_{_L}$,  $((\mu\aa_2)([\cc_1\cdots\cc_p]_{_L}\aa_1))$ and $((\mu\aa_1)([\cc_1\cdots\cc_p]_{_L}\aa_2))$ are trivial modulo $(S_1,\deg(\mu)+\deg([\cc_1\cdots\cc_p\aa_1\aa_2]_{_L})$, if~$\chart(\kk)=2$; and they are  trivial modulo $(S',\deg(\mu)+\deg([\cc_1\cdots\cc_p\aa_1\aa_2]_{_L})$, if $\chart(\kk)\neq 2$.

Now we  assume $\ell(\mu)=1$ and $\aa_1=\aa_2=\aa\in\XX_1$.
Then by Lemmas \ref{multiplication} and \ref{ell2},  we have
\begin{align*}
&\ \ \   (\mu[\cc_1\cdots\cc_p\aa\aa]_{_L})
=((\mu([\cc_1\cdots\cc_p]_{_L}\aa))\aa)
-(-1)^{|\aa||([\cc_1\cdots\cc_p]_{_L}\aa)|}((\mu\aa)([\cc_1\cdots\cc_p]_{_L}\aa))\\
&=[\mu[\cc_1\cdots\cc_p]_{_L}\aa\aa]_{_L}
-(-1)^{|\aa||[\cc_1\cdots\cc_p]_{_L}|}[\mu\aa[\cc_1\cdots\cc_p]_{_L}\aa]_{_L} -(-1)^{|\aa||([\cc_1\cdots\cc_p]_{_L}\aa)|}((\mu\aa)([\cc_1\cdots\cc_p]_{_L}\aa))\\
&\equiv 0 \mod(S',\deg(\mu)+\deg([\cc_1\cdots c_p\aa\aa]_{_L})).
\end{align*}

Thus every left multiplication composition in $S_1$ is trivial, if~$\chart(\kk)=2$, and every left multiplication composition in $S'$ is trivial, if~$\chart(\kk)\neq2$.

Now we show that every inclusion composition in $S_1$ is trivial.
For the inclusion composition~$(f,g)_{\ov{f}}$, suppose that $f=[\cc_1\cdots\cc_p\aa_1\aa_2]_{_L}-(-1)^{|\aa_1||\aa_2|}[\cc_1\cdots\cc_p\aa_2\aa_1]_{_L}$, where $\cc_1,\dots,\cc_p$, $\aa_1,\aa_2$ lie in~$\XX,\ \aa_2<\aa_1$ and $p\geq2$. Then we have two cases to consider.

 \ITEM1 Let $\ov{g}=[\cc_1\cdots\cc_t]_{_L},\ 4\leq t\leq p$ and $\cc_t<\cc_{t-1}$.
 Then we   have
$$[\cc_1\cdots\cc_{t-1}\cc_t\cdots\cc_p\aa_2\aa_1]_{_L}\equiv
(-1)^{|\cc_{t-1}||\cc_t|}[\cc_1\cdots\cc_t\cc_{t-1}\cdots\cc_p\aa_2\aa_1]_{_L}\mod (S_1, \bar{\ff}),$$
$$[\cc_1\cdots\cc_t\cc_{t-1}\cdots\cc_p\aa_1\aa_2]_{_L}\equiv
(-1)^{|\aa_2||\aa_1|}
[\cc_1\cdots\cc_t\cc_{t-1}\cdots\cc_p\aa_2\aa_1]_{_L}\mod (S_1, \bar{\ff}).$$
 So  we deduce
\begin{multline*}(\ff,\gg)_{\bar{\ff}}
=-(-1)^{|\aa_2||\aa_1|}[\cc_1\cdots\cc_p\aa_2\aa_1]_{_L}
+(-1)^{|\cc_{t-1}||\cc_t|}[\cc_1\cdots\cc_t\cc_{t-1}\cdots\cc_p\aa_1\aa_2]_{_L}\\
\equiv -(-1)^{|\aa_2||\aa_1|+|\cc_{t-1}||\cc_t|}
([\cc_1\cdots\cc_t\cc_{t-1}\cdots\cc_p\aa_2\aa_1]_{_L}
    -[\cc_1\cdots\cc_t\cc_{t-1}\cdots\cc_p\aa_2\aa_1]_{_L} )
\equiv0 \mod (S_1, \bar{\ff}).
\end{multline*}

 \ITEM2 Set $\ov{g}=[\cc_1\cdots\cc_p\aa_1]_{_L}$ such that~$\aa_1<\cc_p$ and $p\geq 3$. Then we get~$\aa_2<\aa_1< c_p$ and
$[\cc_1\cdots\cc_p\aa_2\aa_1]_{_L}\equiv
(-1)^{|c_p||\aa_2|+|c_p||\aa_1|}([\cc_1\cdots\cc_{p-1}\aa_2\aa_1\cc_p]_{_L}\mod (S_1, \bar{\ff})$,\\
$[\cc_1\cdots\cc_{p-1}\aa_1\cc_p\aa_2]_{_L}\equiv
(-1)^{|\aa_2||\aa_1|+|c_p||\aa_2|}
([\cc_1\cdots\cc_{p-1}\aa_2\aa_1\cc_p]_{_L}\mod (S_1, \bar{\ff})$.
Thus
\begin{multline*}(\ff,\gg)_{\bar{\ff}}
=-(-1)^{|\aa_1||\aa_2|}[\cc_1\cdots\cc_p\aa_2\aa_1]_{_L}
+(-1)^{|\aa_1||\cc_p|}[\cc_1\cdots\cc_{p-1}\aa_1\cc_p\aa_2]_{_L}\\
\equiv -(-1)^{|\aa_2||\aa_1|+|c_p||\aa_2|+|c_p||\aa_1|}([\cc_1\cdots\cc_{p-1}\aa_2\aa_1\cc_p]_{_L}
    -[\cc_1\cdots\cc_{p-1}\aa_2\aa_1\cc_p]_{_L} )
\equiv0 \mod (S_1, \bar{\ff}).
\end{multline*}

This shows that $S_1$ is a \gsb\ in $\Lbs(\XX)$, if~$\chart(\kk)=2$.
Therefore, by Lemma~\ref{T-metabelian} and by Theorem~\ref{cd-lemma for Lb},  the following set:
$$\Irr(S_1)=\{[\aa_1\aa_2\cdots\aa_n]_{_L}\mid\aa_i\in\XX, n\geq i\geq1, n\geq1, \aa_3\leq\aa_4\leq\cdots\leq\aa_n\}$$
 forms a linear basis of~$\MLbs(X)=\Lbs(\XX)/\Id(S_1)$. So (1) holds.

To show (2), we prove that every inclusion composition~$(f,g)_{\ov{f}}$ of $S'$ is trivial.

If $f=[\cc_1\cdots\cc_p\aa_1\aa_2]_{_L}-(-1)^{|\aa_1||\aa_2|}[\cc_1\cdots\cc_p\aa_2\aa_1]_{_L}$, where $\cc_1,\dots,\cc_p,\aa_1,\aa_2\in\XX,\ \aa_2<\aa_1$ and $p\geq2$, then we have two cases to consider.
 \ITEM1 $\ov{g}=[\cc_1\cdots\cc_t]_{_L}, 4\leq t\leq p$. If $\cc_t<\cc_{t-1}$, then the $(f,g)_{\ov{f}}$ is trivial modulo $(S',\ov{f})$ because every inclusion composition in $S_1$ is trivial.
 If $\cc_{t-1}=\cc_t\in\XX_1$, then
 ${(\ff,\gg)_{\bar{\ff}}
=-(-1)^{|\aa_1||\aa_2|}[\cc_1\cdots\cc_{t-1}\cc_t\cdots\cc_p\aa_2\aa_1]_{_L}
\equiv 0 \mod (S', \bar{\ff}).}$
 \ITEM2~$\ov{g}=[\cc_1\cdots\cc_p\aa_1]_{_L}$.
 If~$\aa_1<\cc_p$, then the result also holds because of (1).
 Assume that~$\cc_p=\aa_1\in\XX_1$ and $p\geq 3$. Then we get~$\aa_2<c_p=\aa_1$ and
$$(\ff,\gg)_{\bar{\ff}}
=-(-1)^{|\aa_1||\aa_2|}[\cc_1\cdots\cc_p\aa_2\aa_1]_{_L}
\equiv -(-1)^{|\aa_2||\aa_1|+|c_p||\aa_2|}[\cc_1\cdots\cc_{p-1}\aa_2\cc_p\aa_1]_{_L}
\equiv0 \mod (S', \bar{\ff}).
$$

If $f=[\cc_1\cdots\cc_p\aa_1\aa_1]_{_L}$, where $\cc_1,\dots,\cc_p\in\XX,\  \aa_1\in\XX_1,\ p\geq2$, then~${(f,g)_{\ov{f}}\equiv0 \mod(S',\ov{f})}$ follows immediately, regardless of whether $\ov{g}=[\cc_1\cdots\cc_p\aa_1]_{_L}$ or $\ov{g}=[\cc_1\cdots\cc_t]_{_L}, {t\leq p}$.

So the set~$S'$ is a \gsb\ in~$\Lbs(X)$ if~$\chart(\kk)\neq2$.
Therefore, by Lemma~\ref{T-metabelian} and  Theorem~\ref{cd-lemma for Lb},  the following set forms a linear basis of~${\MLbs(X)=\Lbs(\XX)/\Id(S')}$:
\begin{align*}
\Irr(S')=\{&[\aa_1\aa_2\cdots\aa_n]_{_L}\mid\aa_1,\ldots,\aa_n\in\XX,\ n\geq1,\  \aa_3\leq\cdots\leq\aa_n;\\
& \mbox{for all}\ j\geq3, \mbox{if}\ |\aa_j|=1, \mbox{then}\ \aa_j\neq\aa_{j+1}\}.
\end{align*}

\end{proof}

In the free Leibniz algebra $\Lb(X)$ generated by $X$, we have $S_2=\emptyset$ in (\ref{s'=}). Then by Theorem \ref{th4.7},
$$
S_1=S'=\{[\cc_1\cdots\cc_p\aa_1\aa_2]_{_L}-[\cc_1\cdots\cc_p\aa_2\aa_1]_{_L}\mid \cc_1,\dots,\cc_p, \aa_1,\aa_2\in\XX,\  p\geq2,\ \aa_2<\aa_1\}
$$
is a \gsb\ in $\Lb(X)$.
Thus we have the following corollary.

\begin{coro}\label{coro4.17}
Let $\MLb(X)$ be the free metabelian Leibniz algebra generated by a well-ordered set $X$. Then the set $
\{[\cc_1\cdots\cc_p\aa_1\aa_2]_{_L}-[\cc_1\cdots\cc_p\aa_2\aa_1]_{_L}\mid \cc_1,\dots,\cc_p, \aa_1,\aa_2\in\XX,\  p\geq2,\ \aa_2<\aa_1\}
$
is a \gsb\ in $\Lb(X)$, and
the following set is a linear basis of $\MLb(X)$:
 $$\{[\aa_1\aa_2\cdots\aa_n]_{_L}\mid\aa_1\dots\aa_n\in\XX,\ n\geq1,\ \aa_3\leq\aa_4\leq\cdots\leq\aa_n\}.$$
\end{coro}

Corollary \ref{coro4.17} is obtained in~\cite{mlei} when $X$ is a finite set.

\subsection{Linear bases of free metabelian Lie algebras}

Recall that a \emph{Lie algebra} $\mathcal{L}$ is a Leibniz algebra that satisfies the identity~$(x x)=0$  for all~$x\in \mathcal{L}$.

\begin{defi}\cite{mlie}\label{mlie}
A Lie algebra $(\mathcal{L}, (- -))$ is called \emph{metabelian} if ~$((x_1x_2)(x_3x_4))=0$ for~$x_1, x_2, x_3, x_4$ in~$\mathcal{L}$.
\end{defi}

\begin{lemm}\label{T}
Let~$T=\{(\mu\nu), (\mu'\nu')+(\nu'\mu'),(\mu'\mu') \mid \mu, \nu, \mu', \nu'\in \NF(\XX), \ell(\mu),\ell(\nu)\geq2\}$. Then $\Lb(X)/\Id(T)$ is the free metabelian Lie algebra generated by $X$.
\end{lemm}
\begin{proof}For all~$f=\sum_i\alpha_i\mu_i\in \Lb(X)$, where each~$\alpha_i\in\kk,\ \mu_i\in\NF(X)$, we get, in~$\Lb(X)/\Id(T)$,
$$(ff)=\sum_{i}\alpha_i^2(\mu_i\mu_i)+\sum_{i\neq j}\alpha_i\alpha_j((\mu_i\mu_j)+(\mu_j\mu_i))=0.$$
It is clear that~$\Lb(X)/\Id(T)$ is the free metabelian Lie algebra generated by $X$.
\end{proof}

 \begin{lemm}\label{lie-st}
Let~$S=\{[\mu\aa_1\aa_2]_{_L}-[\mu\aa_2\aa_1]_{_L}, [\bb\aa\cc]_{_L}-[\cc\aa\bb]_{_L}+[\cc\bb\aa]_{_L}, [\aa_1\aa_2]_{_L}+[\aa_2\aa_1]_{_L}, [dd]_{_L}\mid {\mu\in \NF(\XX)},\ \aa_1,\aa_2, \aa, \bb, \cc, d\in\XX,\ {\ell(\mu)\geq2},\ \aa_2<\aa_1,\ \cc<\bb<\aa\}$,
and  let~$T$ be as  defined in Lemma \ref{T}.
Then in $\Lb(\XX)$, we have~$\Id(T)=\Id(S)$.
\end{lemm}

\begin{proof} We first show that~$S\subseteq \Id(T)$.
For all~$\mu\in\NF(\XX), \aa,\bb,\cc,d,\aa_1,\aa_2$ belong to~$X$ such that~$\aa_2<\aa_1, \cc<\bb<\aa$ and~$\ell(\mu)\geq2$,
we get~$[\mu\aa_1\aa_2]_{_L}-[\mu\aa_2\aa_1]_{_L}=(\mu(\aa_1\aa_2))\in \Id(T)$. It is also clear that both~$[\aa_1\aa_2]_{_L}+[\aa_2\aa_1]_{_L}$ and~$[dd]_{_L}$ lie in~$\Id(T)$.

Since~$\Lb(X)$ is a Leibniz algebra, we get~$[cab]_{_L}=[cba]_{_L}+(c(ab))$ and thus
\begin{multline*}[\bb\aa\cc]_{_L}-[\cc\aa\bb]_{_L}+[\cc\bb\aa]_{_L}
 = [\bb\aa\cc]_{_L}-[cba]_{_L}-(c(ab))+[\cc\bb\aa]_{_L}\\
 =[\bb\aa\cc]_{_L}-(c(ab))
 =(((\bb\aa)+(ab))\cc)-(((ab)\cc)+(c(ab)) )\in \Id(T).
 \end{multline*}

Now we turn to show that~$T\subseteq \Id(S)$. For all~$\mu,\nu$ in~$\NF(X)$, we assume that~$\mu=[\cc_1\cdots\cc_p]_{_L}, \nu=[\dd_1\cdots\dd_t]_{_L}$, where~$\cc_1,\dots,\cc_p,\dd_1,\dots,\dd_t\in\XX$ and $p, t\geq2$.
We use induction on $t$ to show that~$(\mu\nu)\in \Id(S)$. If $t=2$, then we get
$$([\cc_1\cdots\cc_p]_{_L}[\dd_1\dd_2]_{_L})=[\cc_1\cdots\cc_p\dd_1\dd_2]_{_L}
-[\cc_1\cdots\cc_p\dd_2\dd_1]_{_L}\in \Id(S).$$ For~$t>2$, we get \begin{multline*}([\cc_1\cdots\cc_p]_{_L}[\dd_1\cdots\dd_t]_{_L})
=(([\cc_1\cdots\cc_p]_{_L}[\dd_1\cdots\dd_{t-1}]_{_L})\dd_t)
-(([\cc_1\cdots\cc_p]_{_L}\dd_t)[\dd_1\cdots\dd_{t-1}]_{_L}).
\end{multline*}
By induction hypothesis, both~$([\cc_1\cdots\cc_p\dd_t]_{_L}[\dd_1\cdots\dd_{t-1}]_{_L})$ and~$([\cc_1\cdots\cc_p]_{_L}[\dd_1\cdots\dd_{t-1}]_{_L})$ lie in~$\Id(S)$.
In particular, we get~$(\mu\nu)+(\nu\mu)\in\Id(S)$ and~$(\mu\mu)\in\Id(S)$.

Since~$(dd)$ lies in~$\Id(S)$, where~$d\in\XX$, we get~$(\mu'\mu')\in\Id(S)$ for all~$\mu'$ in~$\NF(X)$. It remains to show that~$(\mu'\nu')+(\nu'\mu')\in\Id(S)$, where $\ell(\mu')=1$. We assume that~$\mu'=\aa$ and~$ \nu'=[\bb_1\bb_2\cdots\bb_n]_{_L}$,
where~$\aa, \bb_1,\dots,\bb_n$ lie in~$\XX$. We use induction on~$n$. For~$n=1$, there is nothing to prove.
For~$n=2$, since both~$(b_1b_2)+(b_2b_1)$ and~$(b_1b_1)$ lie in~$\Id(S)$, we may assume that~$\bb_2<\bb_1$.
By the Leibniz identity, we get
$$(\aa[\bb_1\bb_2]_{_L})+([\bb_1\bb_2]_{_L}\aa)
=[\aa\bb_1\bb_2]_{_L}-[\aa\bb_2\bb_1]_{_L}+[\bb_1\bb_2\aa]_{_L}.$$
So if~$a=b_1$ or~$a=b_2$ or~$a<b_2$ or~$b_1<a$, we immediately get~$(\aa[\bb_1\bb_2]_{_L})+([\bb_1\bb_2]_{_L}\aa)$ in~$\Id(S)$. If $\bb_2<\aa<\bb_1$, then we get
\begin{multline*}
   (\aa[\bb_1\bb_2]_{_L})+([\bb_1\bb_2]_{_L}\aa)
   =[\aa\bb_1\bb_2]_{_L}-[\aa\bb_2\bb_1]_{_L}+[\bb_1\bb_2\aa]_{_L}\\
   =[\aa\bb_1\bb_2]_{_L}+[\bb_2\aa\bb_1]_{_L}-[\bb_2\bb_1\aa]_{_L}-([\aa\bb_2\bb_1]_{_L}+[\bb_2\aa\bb_1]_{_L})
   +([\bb_1\bb_2\aa]_{_L}+[\bb_2\bb_1\aa]_{_L})\in \Id(S).
\end{multline*}
  If $n>2$, then by induction hypothesis, we get
\begin{align*}
(\mu'\nu')+(\nu'\mu')
&=(\aa[\bb_1\bb_2\cdots\bb_n]_{_L})+[\bb_1\bb_2\cdots\bb_n\aa]_{_L}\\
&=((\aa[\bb_1\bb_2\cdots\bb_{n-1}]_{_L})\bb_n)-((\aa\bb_n)[\bb_1\bb_2\cdots\bb_{n-1}]_{_L})+[\bb_1\bb_2\cdots\bb_n\aa]_{_L}\\
&=((\aa[\bb_1\bb_2\cdots\bb_{n-1}]_{_L})\bb_n)+(([\bb_1\bb_2\cdots\bb_{n-1}]_{_L}\aa)\bb_n)\\
&\ \ \ +([\bb_1\bb_2\cdots\bb_{n-1}]_{_L}(\bb_n\aa))-((\aa\bb_n)[\bb_1\bb_2\cdots\bb_{n-1}]_{_L}) \in\Id(S).
\end{align*}

 Therefore, we get~$\Id(S)=\Id(T)$.
\end{proof}

\begin{lemm}\label{ell22}
  Let $S$ be defined as in Lemma \ref{lie-st}. If~$\nu, [\aa_1\cdots\aa_p]_{_L}$ lie in $\NF(\XX),$ and $\ell(\nu), p\geq2 $, then in $\Lb(\XX)$ we have
  $ (\nu[\aa_1\cdots\aa_p]_{_L})=0$, or
 $
 (\nu[\aa_1\cdots\aa_p]_{_L})=\sum \beta_j h_{j,s_j},
 $
  where each~$\beta_j\in\kk$, each~$h_{j,s_j}$ is a normal $S$-polynomial and for each~$j$ satisfying~$\beta_j\neq 0$, we have~$\deg(\overline{h_{j,s_j}})\leq\deg(\nu)+\deg([\aa_1\cdots\aa_p]_{_L})$.
\end{lemm}

\begin{proof}
   Assume $(\nu[\aa_1\cdots\aa_p]_{_L})\neq0$.
   We use induction on $p$. Assuming $p=2$,  the result holds clearly.
    If $p>2$, then by induction hypothesis, we get
\begin{multline*}
  (\nu[\aa_1\cdots\aa_p]_{_L})
=((\nu[\aa_1\cdots\aa_{p-1}]_{_L})\aa_p)-((\nu\aa_p)[\aa_1\cdots\aa_{p-1}]_{_L})
=\sum\beta_i' (h'_{i,s_i'}\aa_p)- \sum \beta''_j h''_{j,s_j''},
\end{multline*}
where $\beta_i',\beta''_j\in\kk$, and~$(h'_{i,s_i'}\aa_p)$, $h''_{j,s_j''}$ are normal $S$-polynomials with~$\deg(\overline{(h'_{i,s_i'}\aa_p)})$ and $\deg(\overline{h''_{j,s_j''}})$ are smaller than or equal to $\deg(\nu)+\deg([\aa_1\cdots\aa_p]_{_L})$.
\end{proof}

By the above lemma and Theorem~\ref{cd-lemma for Lb}, we  can construct  alternative linear bases of free metabelian Lie algebras.
\begin{theorem}\label{cor4.13}
Let $\MLie(X)$ be the free metabelian Lie algebra generated by a well-ordered set~$X$.
Then the set
$S=\{[\mu\aa_1\aa_2]_{_L}-[\mu\aa_2\aa_1]_{_L}, [\bb\aa\cc]_{_L}-[\cc\aa\bb]_{_L}+[\cc\bb\aa]_{_L}, [\aa_1\aa_2]_{_L}+[\aa_2\aa_1]_{_L}, [dd]_{_L}\mid {\mu\in \NF(\XX)},\ \aa_1,\aa_2, \aa, \bb, \cc, d\in\XX,\ {\ell(\mu)\geq2},\ \aa_2<\aa_1,\ \cc<\bb<\aa\}$
is a \gsb\ in~$\Lb(\XX)$, and the following set forms a linear basis of $\MLie(X)$:
$$
{\{[\aa_1\aa_2\cdots\aa_n]_{_L}\mid\aa_1,\dots,\aa_n\in\XX,\ n\geq1,\ \aa_1<\aa_2,\ \aa_1\leq \aa_3\leq\aa_4\leq\cdots\leq\aa_n\}}.
$$

\end{theorem}

\begin{proof}

First we prove that all the left
multiplication compositions in $S$ are trivial.
Let $\mu\in \NF(\XX), \cc_1,\dots,\cc_p,\aa_1,\aa_2,\aa,\bb,\cc,d\in \XX,\ \aa_2<\aa_1,\ \cc<\bb<\aa$ and $p\geq2$. Then we have the followings.

(i) $(\mu[dd]_{_L})=[\mu dd]_{_L}-[\mu dd]_{_L}= 0$.

(ii) $(\mu([\aa_1\aa_2]_{_L}+[\aa_2\aa_1]_{_L}))
=[\mu\aa_1\aa_2]_{_L}-[\mu\aa_2\aa_1]_{_L}+[\mu\aa_2\aa_1]_{_L}-[\mu\aa_1\aa_2]_{_L}
= 0.$\begin{align*}
(iii)&\  \ \ (\mu([\bb\aa\cc]_{_L}-[\cc\aa\bb]_{_L}+[\cc\bb\aa]_{_L}))\\
&=((\mu[\bb\aa]_{_L})\cc)-((\mu\cc)[\bb\aa]_{_L})-((\mu[\cc\aa]_{_L})\bb)+((\mu\bb)[\cc\aa]_{_L})+((\mu[\cc\bb]_{_L})\aa)-((\mu\aa)[\cc\bb]_{_L})\\
&=[\mu\bb\aa\cc]_{_L}-[\mu\aa\bb\cc]_{_L}-[\mu\cc\bb\aa]_{_L}+[\mu\cc\aa\bb]_{_L}-[\mu\cc\aa\bb]_{_L}
 +[\mu\aa\cc\bb]_{_L}
 +[\mu\bb\cc\aa]_{_L}-[\mu\bb\aa\cc]_{_L}\\
&\ \ \  +[\mu\cc\bb\aa]_{_L}-[\mu\bb\cc\aa]_{_L}-[\mu\aa\cc\bb]_{_L}
 +[\mu\aa\bb\cc]_{_L}
= 0 .
\end{align*}

The following result holds by using Lemmas \ref{ell22} and \ref{multiplication}.
\begin{align*}
(iv)\ &\ \ \  (\mu([\cc_1\cdots\cc_p\aa_1\aa_2]_{_L}-[\cc_1\cdots\cc_p\aa_2\aa_1]_{_L}))\\
&=((\mu([\cc_1\cdots\cc_p]_{_L}\aa_1))\aa_2)-((\mu\aa_2)([\cc_1\cdots\cc_p]_{_L}\aa_1))
-((\mu([\cc_1\cdots\cc_p]_{_L}\aa_2))\aa_1)+((\mu\aa_1)([\cc_1\cdots\cc_p]_{_L}\aa_2))\\
&=[\mu[\cc_1\cdots\cc_p]_{_L}\aa_1\aa_2]_{_L}-[\mu\aa_1[\cc_1\cdots\cc_p]_{_L}\aa_2]_{_L} -[\mu\aa_2[\cc_1\cdots\cc_p]_{_L}\aa_1]_{_L}+[\mu\aa_2\aa_1[\cc_1\cdots\cc_p]_{_L}]_{_L}\\
&\ \ \  -[\mu[\cc_1\cdots\cc_p]_{_L}\aa_2\aa_1]_{_L}+[\mu\aa_2[\cc_1\cdots\cc_p]_{_L}\aa_1]_{_L} +[\mu\aa_1[\cc_1\cdots\cc_p]_{_L}\aa_2]_{_L}- [\mu\aa_1\aa_2[\cc_1\cdots\cc_p]_{_L}]_{_L}\\
&\equiv ((\mu([\cc_1\cdots\cc_p]_{_L}\aa_1))\aa_2)-[\mu[\cc_1\cdots\cc_p]_{_L}\aa_2\aa_1]_{_L}\\
&\equiv\sum\alpha_i [[\mu\xx_{i_1}\cdots\xx_{i_p}]_{_L}\aa_1\aa_2]_{_L}
-\sum\alpha_i [[\mu\xx_{i_1}\cdots\xx_{i_p}]_{_L}\aa_2\aa_1]_{_L}\\
&\equiv 0 \mod (S,\deg(\mu)+\deg([\cc_1\cdots c_p\aa_1\aa_2]_{_L})),
\end{align*}
where each~$(\xx_{i_1},\dots,\xx_{i_p})$ is a permutation of $(\cc_1,\dots,\cc_p)$.

Next we show that all the inclusion compositions are trivial. All possible inclusion compositions $(f,g_i)_{\ov{f}}$ with $1\leq i\leq 6$ where
$$\begin{array}{llll}
\ov{f}=[\cc_1\cdots\cc_p\aa_1\aa_2]_{_L},& \aa_2<\aa_1,\ 2\leq p;&\ \ \ \ \ \ \ \ \ \ \ \ \ \ \ \ \ \ \ &\\
\ov{g_1}=[\cc_1\cc_2]_{_L},& \cc_2<\cc_1;
&\ov{g_2}=[\cc_1\cc_2]_{_L},& \cc_1=\cc_2;\\
\ov{g_3}=[\cc_1\cc_2\cc_3]_{_L},& p\geq 3,\ \cc_3<\cc_1<\cc_2;
&\ov{g_4}=[\cc_1\cc_2\aa_1]_{_L},& p=2,\ \aa_1<\cc_1<\cc_2;\\
\ov{g_5}=[\cc_1\cdots\cc_t]_{_L},& 4\leq t\leq p,\ \cc_t<\cc_{t-1};
&\ov{g_6}=[\cc_1\cdots\cc_p\aa_1]_{_L},& 3\leq p,\ \aa_1<\cc_p,
\end{array}$$
for all $\cc_1,\dots,\cc_p,\aa_1,\aa_2\in \XX$. We get
\begin{align*}
(\ff, \gg_1)_{\bar{\ff}}
&=-[\cc_1\cdots\cc_p\aa_2\aa_1]_{_L}-[\cc_2\cc_1\cdots\cc_p\aa_1\aa_2]_{_L}\\
&\equiv[\cc_2\cc_1\cdots\cc_p\aa_2\aa_1]_{_L}-[\cc_2\cc_1\cdots\cc_p\aa_2\aa_1]_{_L}
\equiv  0 \mod(S,\overline{\ff});\\
(\ff, \gg_2)_{\bar{\ff}}
&=-[\cc_1\cc_2\cdots\cc_p\aa_2\aa_1]_{_L}
\equiv  0\mod(S,\overline{\ff});\\
(\ff, \gg_3)_{\bar{\ff}}
&=-[\cc_1\cdots\cc_p\aa_2\aa_1]_{_L}
+[\cc_3\cc_2\cc_1\cdots\cc_p\aa_1\aa_2]_{_L}-[\cc_3\cc_1\cc_2\cdots\cc_p\aa_1\aa_2]_{_L}\\
&\equiv -[\cc_3\cc_2\cc_1\cdots\cc_p\aa_1\aa_2]_{_L}+[\cc_3\cc_1\cc_2\cdots\cc_p\aa_1\aa_2]_{_L}
+[\cc_3\cc_2\cc_1\cdots\cc_p\aa_2\aa_1]_{_L}-[\cc_3\cc_1\cc_2\cdots\cc_p\aa_2\aa_1]_{_L}\\
&\equiv  0\mod(S,\overline{\ff}).
\end{align*}

Assume that~$\ff=[\cc_1\cc_2\aa_1\aa_2]_{_L}-[\cc_1\cc_2\aa_2\aa_1]$ and~$ \gg_4=[\cc_1\cc_2\aa_1]_{_L}-[\aa_1\cc_2\cc_1]_{_L}+[\aa_1\cc_1\cc_2]_{_L}$, where $\aa_2<\aa_1<\cc_1<\cc_2$. Then
\begin{align*}
(\ff, \gg_4)_{\bar{\ff}}
&=-[\cc_1\cc_2\aa_2\aa_1]_{_L}+[\aa_1\cc_2\cc_1\aa_2]_{_L}-[\aa_1\cc_1\cc_2\aa_2]_{_L}\\
&\equiv-[\aa_2\cc_2\cc_1\aa_1]_{_L}+[\aa_2\cc_1\cc_2\aa_1]_{_L}+[\aa_1\cc_2\aa_2\cc_1]_{_L}-[\aa_1\cc_1\aa_2\cc_2]_{_L} \\
&\equiv  -[\aa_2\cc_2\aa_1\cc_1]_{_L}+[\aa_2\cc_1\cc_2\aa_1]_{_L}+[\aa_2\cc_2\aa_1\cc_1]_{_L}\\
&\ \ \ -[\aa_2\aa_1\cc_2\cc_1]_{_L}-[\aa_2\cc_1\aa_1\cc_2]_{_L}+[\aa_2\aa_1\cc_1\cc_2]_{_L}
\equiv 0\mod(S,\overline{\ff}).
\end{align*}

 Suppose that $f=[\cc_1\cdots\cc_p\aa_1\aa_2]_{_L}-[\cc_1\cdots\cc_p\aa_2\aa_1]_{_L}$, where $\cc_1,\dots,\cc_p,\aa_1,\aa_2\in\XX,\ \aa_2<\aa_1$ and $p\geq2$.
 Assuming $\ov{g_5}=[\cc_1\cdots\cc_t]_{_L}$ such that $4\leq t\leq p$ and $\cc_t<\cc_{t-1}$. Then we get
$$[\cc_1\cdots\cc_p\aa_2\aa_1]_{_L}\equiv[\cc_1\cdots\cc_t\cc_{t-1}\cdots\cc_p\aa_1\aa_2]_{_L}\equiv
[\cc_1\cdots\cc_t\cc_{t-1}\cdots\cc_p\aa_2\aa_1]_{_L}\mod (S, \bar{\ff}).$$
Thus
$$(\ff,\gg_5)_{\bar{\ff}}
=-[\cc_1\cdots\cc_p\aa_2\aa_1]_{_L}
+[\cc_1\cdots\cc_t\cc_{t-1}\cdots\cc_p\aa_1\aa_2]_{_L}
\equiv0 \mod (S, \bar{\ff}).
$$

Assume $\ov{g_6}=[\cc_1\cdots\cc_p\aa_1]_{_L}$ such that~$\aa_1<\cc_p$ and $p\geq 3$. Then we get~$\aa_2<\aa_1<c_p$ and
$[\cc_1\cdots\cc_p\aa_2\aa_1]_{_L}
\equiv
[\cc_1\cdots\cc_{p-1}\aa_2\aa_1\cc_p]_{_L}
\equiv[\cc_1\cdots\cc_{p-1}\aa_1\cc_p\aa_2]_{_L}\mod (S, \bar{\ff})$.
Thus
$$(\ff,\gg_6)_{\bar{\ff}}
=-[\cc_1\cdots\cc_p\aa_2\aa_1]_{_L}
+[\cc_1\cdots\cc_{p-1}\aa_1\cc_p\aa_2]_{_L}
\equiv0 \mod (S, \bar{\ff}).
$$

So~$S$ is a \gsb\ in~$\Lb(\XX)$.
Therefore, by Theorem~\ref{cd-lemma for Lb}, the following set forms a linear basis of~$\MLie(X)$:
$$
\Irr(S)={\{[\aa_1\aa_2\cdots\aa_n]_{_L}\mid\aa_1,\dots,\aa_n\in\XX,\ n\geq1,\ \aa_1<\aa_2,\ \aa_1\leq \aa_3\leq\aa_4\leq\cdots\leq\aa_n\}}.
$$
\end{proof}

In Theorem \ref{cor4.13}, if we apply~$(a_1\aa_2)=-(a_2a_1)$, then we get a set
 $$\{[\aa_1\aa_2\cdots\aa_n]_{_L}\mid\aa_1,\dots,\aa_n\in\XX,\ n\geq1,\  \aa_2<\aa_1,\ \aa_2\leq \aa_3\leq\aa_4\leq\cdots\leq\aa_n\},$$ which is exactly the linear basis of $\MLie(X)$ constructed in~\cite{mlie}.

\section{A characterization of extensions of Leibniz superalgebras}
Let $\mathfrak{A},\mathfrak{B}$ be Leibniz (super)algebras over a  field $k$. In this subsection, by using  Gr\"{o}bner-Shirshov bases theory for Leibniz (super)algebras, we give a complete characterization of   extensions  of   $\mathfrak{B}$ by   $\mathfrak{A}$, where  $\mathfrak{B}$ is presented by generators and relations.

Loday \cite{loday} has already given the definition of a module over Leibniz algebra. Here we extend it to  a definition of a  supermodule over Leibniz superalgebra.

\begin{defi}
Let $ \mathfrak{B}= \mathfrak{B}_0 \oplus\mathfrak{B}_1$ be a Leibniz superalgebra over a field $k$. A supermodule $\mathfrak{A}$ over $\mathfrak{B}$ is a direct sum decomposition $\mathfrak{A}=\mathfrak{A}_0\oplus \mathfrak{A}_1$ (as vector spaces) with two $k$-bilinear multiplications (we use only one `` $\cdot$ " to represent two supermodule operations as long as it will not cause any ambiguity):
\begin{equation}\label{algmod1}
\mathfrak{B} \times \mathfrak{A}\longrightarrow \mathfrak{A},  \ (x, f)\mapsto x\cdot f
\end{equation}
\begin{equation}\label{algmod2}
\mathfrak{A} \times \mathfrak{B}\longrightarrow \mathfrak{A}, \ (f, x)\mapsto f\cdot x
\end{equation}
satisfying
$$\mathfrak{A}_i\cdot \mathfrak{B}_j,\ \mathfrak{B}_i\cdot \mathfrak{A}_j \subseteq \mathfrak{A}_{i+j}\ \mbox{for all}\ i , j \ \mbox{in}\ \mathbb{Z}_2,$$
and the following three axioms,
\begin{align*}
f\cdot (xy)&=(f\cdot x)\cdot y-(-1)^{|x||y|}(f\cdot y)\cdot x,\\
x\cdot (y\cdot f)&=(xy)\cdot f-(-1)^{|f||y|}(x\cdot f)\cdot y,\\
x\cdot (f\cdot y)&=(x\cdot f)\cdot y-(-1)^{|f||y|}(xy)\cdot f,
\end{align*}
for all $f\in \mathfrak{A}_0\cup \mathfrak{A}_1$ and $x,y\in \mathfrak{B}_0\cup\mathfrak{B}_1$.

 Suppose that $ \mathfrak{A}= \mathfrak{A}_0 \oplus\mathfrak{A}_1$ and $ \mathfrak{B}= \mathfrak{B}_0 \oplus\mathfrak{B}_1$ are two  Leibniz superalgebras. Then $ \mathfrak{A}$  is   called a  compatible~$ \mathfrak{B}$-supermodule if $ \mathfrak{A}$  is a~$ \mathfrak{B}$-supermodule
  and satisfies the following three axioms, for all~$f,f'\in \mathfrak{A}_0\cup\mathfrak{A}_1 , x\in\mathfrak{B}_0\cup\mathfrak{B}_1$,
\begin{align*}
x\cdot (\ff\ff')&=(x\cdot \ff)\cdot \ff'-(-1)^{|\ff||\ff'|}(x\cdot \ff')\cdot \ff,\\
\ff\cdot (\ff'\cdot x)&=(\ff\ff')\cdot x-(-1)^{|x||\ff'|}(\ff\cdot x)\cdot \ff',\\
\ff\cdot (x\cdot \ff')&=(\ff\cdot x)\cdot \ff'-(-1)^{|x||\ff'|}(\ff\ff')\cdot x.
\end{align*}
\end{defi}

For example, let~$\mA, \mathcal{B}$ be two Leibniz superalgebras. Then clearly $\mA$ is a compatible~$\mA$-supermodule.  And
if we define  supermodule operations: $\ff\cdot x=x\cdot \ff=0$, for all~$f\in \mA, x\in\mathcal{B}$,  then~$\mA$ is a compatible~$\mathcal{B}$-supermodule.

\begin{rema}
Assume that $ (\mathfrak{A}, (--))$ and ~$ (\mathfrak{B}, (--))$ are two Leibniz superalgebras. Then with the $k$-bilinear operations (\ref{algmod1}) and (\ref{algmod2}), $\mathfrak{A}$  is a  compatible~$\mathfrak{B}$-supermodule if and only if the following identity holds:
$$
(x\cdot(y\cdot z)) =((x\cdot y)\cdot z)-(-1)^{|z||y|}((x\cdot z)\cdot y), \ \mbox{for all}\  x,y,z\in \mathfrak{A}_0\cup\mathfrak{A}_1\cup\mathfrak{B}_0\cup\mathfrak{B}_1,
$$
where $x'\cdot y'$ means the product~$(x'y')$ of $\mathfrak{A}$ (resp. $\mathfrak{B}$)  if~$x',y'\in\mathfrak{A}$  (resp. $x',y'\in\mathfrak{B}$).

\end{rema}

Let $(\mathfrak{A}=\mathfrak{A}_0\oplus\mathfrak{A}_1, (--))$  be  a Leibniz superalgebra with  a well-ordered linear basis
$A=A_0\cup A_1$, where~$A_i$ is a linear basis of~$\mathfrak{A}_i$, $i=0,1$, and with the    multiplication table:
$$(aa')=\{a\cdot a'\},\ \ a, a'\in A,$$  where
$$
\{a\cdot a'\}:=\{(a a')\}
$$ is the product in~$\mA$; in particular, it is a linear combination of elements in $A_0$ (resp. $ A_1$), if~$(aa')\in \mathfrak{A}_0$ (resp. $(aa')\in \mathfrak{A}_1$).

Define $\deg(\aa)=1$  for all~$\aa\in A$. Then the set $\{(aa')-\{a\cdot a'\}\mid a, a'\in A\}$ is clearly a \gsb\ in $\Lbs( A)$.  Thus by Theorem \ref{cd-lemma for Lb}, $\mathfrak{A}$ has a presentation
 $$\mathfrak{A}=\Lbs( A\mid (aa')-\{a\cdot a'\}, a, a'\in A).$$
Let $\mathfrak{B}=\Lbs( B| R)$ be  a Leibniz  superalgebra
generated by a well-ordered set $B=B_0\cup B_1$  with defining relations $R$, where $R$ is a homogenous subset of $\Lbs(B)$. Define $\deg(\bb)=1$  for all~$\bb\in B$.

 Let $<$ be an  order on $A\cup B $ such that~${a<b}$  for all~$a\in A, b\in B$. It is clear that $<$ is a well
order on $A\cup B $. We extend~$<$ to the deg-length-lex order on $\NF(A\cup B)$ as Definition~\ref{order}. Then $<$  restricted  on~$\NF(A)$ or $\NF(B)$ is the deg-length-lex order on the corresponding set.

Let   $\lfloor\  \rfloor : R \rightarrow \mathfrak{A}_0\cup\mathfrak{A}_1,\ f\mapsto \lfloor f\rfloor$ be a map satisfying $|f|=|\lfloor f\rfloor|$, for all $f\in R$. Then we call $\lfloor\  \rfloor$ a \emph{factor set} of $\mathfrak{B}$ in $\mathfrak{A}$.

Suppose that $\mathfrak{A}$ is a  compatible $\Lbs( B)$-supermodule with the supermodule operation $``\cdot"$.
Then for all~$a\in A,\ \nu\in \NF(B),$
$$\{a\cdot\nu\}:=a\cdot\nu\ \mbox{ and}\  \ \{\nu\cdot a\}:=\nu\cdot a$$
are linear combinations of elements in~$A$.
Consider the following homogeneous polynomials in $\Lbs( A\cup B)$:
$$
\Omega _{aa'}:= (aa')- \{a\cdot a'\},\
\Upsilon_{_{ab}}:=(ab)-\{a\cdot b\},\
\Gamma _{\nu a}:=(\nu a)-\{\nu\cdot a\},\
\Theta _{_f}:=f-\lfloor f\rfloor,
$$
where~$a,a'\in A,\ b\in B,\ \nu\in \NF(B),\ f\in R$. Noting that~$\{-\cdot -\}$ is bilinear, if~${h=\sum\beta_j a_j\in\mathfrak{A}}$, $q=\sum\alpha_i \mu_i\in\Lbs(B)$, where ${\alpha_i, \beta_j\in\kk}$, $a_j\in A,\ \mu_i\in \NF(B)$, then we extend the above notation to
$$
\Omega_{ha}:=\sum\beta_j \Omega_{a_j a},\ \Omega_{ah}:=\sum\beta_j \Omega_{ a a_j},\ \Upsilon_{_{hb}}:=\sum\beta_j \Upsilon_{_{a_j b}},\ \Gamma_{\nu h}:=\sum\beta_j \Gamma_{\nu a_j},\ \Gamma_{ q a}:=\sum\alpha_i \Gamma_{\mu_i a},
$$
$$
\{h\cdot q\}:=\sum_{i,j}\beta_j\alpha_i \{a_j\cdot \mu_i \},\ \{q\cdot h\}:=\sum_{i,j}\alpha_i\beta_j \{\mu_i \cdot a_j \},\ \{h\cdot (\sum\beta'_ta_t)\}:=\sum_{j,t}\beta_j\beta'_t \{a_j\cdot a_t \},
$$
where~$\beta'_t\in \kk,\ a_t,a\in A,\ b\in B,\ \nu\in\NF(B)$.
Denote by
\begin{align*}
 &R_{(\lfloor\  \rfloor, \cdot)}:=\{\Omega _{aa'},\ \Upsilon_{_{ab}}, \ \Gamma _{\nu a},\ \Theta _{_f}\mid \ a,a'\in A,\ b\in B,\ \nu\in \NF(B),\ f\in R \},\\
 &R'_{(\lfloor\  \rfloor, \cdot)}:=\{\Omega _{aa'},\ \Upsilon_{_{ab}}, \ \Gamma _{b a},\ \Theta _{_f}\mid \ a,a'\in A,\ b\in B,\  f\in R \},\\
&\ \ \ \ \ \ \  E_{ (\mathfrak{A}, \mathfrak{B},\lfloor\  \rfloor,\cdot)}:=\Lbs(A\cup B\mid R_{(\lfloor\  \rfloor, \cdot)}).
 \end{align*}

\begin{rema} The subset  $R_{(\lfloor\  \rfloor, \cdot)}$ of $\Lbs(A\cup B)$ depends on factor set~$\lfloor\  \rfloor$ and supermodule operation~$``\cdot" $. A normal $R'_{(\lfloor\  \rfloor, \cdot)}$-polynomial in $\Lbs(A\cup B)$ is also a normal $R_{(\lfloor\  \rfloor, \cdot)}$-polynomial.
\end{rema}

\begin{lemm} \label{le2.2}
Assume $\mathfrak{B}=\Lbs(B| R)$ is  a Leibniz superalgebra
generated by a set $B$ with defining relations $R$,
where $R$  is a reduced Gr\"{o}bner-Shirshov basis in
$\Lbs(B)$ and the length of leading monomial of each polynomial in $R$ is greater than 1.
In $\Lbs(B)$, if $0\neq h\in\Id(R)$, then there exist unique $t\in\mathbb{Z}_{>0}$, normal $R$-polynomial $h_{i,s_i}$ and $0\neq \alpha_i\in\kk$, with $1\leq i\leq t$, such that $h=\sum\alpha_i h_{i,s_i}$, where $\ov{h_{t,s_t}}<\cdots<\ov{h_{2,s_2}}<\ov{h_{1,s_1}}=\ov{h}$, i.e., $h$ has a unique expression $h=h_{_R}:=\sum\alpha_i h_{i,s_i}$.
\end{lemm}

\begin{proof}
Note that every normal $R$-polynomial in $\Lbs(B)$ has the form: $[s\bb_1\cdots\bb_n]_{_L}$, where $s\in R, \bb_1,\dots,\bb_n\in B, n\geq 1,\ \ell(\ov{s})\geq 2$.
For each $0\neq h\in \Id(R)$, there is a normal $R$-polynomial $h_{1,s_1}$ such that $\ov{h}=\ov{h_{1,s_1}}$, since $R$ is a Gr\"{o}bner-Shirshov basis in
$\Lbs(B)$. If there is another normal $R$-polynomial $h'_{1,s'_1}$ such that $\ov{h}=\ov{h'_{1,s'_1}}$, then  $h'_{1,\ov{s'_1}}=\ov{h'_{1,s'_1}}=\ov{h_{1,s_1}}=h_{1,\ov{s_1}}$. It follows that $s_1,s'_1$ has inclusion composition. This contradicts   with the fact  that $R$ is a reduced Gr\"{o}bner-Shirshov basis in $\Lbs(B)$. So there is a unique normal $R$-polynomial $h_{1,s_1}$ such that $\ov{h}=\ov{h_{1,s_1}}$.  Moreover, we have $h-\alpha_1 h_{1,s_1}\in\Id(R)$ where $\alpha_1=lc(h)$ and  $\ov{h-\alpha_1h_{1,s_1}}<\ov{h}$.
Therefore the result follows by induction on $\ov{h}$.
\end{proof}

According to Theorem \ref{reduced} and Proposition \ref{greater1}, in this section, we  can  always assume that~$R$ is a reduced \gsb\  in $\Lbs( B)$ and for each~$f\in R$, $\ell(\ov{f})$ is greater than $1$.
Thus, in $\Lbs( B)$, each normal  $R$-polynomial $\hh_s$ has the form  $\hh_s=[s\cc_1\cdots\cc_n]_{_L}$ where $\cc_1,\ldots,\cc_n\in B,\ s\in R$, and $R$ has only left multiplication composition $(\mu\ff)$ where $\mu\in \NF(B),\ f\in R$. By Lemma \ref{le2.2}, $(\mu\ff)$ has a unique expression $(\mu\ff)=(\mu\ff)_{_R}=\sum \alpha_{i}h_{i,s_i}$,  where each $\alpha_i\in\kk$ and $h_{i,s_i}$ is a normal $R$-polynomial. If this is the case, we have $\mathfrak{B}=\mathfrak{B}_0\oplus\mathfrak{B}_1$, where $\mathfrak{B}_0$ (resp. $\mathfrak{B}_1$) is a vector space spanned by the set of all elements in $\Irr(R)$ with parity 0 (resp. 1).
We denote
$$
h_{\lfloor s\rfloor}:=\{\{\{\lfloor s\rfloor \cdot \cc_1\}\cdots\}\cdot\cc_n\},\ \mbox{ where }\ \{x\}=x \ \mbox{ if }\ x\in \mathfrak{A},
$$
$$
(\mu\ff)_{_{\lfloor R\rfloor}}:=\sum\alpha_i h_{i,\lfloor s_i\rfloor},
$$
$$
\Irr(R)=\{\mu\in \NF(B)\mid \mu\neq \ov{\hh_s} \mbox{ for any normal } R\mbox{-polynomial } \hh_s\ \mbox{in }\Lbs(B) \},
$$
$$
\Irr(R_{(\lfloor\  \rfloor, \cdot)})=\{\mu\in \NF(A\cup B)\mid \mu\neq \ov{\hh_r} \mbox{ for any normal } R_{(\lfloor\  \rfloor, \cdot)}\mbox{-polynomial } \hh_r\ \mbox{in}\ \Lbs(A\cup B) \}.
$$

\begin{lemm} \label{le2.3-}
The following statements hold.
\begin{enumerate}
\item[(a)]\ Suppose $\{(bb')\mid b,b'\in B\}\subseteq \ov{R}$. If $\mu\in$ $\NF(B)$ then
\begin{align}\label{mu}
\mu=\sum \alpha_ih_{i,r_i}+\sum \beta_j\aa_j+\sum \gamma_t \bb_t,
\end{align}
where each $\alpha_i, \beta_j,\gamma_t\in\kk$, $\aa_j\in A,\ \bb_t\in B$, $h_{i, r_{i}}$ is a normal $R'_{(\lfloor\  \rfloor, \cdot)}$-polynomial in~$\Lbs(A\cup B)$ and $\deg(\ov{h_{i, r_{i}}})\leq\deg(\mu)$.

\item[(b)]\ If $\mu\in\NF(A\cup B)\sm \NF(B)$ then
\begin{align}\label{equamu}
\mu=\sum \alpha_i h_{i, r_{i}}+\sum \beta_ja_j,
\end{align}
\begin{align}\label{abbb}
[a b_1\cdots b_n]_{_L}=\sum \alpha_i h_{i, r_{i}}+\{\{\{a\cdot b_1\}\cdots\}\cdot\bb_n\},
\end{align}
where each $\alpha_i,\beta_j\in\kk$, $a_j,a\in A,\ b_1,\dots, b_n\in B$, $\deg(\ov{h_{i, r_{i}}})\leq\deg(\mu)$ and in (\ref{equamu}) (resp. (\ref{abbb})) $h_{i, r_{i}}$ is a normal $R_{(\lfloor\  \rfloor, \cdot)}\ (resp.\ R'_{(\lfloor\  \rfloor, \cdot)})$-polynomial in~$\Lbs(A\cup B)$.

Moreover, if $\{(bb')\mid b,b'\in B\}\subseteq \ov{R}$, then in (\ref{equamu})
each $h_{i, r_{i}}$ is a normal $R'_{(\lfloor\  \rfloor, \cdot)}$-polynomial in~$\Lbs(A\cup B)$.

\end{enumerate}

\end{lemm}

\begin{proof} Note that if $\{(bb')\mid b,b'\in B\}\subseteq \ov{R}$, $R$ is a reduced \gsb\ in $\Lbs(B)$ and $f\in R$ with $\bar f=(bb')$, then $f=(bb')+\sum\alpha_i\bb_i$, where each $\alpha_i\in\kk,\ \bb_i\in B$.

We show (a) and (b)  by induction on~$\deg(\mu)$.

If $\deg(\mu)=1$, there is nothing to prove.

Suppose $\deg(\mu)=2$. Then $\mu$ is equal to $(aa')$ or $(ab)$ or $(ba)$ or $(bb')$, where~$a,a'\in A,\ b,b'\in B$. Thus, (a) and (b)  hold, since $(aa')=\Omega_{aa'}+\{a\cdot a'\}$; $(ab)=\Upsilon_{_{ab}}+\{a\cdot b\}$; $(ba)=\Gamma_{ba}+\{b\cdot a\}$; $(bb')=\Theta_{_{f}}-\sum\alpha_i\bb_i+\lfloor f\rfloor$, where $f=(bb')+\sum\alpha_i\bb_i\in R,\ \alpha_i\in\kk,\ \bb_i\in B$.

Suppose $\mu=[c_1c_2\cdots c_n]_{_L},\ c_1,\cdots, c_n\in A\cup B,\ n>2$.

(a) Suppose $\{(bb')\mid b,b'\in B\}\subseteq \ov{R}$ and $\mu\in\NF(B)$, i.e., $c_1,\cdots, c_n\in  B$.
Then by induction hypothesis $\mu=([c_1\cdots c_{n-1}]_{_L}c_n)=\sum \alpha_i(h_{i,r_i}c_n)+\sum \beta_j(\aa_jc_n)+\sum \gamma_t (\bb_t c_n)$, where $\alpha_i, \beta_j,\gamma_t\in\kk$, $\aa_j\in A,\ \bb_t\in B$, $h_{i,r_i}$ is a normal $R'_{(\lfloor\  \rfloor, \cdot)}$-polynomial in $\Lbs(A\cup B)$, and~$\deg(\ov{h_{i,r_i}})\leq\deg([c_1\cdots c_{n-1}]_{_L})$.
Then each $(h_{i,r_i}c_n)$ is a normal $R'_{(\lfloor\  \rfloor, \cdot)}$-polynomial in $\Lbs(A\cup B)$, and~$\deg(\ov{h_{i,r_i}})+\deg(c_n)=\deg(\ov{(h_{i,r_i} c_n)})\leq\deg([c_1\cdots c_{n}]_{_L})$.
Thus by using the above result when $\deg(\mu)=2$, (a) holds.

(b) Suppose $\mu\in\NF(A\cup B)\sm \NF(B)$. Then there are two cases  to consider.

Case 1. $\cc_1=a\in A$. Thus we have the following equation $\mu=[h_rc_3\cdots c_n]_{_L}+[\{a\cdot c_2\}c_3\cdots c_n]_{_L}$, where $h_r=\Omega_{aa'}$ if $c_2=a'\in A$; and $h_r=\Upsilon_{_{ab}}$ if $c_2=b\in B$.
Now $\deg(\ov{[h_rc_3\cdots c_n]_{_L}})=\deg(\ov{h_r})+\sum_{i=3}^n\deg(c_i)=\deg([c_1c_2c_3\cdots c_n]_{_L})$, and
the degree of each element in $\supp([\{a\cdot c_2\}c_{3}\cdots c_n]_{_L})$ is equal to $n-1$. Thus, (\ref{abbb}) and of course (\ref{equamu})  hold by induction.

Case 2. There exists $c_t, 1< t\leq n$ such that $c_1,\cdots, c_{t-1}\in B,\ c_t=a\in A$. Then $\mu=[[c_1\cdots c_{t-1} a]_{_L}c_{t+1}\cdots c_n]_{_L}=[\Gamma_{[c_1\cdots c_{t-1}]_{_L}a}c_{t+1}\cdots c_n]_{_L}+[\{[c_1\cdots c_{t-1}]_{_L}\cdot a\}c_{t+1}\cdots c_n]_{_L}$, where $[\Gamma_{[c_1\cdots c_{t-1}]_{_L}a}c_{t+1}\cdots c_n]_{_L}$ is a normal $R_{(\lfloor\  \rfloor, \cdot)}$-polynomial in $\Lbs(A\cup B)$, and
the degree of each element in
$\supp(\{[c_1\cdots c_{t-1}]_{_L}\cdot a\}c_{t+1}\cdots c_n]_{_L})$ is equal to $n-t+1$. So, (\ref{equamu}) follows by induction.

Now, suppose $\{(bb')\mid b,b'\in B\}\subseteq \ov{R}$. By (a) we may assume that $[c_1\cdots c_{t-1}]_{_L}=\sum \alpha_ih_{i,r_i}+\sum \beta_j\aa_j+\sum \gamma_l \bb_l$, where each $h_{i,r_i}$ is a normal $R'_{(\lfloor\  \rfloor, \cdot)}$-polynomial in $\Lbs(A\cup B)$. Therefore,
$\mu=[[c_1\cdots c_{t-1}]_{_L} a c_{t+1}\cdots c_n]_{_L}=\sum_{i} \alpha_i [h_{i, r_{i}}ac_{t+1}\cdots c_n]_{_L}+\sum_{j} \beta_j[a_jac_{t+1}\cdots c_n]_{_L}+\sum \gamma_l [\Gamma_{b_la}c_{t+1}\cdots c_n]_{_L}+\sum \gamma_l [\{b_l\cdot a\}c_{t+1}\cdots c_n]_{_L}$. By Case 1, both $[a_jac_{t+1}\cdots c_n]_{_L}$ and $[\{b_l\cdot a\}c_{t+1}\cdots c_n]_{_L}$ are linear combinations of normal $R'_{(\lfloor\  \rfloor, \cdot)}$-polynomials and elements in $A$. Thus, (\ref{equamu}) holds.
\end{proof}

\begin{lemm} \label{le2.3}
The set
$R_{(\lfloor\  \rfloor, \cdot)}$ is a Gr\"{o}bner-Shirshov basis in
$\Lbs( A\cup B)$ if and only if the following two  conditions hold in $\mathfrak{A}$:

\ITEM1 For all $a\in A, f\in R$,
 $
a\cdot f=\{ a\cdot\lfloor f\rfloor\}
$
and~$
{f \cdot a=\{\lfloor f\rfloor \cdot a\}}
$.

  \ITEM2 For all $\mu\in \NF(B), f\in R$, and $(\mu\ff)_{_R}=\sum \alpha_{i}h_{i,s_i}$ (by Lemma \ref{le2.2}), we have
  $$
\mu\cdot\lfloor f\rfloor=(\mu\ff)_{_{\lfloor R\rfloor}}.
$$
Moreover, if $R_{(\lfloor\  \rfloor, \cdot)}$ is a Gr\"{o}bner-Shirshov basis, then
$
E_{ (\mathfrak{A}, \mathfrak{B},\lfloor\  \rfloor,\cdot)}
$
is an extension of $\mathfrak{B}$ by~$\mathfrak{A}$.
\end{lemm}

\begin{proof}
 For the first claim, since $R$ is a reduced  Gr\"{o}bner-Shirshov basis in $\Lbs(B)$, the  only possible compositions in
$R_{(\lfloor\ \rfloor,\cdot)}$ are as follows:
$$
(\mu \Omega_{aa'}),\ \ (\mu \Upsilon_{_{ab}}),  \ (\mu \Gamma_{\nu a}), \ \ (\mu \Theta_{_f}), \ \ \mbox{and inclusion  composition}\ (\Gamma_{\nu a}, \Theta_{_f})_{_{\ov{\Gamma_{\nu a}}}},
$$
where $\mu\in \NF(A\cup B),\ a,a'\in A,\ b\in B,\ \nu\in \NF(B),\ f\in R$.

By (\ref{equamu}), we assume $
\mu=\sum \alpha_i h_{i, r_{i}}+\sum \beta_ja_j
$, if $\mu\in\NF(A\cup B)\sm \NF(B)$,
where each $\alpha_i,\beta_j\in\kk$, $a_j\in A$, $\deg(\ov{h_{i, r_{i}}})\leq\deg(\mu)$ and  $h_{i, r_{i}}$ is a normal $R_{(\lfloor\  \rfloor, \cdot)}$-polynomial in~$\Lbs(A\cup B)$.
Therefore, for all $g\in \Lbs(A\cup B)$, we have
\begin{align}\label{mug}
(\mu g)\equiv \sum \beta_j(\aa_j g)\ \ \mod(R_{(\lfloor\  \rfloor, \cdot)}, \deg(\mu )+\deg(\ov{g}))\ \ \mbox{ if } \mu\in\NF(A\cup B)\sm \NF(B).
\end{align}
Assume~$\nu=[b_1\cdots b_n]_{_L},\ b_1,\dots,b_n\in B,\ n\geq 1$. Then by Lemma \ref{multiplication}, it follows that~$(g\nu)=(g[b_1\cdots b_n]_{_L})=\sum\gamma_t [g x_{t_1}\cdots x_{t_n}]_{_L}$ for all~$g\in \Lbs(A\cup B)$, where each ~$\gamma_t\in \kk$, $(\xx_{t_1},\ldots,\xx_{t_{n}})$ is a permutation of~$(\bb_1,\ldots,\bb_n)$.
If~$g\in \mathfrak{A}$, then using (\ref{abbb}), we have
\begin{align}\label{2.00}
(g\nu)=\sum_{t}\gamma_t (\{\{\{g\cdot x_{t1}\}\cdots\}\cdot x_{tn}\}+\sum_m \alpha_{tm} h_{tm, r_{tm}}),
\end{align}
where each $\alpha_{tm}\in\kk,\ r_{tm}\in R_{(\lfloor\  \rfloor, \cdot)}$, $h_{tm, r_{tm}}$ is a normal $R_{(\lfloor\  \rfloor, \cdot)}$-polynomial and $\deg(\ov{h_{tm, r_{tm}}})\leq\deg(\overline{(g\nu)})$.
 According to the fact that $\mathfrak{A}$ is a  compatible $\Lbs( B)$-supermodule,
for all $\mu\in\NF(B),\ b\in B$ and $a,a',a_j\in A$, we have
\begin{align}\label{2.01}
 &\{\{\mu\cdot a\}\cdot a'\}-(-1)^{|a||a'|}\{\{\mu\cdot a'\}\cdot a\}-\{\mu\cdot\{a\cdot a'\}\}=0,
\end{align}
\begin{align}\label{2.02}
 &\{\{a_j\cdot a\}\cdot a'\}-(-1)^{|a||a'|}\{\{a_j\cdot a'\}\cdot a\}-\{a_j\cdot\{a\cdot a'\}\}=0,
 \end{align}
\begin{align}\label{2.1}
 &\{\{\mu\cdot a\}\cdot b\}-(-1)^{|a||b|}\{(\mu b)\cdot a\}-\{\mu\cdot\{a\cdot b\}\}=0,
\end{align}
\begin{align}\label{2.2}
 &\{\{a_j\cdot a\}\cdot b\}-(-1)^{|a||b|}\{\{a_j\cdot b\}\cdot a\}-\{a_j\cdot\{a\cdot b\}\}=0,
\end{align}
and by a similar proof of Lemma \ref{multiplication}, we have
\begin{align}\label{2.3}
 &\{(\mu \nu)\cdot a\}-(-1)^{|\nu||a|}(\sum\gamma_t\{\{\{\{\mu\cdot a\} \cdot x_{t_1}\}\cdots\}\cdot x_{t_n}\})-\{\mu\cdot\{\nu\cdot a\}\}=0,
\end{align}
\begin{align}\label{2.4}
 &\sum_t\gamma_t \{\{\{\{a_j\cdot x_{t1}\}\cdots\}\cdot x_{tn}\}\cdot a\}-(-1)^{|\nu||a|}(\sum_t\gamma_t (\{\{\{\{a_j\cdot a\}\cdot x_{t1}\}\cdots\}\cdot x_{tn}\})-\{a_j\cdot\{\nu\cdot a\}\}=0.
\end{align}

 Now we claim that $
(\mu \Omega_{aa'})$, $(\mu \Upsilon_{_{ab}})$ and~$\ (\mu \Gamma_{\nu a})$ are trivial modulo $R_{(\lfloor\ \rfloor,\cdot)}$.

Case I.
$
(\mu \Omega_{aa'})=((\mu a)a')-(-1)^{|a||a'|}((\mu a')a)-(\mu\{a\cdot a'\}).
$ There are two cases to consider.

(a) If $\mu\in\NF(B)$, then
\begin{align*}
(\mu \Omega_{aa'})
&=(\Gamma_{\mu a}a')+\Omega_{\{\mu\cdot a\}a'}+\{\{\mu\cdot a\}\cdot a'\}
-(-1)^{|a||a'|}((\Gamma_{\mu a'}a )+\Omega_{\{\mu\cdot a'\}a }+\{\{\mu\cdot a'\}\cdot a \})\\
&\ \ \ -\Gamma_{\mu\{a\cdot a'\}}-\{\mu\cdot\{a\cdot a'\}\}\\
&\equiv0\ \mod(R_{(\lfloor\  \rfloor, \cdot)},\deg(\mu )+\deg(\ov{\Omega_{aa'}}))\ \ (by\ (\ref{2.01})).
\end{align*}

(b) If $\mu\in\NF(A\cup B)\sm\NF(B)$, then by (\ref{equamu}),
we have
\begin{align*}
(\mu \Omega_{aa'})&\equiv\sum \beta_j(a_j \Omega_{aa'})\equiv\sum \beta_j(((a_j a)a')-(-1)^{|a||a'|}((a_j a')a)-(a_j\{a\cdot a'\}))\ (by\ \ref{mug})\\
&\equiv\sum \beta_j((\Omega_{a_ja}a')+\Omega_{\{a_j\cdot a\}a'}+\{\{a_j\cdot a\}\cdot a'\}-(-1)^{|a||a'|}(\Omega_{a_ja'}a )-(-1)^{|a||a'|}\Omega_{\{a_j\cdot a'\}a }\\
&\ \ \ \ \ \ \ \ \ \ \ \ \ -(-1)^{|a||a'|}\{\{a_j\cdot a'\}\cdot a\}-\Omega_{a_j\{a\cdot a'\}}-\{a_j\cdot\{a\cdot a'\}\})\\
&\equiv0 \mod(R_{(\lfloor\  \rfloor, \cdot)},\deg(\mu )+\deg(\ov{\Omega_{aa'}})) \ \ (by\ (\ref{2.02})).
\end{align*}

Case II.
$
(\mu \Upsilon_{_{ab}})=((\mu a)b)-(-1)^{|a||b|}((\mu b)a)-(\mu\{a\cdot b\}).
$ We also consider in two cases.

(a) If $\mu\in\NF(B)$, then
\begin{align*}
(\mu \Upsilon_{_{ab}})&=(\Gamma_{\mu a}b)+\Upsilon_{_{\{\mu\cdot a\}b}}+\{\{\mu\cdot a\}\cdot b\}
             -(-1)^{|a||b|}(\Gamma_{(\mu b)a}+\{(\mu b)\cdot a\})-\Gamma_{\mu\{a\cdot b\}}-\{\mu\cdot\{a\cdot b\}\}\\
&\equiv0 \mod(R_{(\lfloor\  \rfloor, \cdot)},\deg(\mu)+\deg(\ov{\Upsilon_{ab}}))\ \ (by\ (\ref{2.1})).
\end{align*}

(b) If $\mu\in\NF(A\cup B)\sm\NF(B)$, then by (\ref{equamu}), we have
\begin{align*}
(\mu \Upsilon_{ab})&\equiv\sum \beta_j(a_j \Upsilon_{ab})\equiv\sum \beta_j(((a_j a)b)-(-1)^{|a||b|}((a_j b)a)-(a_j\{a\cdot b\}))\ \ (by\ (\ref{mug}))\\
&\equiv\sum \beta_j((\Omega_{a_ja}b)+\Upsilon_{_{\{a_j\cdot a\}b}}+\{\{a_j\cdot a\}\cdot b\}-(-1)^{|a||b|}(\Upsilon_{_{a_j b}}a )-(-1)^{|a||b|}\Omega_{\{a_j\cdot b\}a }\\
&\ \ \ \ \ \ \ \ \ \ \ \ \ -(-1)^{|a||b|}\{\{a_j\cdot b\}\cdot a\}-\Omega_{a_j\{a\cdot b\}}-\{a_j\cdot\{a\cdot b\}\})\\
&\equiv0 \mod(R_{(\lfloor\  \rfloor, \cdot)},\deg(\mu)+\deg(\ov{\Upsilon_{ab}}))\ \ (by\ (\ref{2.2})).
\end{align*}

Case III.
$
(\mu \Gamma_{\nu a})=((\mu \nu)a)-(-1)^{|\nu||a|}((\mu a)\nu)-(\mu\{\nu\cdot a\}).
$
Assume~$\nu=[b_1\cdots b_n]_{_L},\ b_1,\dots,b_n\in B,\ 1\leq n$.
We also need to consider this case in two situations.

(a) If $\mu\in\NF(B)$, then
\begin{align*}
(\mu \Gamma_{\nu a})
&=\Gamma_{(\mu \nu) a}+\{(\mu \nu)\cdot a\}-(-1)^{|\nu||a|}((\Gamma_{\mu a}\nu)+
(\{\mu\cdot a\} \nu)) -\Gamma_{\mu\{\nu\cdot a\}}-\{\mu\cdot\{\nu\cdot a\}\}\\
&=\Gamma_{(\mu \nu) a}+\{(\mu \nu)\cdot a\}-(-1)^{|\nu||a|}((\Gamma_{\mu a}\nu)+\sum\gamma_t
(\{\{\{\{\mu\cdot a\} \cdot x_{t_1}\}\cdots\}\cdot x_{t_n}\}+\sum_m \alpha_{tm} h_{tm,r_{tm}}))\\
&\ \ \ -\Gamma_{\mu\{\nu\cdot a\}}-\{\mu\cdot\{\nu\cdot a\}\}\ \ (by\ (\ref{2.00}))\\
&\equiv0 \mod(R_{(\lfloor\  \rfloor, \cdot)},\deg(\mu)+\deg(\ov{\Gamma_{\nu a}}))\ \ (by\ (\ref{2.3})).
\end{align*}

(b) If $\mu\in\NF(A\cup B)\sm\NF(B)$, then by (\ref{equamu}), we have
\begin{align*}
(\mu \Gamma_{\nu a})&\equiv\sum_j \beta_j(a_j \Gamma_{\nu a})\equiv\sum_j \beta_j(((a_j \nu)a)-(-1)^{|\nu||a|}((a_j a)\nu)-(a_j\{\nu\cdot a\}))\ \ \ (by\ (\ref{mug}))\\
&\equiv\sum_j \beta_j((\sum\gamma_t (\{\{\{a_j\cdot x_{t1}\}\cdots\}\cdot x_{tn}\}+\sum_m \alpha_{jtm} h_{jtm, r_{jtm}})a)\\
&\ \ \ -(-1)^{|\nu||a|}(\Omega_{a_j a}\nu)-(-1)^{|\nu||a|}(\sum_t\gamma_t (\{\{\{\{a_j\cdot a\}\cdot x_{t1}\}\cdots\}\cdot x_{tn}\}+\sum_m \alpha'_{jtm} h'_{jtm, r'_{jtm}}))\\
&\ \ \ -\Omega_{a_j\{\nu\cdot a\}}-\{a_j\cdot \{\nu\cdot a\}\})\ \ \ (by\ (\ref{2.00}))\\
&\equiv\sum_j \beta_j(\sum_t\gamma_t (\Omega_{_{ \{\{\{a_j\cdot x_{t1}\}\cdots\}\cdot x_{tn}\}a}}+ \{\{\{a_j\cdot x_{t1}\}\cdots\}\cdot x_{tn}\}\cdot a\})\\
&\ \ \ \ \ \ \ \ \ \ \ \ -(-1)^{|\nu||a|}\sum_t\gamma_t \{\{\{\{a_j\cdot a\}\cdot x_{t1}\}\cdots\}\cdot x_{tn}\}-\{a_j\cdot \{\nu\cdot a\}\})\\
&\equiv 0 \mod(R_{(\lfloor\  \rfloor, \cdot)},\deg(\mu)+\deg(\ov{\Gamma_{\nu a}}))\ \ (by\ (\ref{2.4})).
\end{align*}

Then $(\mu \Omega_{aa'})$, $(\mu \Upsilon_{_{ab}})$ and~$\ (\mu \Gamma_{\nu a})$ are trivial modulo $R_{(\lfloor\ \rfloor,\cdot)}$.

Case IV. Now we consider the composition $(\mu \Theta_{_f})$ in two cases.

(a) Suppose $\mu$ lies in~$\NF(B)$. For all $f\in R$, by Lemma \ref{le2.2}, we have
$
(\mu\ff)=(\mu\ff)_{_R}=\sum \alpha_{i}h_{i,s_i},
 $
 where each $h_{i,s_i}=[s_i b_{i1}\cdots b_{in_i}]_{_L}$ is a normal $R$-polynomial in $\Lbs(B)$, $s_i\in R$, $\alpha_{i}\in \kk,\ b_{i1},\dots, b_{in_i}\in B$ and $\deg(\overline{h_{i,s_i}})\leq\deg(\mu)+\deg(\bar{\ff})$.
Moreover, we have the following equation in $\Lbs(A\cup B)$
\begin{align*}
h_{i,s_i}&=[\Theta_{s_i} b_{i1}\cdots b_{in_i}]_{_L}+[\lfloor s_i \rfloor b_{i1}\cdots b_{in_i}]_{_L}\\
&\equiv[\Upsilon_{_{\lfloor s_i \rfloor b_{i1}}}b_{i2}\cdots b_{in_i}]_{_L}+\cdots+\Upsilon_{_{\{\{\{\lfloor s_i \rfloor\cdot b_{i1}\}\cdots\}\cdot b_{i(n_i-1)}\}b_{in_i}}} +\{\{\{\lfloor s_i \rfloor\cdot b_{i1}\}\cdots\}\cdot b_{in_i}\}\\
&\equiv h_{i,\lfloor s_i \rfloor}\mod(R_{(\lfloor\ \rfloor,\cdot)},\deg(\ov{h_{i,s_i}})),
\end{align*}
where
$\deg(\ov{h_{i,s_i}})\leq\deg(\mu)+\deg(\ov{f})=\deg(\mu)+\deg(\ov{\Theta_{_f}})$.
Thus
\begin{align*}
(\mu \Theta_{_f})&= (\mu f)-(\mu \lfloor  f \rfloor)
\equiv \sum \alpha_{i}h_{i,\lfloor s_i \rfloor}-\Gamma_{\mu \lfloor  f \rfloor}-\{\mu\cdot \lfloor  f \rfloor\} \mod( R_{(\lfloor\ \rfloor,\cdot)}, \deg(\mu)+\deg(\ov{\Theta_{_f}})),
\end{align*}
where $\Gamma_{\mu \lfloor  f \rfloor}$ is a normal~$R_{(\lfloor\ \rfloor,\cdot)}$-polynomial and its degree $\leq\deg(\mu)+\deg(\ov{\Theta_{_f}})$. So we know that for~$\mu\in \NF(B)$, $(\mu  \Theta_{_f})\equiv 0 \mod (R_{(\lfloor\ \rfloor,\cdot)}, \deg(\mu)+\deg(\ov{\Theta_{_f}}))$ if and only if, in $\mathfrak{A}$,
$$
\mu\cdot \lfloor  f \rfloor=\sum \alpha_{i}h_{i,\lfloor s_i \rfloor}.
$$
(b) If $\mu\in \NF(A\cup B)\sm \NF(B)$, then by (\ref{equamu}), it follows that
\begin{align*}
(\mu \Theta_{_f})&\equiv\sum \beta_j(a_j\Theta_{_f})\equiv \sum \beta_j((a_jf)-(a_j \lfloor  f \rfloor))\ \mod( R_{(\lfloor\ \rfloor,\cdot)}, \deg(\mu)+\deg(\ov{\Theta_{_f}}))\ (by\ (\ref{mug})).
\end{align*}
So we only consider~$(a f)$, where $a\in A$, and~$f=\sum_j\lambda_j [b_{j1}\cdots b_{jn_j}]_{_L}\in R,$ where $\lambda_j\in k$, $b_{j1},\dots, b_{jn_j}\in B$.
By a similar proof of Lemma \ref{multiplication}, we have
$$a\cdot f= \sum_{j,t}\lambda_j \beta_{tj}\{\{\{a\cdot x_{tj_1}\}\cdots\}\cdot x_{tj_{n_j}}\}.$$
According to (\ref{2.00}), we have
\begin{align*}
(af)&=\sum_j\lambda_j (a[b_{j_1}\cdots b_{j_{n_j}}]_{_L})=\sum_{j,t}\lambda_j \beta_{tj}(\{\{\{a\cdot x_{tj_1}\}\cdots\}\cdot x_{tj_{n_j}}\}+\sum_m \alpha_{tjm} h_{tjm, r_{tjm}}),
\end{align*}
where $(x_{tj_1},\dots, x_{tj_{n_j}})$ is a permutation of $(b_{j_1},\dots, b_{j_{n_j}})$, each $\beta_{tj},\alpha_{tjm}\in\kk$, $h_{tjm, r_{tjm}}$ is a normal $R_{(\lfloor\  \rfloor, \cdot)}$-polynomial and $\deg(\ov{h_{tjm, r_{tjm}}})\leq\deg([a x_{j1}\cdots x_{jn_j}]_{_L})$. It follows that
\begin{align*}
(a \Theta_{_f})&=(a f)-(a\lfloor  f \rfloor)\\
&=\sum_{j,t}\lambda_j \beta_{tj}(\{\{\{a\cdot x_{tj_1}\}\cdots\}\cdot x_{tj_{n_j}}\}+\sum_m \alpha_{tjm} h_{tjm, r_{tjm}})-\Omega_{a \lfloor  f \rfloor}-\{a \cdot\lfloor  f \rfloor\}\\
&\equiv a\cdot f-\{a \cdot\lfloor  f \rfloor\}\mod(R_{(\lfloor\  \rfloor, \cdot)},\deg(\aa)+\deg(\ov{\Theta_{_f}})).
\end{align*}
Thus for any~$\mu\in \NF(A\cup B)\sm \NF(B)$, $(\mu \Theta_{_f})$ is trivial modulo $(R_{(\lfloor\ \rfloor,\cdot)}, \deg(\mu)+\deg(\ov{\Theta_{_f}}))$ if and only if for any $a\in A$,
$
a\cdot f=\{a\cdot \lfloor  f \rfloor\}.
$

Next we think about the inclusion composition. Let $h_f=[f b_1b_2\cdots b_n]_{_L}$ be an arbitrary normal $R$-polynomial in $\Lbs(B)$, where $b_1,\cdots, b_n\in B, n\geq 0$ and $f=\bar f+r_{_f}\in R$. Then we have the inclusion composition $(\Gamma_{\nu a},\Theta_{_f})_{_{\ov{\Gamma_{\nu a}}}}$, where $\nu=[\bar{f}b_1b_2\cdots b_n]_{_L}\in \NF(B),\ a\in A$.
\begin{align*}
(\Gamma_{\nu a},\Theta_{_f})_{_{\ov{\Gamma_{\nu a}}}}&=\Gamma_{\nu a}-[\Theta_{_f}b_1\cdots b_n a]_{_L}\\
&=-\{\nu\cdot a\}-\Theta_{_{[r_{_f} b_1\cdots b_n]_{_L}a}}-\{[r_{_f} b_1\cdots b_n]_{_L}\cdot a\}+[\lfloor  f \rfloor b_1\cdots b_n a]_{_L}\\
&\equiv -[f b_1\cdots b_n]_{_L}\cdot a +\{\{\{\{\lfloor  f \rfloor\cdot b_1\}\cdot \cdots\}\cdot b_n\} \cdot a\}\\
&\equiv -h_f\cdot a+\{h_{\lfloor  f \rfloor}\cdot a\}\mod(R_{(\lfloor\ \rfloor,\cdot)}, (\nu a)).
\end{align*}
Assume that $f\cdot a=\{\lfloor  f \rfloor\cdot a\}$ holds, for each~$f\in R$ and  $a\in A$.
Then we have
$
[f \bb_1]_{_L}\cdot a
=f\cdot \{\bb_1\cdot a\}+(-1)^{|a||\bb_1|}\{\{f\cdot a\}\cdot \bb_1\}
=\{\lfloor  f \rfloor\cdot \{\bb_1\cdot a \}\}
+(-1)^{|a||\bb_1|}\{\{\lfloor  f \rfloor\cdot a\}\cdot \bb_1\}
=\{\{f\cdot \bb_1\}\cdot a\}.
$
From this, we get
\begin{align*}
&\ \ \ \ h_f\cdot a=[f \bb_1\bb_2\cdots\bb_n]_{_L}\cdot a\\
&=[f \bb_1\bb_2\cdots\bb_{n-1}]_{_L}\cdot \{\bb_n\cdot a\}+(-1)^{|a||\bb_n|}\{\{[f \bb_1\bb_2\cdots\bb_{n-1}]_{_L}\cdot a\}\cdot \bb_n\}\\
&= \cdots\\
&=\{\{\{\{\lfloor  f \rfloor\cdot\bb_1\}\cdots\}\cdot\bb_{n-1}\}\cdot \{\bb_n\cdot a \}\}
+(-1)^{|a||\bb_n|}\{\{\{\{\{\lfloor  f \rfloor\cdot\bb_1\}\cdots\}\cdot\bb_{n-1}\}\cdot a\}\cdot \bb_n\}\\
&=\{\{\{\{\{\{\lfloor  f \rfloor\cdot\bb_1\}\cdot\bb_2\}\cdots\}\cdot\bb_{n-1}\}\cdot \bb_n\}\cdot a\}=\{h_{\lfloor  f \rfloor}\cdot a\}.
\end{align*}

Thus all possible inclusion compositions are trivial if and only if
for each~$f\in R$ and  $a\in A$, $f\cdot a=\{\lfloor  f \rfloor\cdot a\}$ holds.

Therefore, we can see that $R_{(\lfloor\  \rfloor, \cdot)}$ is a
Gr\"{o}bner-Shirshov basis in $\Lbs( A\cup B)$ if and
only if the conditions  \ITEM1 and \ITEM2 hold.

We now turn to the second  claim. If $R_{(\lfloor\  \rfloor, \cdot)}$ is a Gr\"{o}bner-Shirshov basis,
then by Theorem \ref{cd-lemma for Lb}, we
have that $\Irr(R_{(\lfloor\  \rfloor, \cdot)})=A\cup \Irr(R)$ is a linear
basis of $E_{(\mathfrak{A}, \mathfrak{B},\lfloor\  \rfloor,\cdot)}$,
where $ \Irr(R)=\{ \nu\in \NF(B)\mid \nu\neq h_{\bar{g}}$ for any normal $R$-polynomial~$h_g$ in $\Lbs(B) \}. $
Thus~$E_{(\mathfrak{A}, \mathfrak{B},\lfloor\  \rfloor,\cdot)}=
 \mathfrak{A}\oplus \mathfrak{B}$ as a vector  space, where~$\mathfrak{B}$ is spanned by~$\Irr(R)$. Because $R$ is a \gsb\ in $\Lbs(B)$,  for all~$\nu,\nu'\in \Irr(R)$, by Lemmas \ref{f=Irr+n-s-polynomials}, \ref{le2.2} and Theorem \ref{cd-lemma for Lb}, in
$\Lbs( B)$, we have unique $\delta_i, \beta_j\in k$, $\nu_i\in \Irr(R)$ and unique normal $R$-polynomial $h_{j, g_j}$  such that
 $$(\nu\nu')=\sum \delta_i\nu_i+\sum \beta_j h_{j, g_j}.$$
  Thus the multiplication in~$E_{(\mathfrak{A}, \mathfrak{B},\lfloor\  \rfloor,\cdot)}$  is defined as below: for all
$a,a'\in A$ and~$\nu,\nu'\in \Irr(R)$,
\begin{align}\label{5.6multi}
((a+\nu)(a'+\nu'))=\{a\cdot a'\}+\{a\cdot \nu'\}+\{\nu\cdot a'\}+\sum \delta_i\nu_i+\sum \beta_j h_{j,\lfloor g_j\rfloor}.
\end{align}

It is easy to see that
$\mathfrak{A}$ is an ideal of $E_{(\mathfrak{A}, \mathfrak{B},\lfloor\
\rfloor,\cdot)}$. Define
$$
\pi : E_{(\mathfrak{A}, \mathfrak{B},\lfloor\
\rfloor,\cdot)}= \mathfrak{A}\oplus \mathfrak{B}\rightarrow \mathfrak{B},\ x+y\mapsto y,\ x\in \mathfrak{A},\ y\in \mathfrak{B}.
$$
 Obviously~$\pi$ is a surjective linear map, and we have $\pi(a)=0\in\mathfrak{B}_0\cap\mathfrak{B}_1$, $\pi(\nu)=\nu\in \Irr(R)$,
 $$\pi(((a+\nu)(a'+\nu')))=\sum \delta_i\nu_i=(\nu\nu')=(\pi(a+\nu)\pi(a'+\nu')).$$
Thus $\pi$ is a homomorphism. Moreover, we get $\ker\pi=\mathfrak{A}$. Therefore, $E_{ (\mathfrak{A}, \mathfrak{B},\lfloor\  \rfloor,\cdot)}$
is an extension of $\mathfrak{B}$ by $\mathfrak{A}$.
\end{proof}

The following lemma shows that in Lemma \ref{le2.3}, we can not replace $R_{(\lfloor\  \rfloor, \cdot)}$ with $R'_{(\lfloor\  \rfloor, \cdot)}$.

\begin{lemm}\label{lemma4.16}
If $\{(bb')\mid b,b'\in B\}\subseteq \ov{R}$, then
$R'_{(\lfloor\  \rfloor, \cdot)}$ is a Gr\"{o}bner-Shirshov basis in
$\Lbs( A\cup B)$ if and only if the condition $(i)$ in Lemma \ref{le2.3} holds, and
$(ii)'$ for all $b\in B, f\in R$, and $(b\ff)_{_R}=\sum \alpha_{i}h_{i,s_i}$ (by Lemma \ref{le2.2}), we have, in $\mathfrak{A}$,
  $$
b\cdot\lfloor f\rfloor=(b\ff)_{_{\lfloor R\rfloor}}.
$$
Moreover, if $\{(bb')\mid b,b'\in B\}\nsubseteq \ov{R}$, then $R'_{(\lfloor\  \rfloor, \cdot)}$ is not a Gr\"{o}bner-Shirshov basis in $\Lbs(A\cup B)$.
\end{lemm}

\begin{proof}
Since $R$ is a reduced  Gr\"{o}bner-Shirshov basis in $\Lbs(B)$, the  only possible compositions in
$R'_{(\lfloor\ \rfloor,\cdot)}$ are as follows:
$$
(\mu \Omega_{aa'}),\ \ (\mu \Upsilon_{_{ab}}),  \ (\mu \Gamma_{b a}), \ \ (\mu \Theta_{_f}),
$$
where $\mu\in \NF(A\cup B),\ a,a'\in A,\ b\in B,\ f\in R$.

According to Lemma \ref{le2.3-} (b) and the proof of Case I(b)--III(b), Case IV in Lemma \ref{le2.3}, we know that for all $a,a'\in A,\ b\in B,\ f\in R,\ \mu \in \NF(A\cup B)\backslash\NF(B)$ and normal $R$-polynomial $h_{s}$ in $\Lbs(B)$, $s\in R$,
\begin{align}\label{5.18}
(\mu \Omega_{aa'})\equiv0\mod(R'_{(\lfloor\  \rfloor, \cdot)},\deg(\mu)+\deg(\ov{ \Omega_{aa'}}))\ \ \ \ \mbox{(by Case I (b))};
\end{align}
\begin{align}\label{5.19}
(\mu \Upsilon_{_{ab}})\equiv0\mod(R'_{(\lfloor\  \rfloor, \cdot)},\deg(\mu)+\deg(\ov{ \Upsilon_{_{ab}}}))\ \ \ \ \mbox{(by Case II (b))};
\end{align}
\begin{align}\label{5.20}
(\mu \Gamma_{b a})\equiv0\mod(R'_{(\lfloor\  \rfloor, \cdot)},\deg(\mu)+\deg(\ov{ \Gamma_{b a}}))\ \ \ \ \mbox{(by Case III (b))};
\end{align}
\begin{align}\label{5.21}
(a \Theta_{_f})&\equiv a\cdot f-\{a \cdot\lfloor  f \rfloor\}\mod(R'_{(\lfloor\  \rfloor, \cdot)},\deg(\aa)+\deg(\ov{\Theta_{_f}}))\ \ \ \ \mbox{(by Case IV (b))};
\end{align}
\begin{align}\label{5.22}
h_{s}&\equiv h_{\lfloor s \rfloor}\mod(R'_{(\lfloor\  \rfloor, \cdot)},\deg(\ov{h_{s}}))\ \ \ \ \mbox{(by Case IV (a))};
\end{align}
while
\begin{align}\label{5.2-}
(\mu \Theta_{_f})\equiv0\mod(R'_{(\lfloor\  \rfloor, \cdot)},\deg(\mu)+\deg(\ov{ \Theta_{_f}}))
\ \mbox{holds}
\end{align}
if and only if $a\cdot f=\{ a\cdot\lfloor f\rfloor\}$,
 for all $a\in A$ holds in $\mathfrak{A}$, by Case IV (b) in Lemma \ref{le2.3}.

Now we turn to consider left multiplication compositions when $\mu\in\NF(B)$. If $\mu\in\NF(B)$, then by (\ref{mu}), we have $\mu=\sum \alpha_ih_{i,r_i}+\sum \beta_j\aa_j+\sum \gamma_t \bb_t$, where $\alpha_i, \beta_j,\gamma_t\in\kk$, $\aa_j\in A,\bb_t\in B$, $h_{i,r_i}$ is a normal $R'_{(\lfloor\  \rfloor, \cdot)}$-polynomial in $\Lbs(A\cup B)$ and $\deg(\ov{h_{i,r_i}})\leq\deg(\mu)$. Therefore, for all $g\in \Lbs(A\cup B)$, we have
\begin{align}\label{5.17}
(\mu g)\equiv \sum \beta_j(\aa_j g)+\sum \gamma_t (\bb_t g)\ \ \mod(R'_{(\lfloor\  \rfloor, \cdot)}, \deg(\mu )+\deg(\ov{g})).
\end{align}
\begin{align*}
(\mu \Omega_{aa'})&\equiv\sum \gamma_t (\bb_t \Omega_{aa'})\ \ \ \ (by\ (\ref{5.17}), (\ref{5.18})) \\
&\equiv\sum \gamma_t (((\bb_t a)a')-(-1)^{|a||a'|}((\bb_t a')a)-(\bb_t\{a\cdot a'\}))\\
&\equiv\sum \gamma_t ((\Gamma_{b_t a}a')+\Omega_{\{b_t\cdot a\}a'}+\{\{b_t\cdot a\}\cdot a'\}
-(-1)^{|a||a'|}((\Gamma_{b_t a'}a )+\Omega_{\{b_t\cdot a'\}a}\\
&\ \ \ +\{\{b_t\cdot a'\}\cdot a \})
-\Gamma_{b_t\{a\cdot a'\}}-\{b_t\cdot\{a\cdot a'\}\})\\
&\equiv\sum \gamma_t (\{\{b_t\cdot a\}\cdot a'\}-(-1)^{|a||a'|}\{\{b_t\cdot a'\}\cdot a\}-\{b_t\cdot\{a\cdot a'\}\})\\
&\equiv 0\ \  \mod(R'_{(\lfloor\  \rfloor, \cdot)}, \deg(\mu )+\deg(\ov{\Omega_{aa'}}))\ \ \ \ (by\ (\ref{2.01})).
\end{align*}

For $b\in B$, suppose ${f_t}=\ov{{f_t}}+\sum_q\lambda_{tq} b_{tq}\in R$, where $\ov{{f_t}}=(b_tb),\ \lambda_{tq}\in\kk,\ b_{tq}\in B$. Then we have $(b_tb)=\Theta_{f_t}-\sum\lambda_{tq} b_{tq}+\lfloor {f_t}\rfloor$.
\begin{align*}
(\mu \Upsilon_{_{ab}})&\equiv\sum \gamma_t (\bb_t \Upsilon_{_{ab}})\ \ \ \ (by\ (\ref{5.17}), (\ref{5.19})) \\
&\equiv\sum \gamma_t(((b_t a)b)-(-1)^{|a||b|}((b_t b)a)-(b_t\{a\cdot b\}))\\
&\equiv\sum_t \gamma_t ((\Gamma_{b_t a}b)+\Upsilon_{_{\{b_t\cdot a\}b}}+\{\{b_t\cdot a\}\cdot b\} -(-1)^{|a||b|}((\Theta_{{f_t}}a)\\
&\ \ \ \ \ \ \ \ \ \ \ -\sum_q\lambda_{tq} (\Gamma_{b_{tq}a}+\{b_{tq}\cdot a\})+\Omega_{\lfloor {f_t}\rfloor a}+\{\lfloor {f_t}\rfloor\cdot a\})-\Gamma_{b_t\{a\cdot b\}}-\{b_t\cdot\{a\cdot b\}\})\\
&\equiv\sum_t \gamma_t (\{\{b_t\cdot a\}\cdot b\}+(-1)^{|a||b|}\sum_q\lambda_{tq} \{b_{tq}\cdot a\}-(-1)^{|a||b|}\{\lfloor {f_t}\rfloor\cdot a\}-\{b_t\cdot\{a\cdot b\}\})\\
&\equiv(-1)^{|a||b|}\sum_t\gamma_t(\{(b_tb)\cdot a\}+\sum_q\lambda_{tq} \{b_{tq}\cdot a\}-\{\lfloor f_t\rfloor\cdot a\})\ \ \ (by\ (\ref{2.1}))\\
&\equiv\sum\gamma_t(f_t\cdot a-\{\lfloor f_t\rfloor\cdot a\})\ \ \mod(R'_{(\lfloor\  \rfloor, \cdot)},\deg(\mu)+\deg(\ov{\Upsilon_{ab}})).
\end{align*}
\begin{align*}
(\mu \Gamma_{\bb a})&\equiv\sum \gamma_t (\bb_t \Gamma_{\bb a})\ \ \ \ (by\ (\ref{5.17}), (\ref{5.20})) \\
&\equiv\sum \gamma_t(((b_t \bb)a)-(-1)^{|\bb||a|}((b_t a)\bb)-(b_t\{\bb\cdot a\}))\\
&\equiv\sum_t\gamma_t((\Theta_{{f_t}}a)-\sum_q\lambda_{tq} (\Gamma_{b_{tq}a}+\{b_{tq}\cdot a\})+\Omega_{\lfloor {f_t}\rfloor a}+\{\lfloor {f_t}\rfloor\cdot a\}\\
&\ \ \ \ \ \ \ \ \ \ \ \ \ -(-1)^{|\bb||a|}((\Gamma_{b_t a}b)+\Upsilon_{_{\{b_t\cdot a\}b}}+\{\{b_t\cdot a\}\cdot b\})-\Gamma_{b_t\{\bb\cdot a\}}-\{b_t\cdot\{\bb\cdot a\}\})\\
&\equiv\sum_t\gamma_t(-\sum_q\lambda_{tq} \{b_{tq}\cdot a\}+\{\lfloor f_t\rfloor\cdot a\}-(-1)^{|\bb||a|}\{\{b_t\cdot a\}\cdot b\}-\{b_t\cdot\{\bb\cdot a\}\})\\
&\equiv\sum_t\gamma_t(-\{(b_tb)\cdot a\}-\sum_q\lambda_{tq} \{b_{tq}\cdot a\}+\{\lfloor f_t\rfloor\cdot a\})\ \ \ (by\ (\ref{2.3}))\\
&\equiv\sum\gamma_t(-f_t\cdot a+\{\lfloor f_t\rfloor\cdot a\})\ \mod(R'_{(\lfloor\  \rfloor, \cdot)},\deg(\mu)+\deg(\ov{\Gamma_{ba}})).
\end{align*}
Then for all $\mu\in\NF(B),\ f\in R$,  $(\mu \Upsilon_{_{ab}})$ and~$\ (\mu \Gamma_{b a})$ are trivial modulo $R'_{(\lfloor\ \rfloor,\cdot)}$ if and only if
${f \cdot a=\{\lfloor f\rfloor \cdot a\}}$ holds in $\mathfrak{A}$ for all $a\in A, f\in R$.

For all $b_t\in B,\ f\in R$, by Lemma \ref{le2.2}, we have
$
(\bb_t\ff)=(\bb_t\ff)_{_R}=\sum \alpha_{ti}h_{ti,s_{ti}},
 $
 where each $h_{ti,s_{ti}}$ is a normal $R$-polynomial in $\Lbs(B)$, $s_{ti}\in R$, $\alpha_{ti}\in \kk$ and $\deg(\overline{h_{ti,s_{ti}}})\leq\deg(b_t)+\deg(\bar{\ff})\leq\deg(\mu)+\deg(\ov{f})=\deg(\mu)+\deg(\ov{\Theta_{_{f}}})$.
\begin{align*}
(\mu \Theta_{_{f}})&\equiv\sum \gamma_t (\bb_t \Theta_{_{f}})+\sum \beta_j (\aa_j  \Theta_{_{f}})\ \ (by\ (\ref{5.17}))\\
&\equiv\sum \gamma_t (\sum \alpha_{ti}h_{ti,s_{ti}}-(\bb_t \lfloor f\rfloor))+\sum \beta_j (\{\aa_j\cdot f\}- \{\aa_j\cdot \lfloor f\rfloor\})\ \ (by\ (\ref{5.21})) \\
&\equiv\sum \gamma_t (\sum \alpha_{ti}h_{ti,\lfloor s_{ti}\rfloor}-\{\bb_t \cdot \lfloor f\rfloor\})+\sum \beta_j(\{\aa_j\cdot f\}- \{\aa_j\cdot \lfloor f\rfloor\})\\
&\ \ \ \mod(R'_{(\lfloor\  \rfloor, \cdot)},\deg(\mu)+\deg(\ov{\Theta_{_{f}}}))\ \ (by\ (\ref{5.22})).
\end{align*}
Therefore, by using (\ref{5.2-}), for all $\mu\in\NF(A\cup B)$, $(\mu \Theta_{_f})$ is trivial modulo $(R'_{(\lfloor\  \rfloor, \cdot)},\deg(\mu)+\deg(\ov{ \Theta_{_f}}))$ if and only if
$a\cdot f=\{ a\cdot\lfloor f\rfloor\}$ holds in $\mathfrak{A}$, for all $a\in A, f\in R$ and condition $(ii)'$ holds.

So, we prove that $R'_{(\lfloor\  \rfloor, \cdot)}$ is a Gr\"{o}bner-Shirshov basis in
$\Lbs( A\cup B)$ if and only if $(i)$ and $(ii)'$ hold.

Finally we consider the second  claim. Assume $R'_{(\lfloor\  \rfloor,
\cdot)}$ is a \gsb\ in $\Lbs(A\cup B)$, and there exist $b_1, b_2\in B$, such that $(b_1 b_2)\notin \ov{R}$. Note that  $(b_1\Gamma_{b_2a})\in \Id(R'_{(\lfloor\  \rfloor,
\cdot)})$ and~$\ov{(b_1\Gamma_{b_2a})}=((b_1b_2)a)\in \Irr(R'_{(\lfloor\  \rfloor,
\cdot)})$. This is a contradiction with Theorem \ref{cd-lemma for Lb}.
\end{proof}


On the other hand, let  $0\rightarrow \mathfrak{A}  \rightarrow \mathfrak{E}\rightarrow \mathfrak{B} \rightarrow 0$
 be  a short exact sequence. We shall construct a  compatible $\Lbs(B)$-supermodule $\mathfrak{A}$  and a factor set $\lfloor \ \rfloor$ of $\mathfrak{B}$ in $\mathfrak{A}$ such that $E_{(\mathfrak{A}, \mathfrak{B},\lfloor\
\rfloor,\cdot)}\cong \mathfrak{E}$.

\begin{lemm}\label{le2.4}
Let  $0\rightarrow \mathfrak{A}  \overset{i}\rightarrow \mathfrak{E}\overset{\pi}\rightarrow \mathfrak{B} \rightarrow 0$
 be  a short exact sequence, where $i$ is the inclusion map. For each~$b\in B$, choose an element~$\tilde{b}\in \mathfrak{E}$ with $|\tilde{b}|=|b|$
such that $\pi(\tilde{b})=b+\Id(R) $ in $\mathfrak{B}=\Lbs(B)/\Id(R)$. Let~$\widetilde{\theta}:\Lbs( A\cup B)\rightarrow \mathfrak{E}$
 be the unique homomorphism from $\Lbs( A\cup B)$ to $\mathfrak{E}$
such that~$\widetilde{\theta}(b)=\tilde{b}$ and $\widetilde{\theta}(a)=a$ for all~$b\in B$ and $a\in A$. Define $\Lbs(B)$-supermodule operation $\cdot$ on $\mathfrak{A}$ as follows:
$$
g\cdot a=(\tilde{\theta}(g)a),  \   \ a\cdot
g=(a\tilde{\theta}(g)),\ a\in  A,\ g\in \Lbs(B)
$$
and define the map $
\lfloor \ \rfloor: R\rightarrow  \mathfrak{A},\ f\mapsto \lfloor f\rfloor=
\widetilde{\theta}(f)$.

Then  $\mathfrak{A}$ is a    compatible $\Lbs(B)$-supermodule and $\lfloor \ \rfloor$ is a   factor set of $\mathfrak{B}$ in $\mathfrak{A}$ such that the conditions \ITEM1 and \ITEM2 in Lemma \ref{le2.3} hold.
Moreover,~$E_{(\mathfrak{A}, \mathfrak{B},\lfloor\
\rfloor,\cdot)}= \Lbs( A\cup B|R_{(\lfloor\  \rfloor,
\cdot)})\cong
\mathfrak{E}$.

\end{lemm}

\begin{proof}
 Since $\mathfrak{A}$ is an ideal of $\mathfrak{E}$, we have
$(\tilde{\theta}(g)a),(a\tilde{\theta}(g))\in \mathfrak{A}$ for all~$g\in \Lbs(B)$ and~$ a\in  A$.  Thus, the operation is well defined.  It is easy to prove that~$\mathfrak{A}$ is a compatible $\Lbs(B)$-supermodule.
For each~$f\in R$, since $\pi(\widetilde{\theta}(f))=f +\Id(R)=0$ in $\mathfrak{B}$,  we have
$\widetilde{\theta}(f)\in \mathfrak{A}$. Thus,  $\lfloor \ \rfloor$ is well defined and $|f|=|\lfloor f\rfloor|$, i.e., $\lfloor \ \rfloor$ is a factor set of $\mathfrak{B}$ in $\mathfrak{A}$. Assume that~$h_s=[s\bb_1\cdots\bb_n]_{_L}$ is a normal~$R$-polynomial in $\Lbs(B)$, where~$s\in R,\ \bb_1,\dots,\bb_n\in B$. Then $\tilde{\theta}(h_s)=\{\{\{\lfloor s \rfloor\cdot \bb_1\}\cdot \cdots\}\cdot b_n\}=h_{\lfloor s \rfloor}$.

We now show that the conditions \ITEM1 and \ITEM2 in Lemma \ref{le2.3} hold.
  For all~$a\in A, f\in R$, we have~$
 a\cdot f=(a\tilde{\theta}(f))=\{ a\cdot \lfloor f\rfloor\}$, and~$
 f\cdot a=(\tilde{\theta}(f)a)=\{ {\lfloor f\rfloor}\cdot a\}$.  So \ITEM1 holds.
As for~\ITEM2, for all~$\mu\in \NF(B), f\in R$, by Lemma \ref{le2.2},
$(\mu\ff)=(\mu\ff)_{_R}=\sum \alpha_{i}h_{i,s_i}$,
 where each $h_{i,s_i}$ is a normal $R$-polynomial, $s_i\in R$, $\alpha_{i}\in k$ and $\deg(\overline{h_{i,s_i}})\leq \deg(\mu)+\deg(\bar{\ff})$.
So
$$
0 =  \tilde{\theta}((\mu\ff)-\sum \alpha_{i}h_{i,s_i})\\
 =    (\tilde{\theta}(\mu) \tilde{\theta}(f))-\sum \alpha_{i}\tilde{\theta}(h_{i,s_i})\\
 =   \mu\cdot \lfloor  f \rfloor-\sum \alpha_{i}h_{i,\lfloor s_i  \rfloor}.
$$
This shows (ii) holds.

Now we turn to prove that~$
\mathfrak{E}\cong E_{(\mathfrak{A}, \mathfrak{B},\lfloor\
\rfloor,\cdot)}$.
 Because (i) and (ii) hold, by Lemma \ref{le2.3}, we have $R_{(\lfloor\  \rfloor,
\cdot)}$ is a \gsb\ in $\Lbs( A\cup B)$. Thus,  by Theorem \ref{cd-lemma for Lb}, $\Irr(R_{(\lfloor\  \rfloor,
\cdot)})=A\cup\Irr(R)$ is a linear basis of~$E_{(\mathfrak{A}, \mathfrak{B},\lfloor\
\rfloor,\cdot)}= \Lbs( A\cup B|R_{(\lfloor\  \rfloor,
\cdot)})$.

Let $\widetilde{\mathfrak{B}}$ be the subspace of $\mathfrak{E}$ spanned by $\tilde{\theta}(\Irr(R))=\{\tilde{\theta}(\nu)\mid \nu\in \Irr(R)\}$.
Then $\mathfrak{E}=\mathfrak{A}\oplus\widetilde{\mathfrak{B}}$ as a vector space.
In fact, for each~$e\in \mathfrak{E}$, we have  $\pi(e)=\sum \beta_i \nu_i+\Id(R)$, where~$\beta_i\in\kk,\ \nu_i\in\Irr(R)$.
 Furthermore, we get that~$e=e-\sum \beta_i \tilde{\theta}(\nu_i)+\sum \beta_i \tilde{\theta}(\nu_i)$, where~$\pi(e-\sum \beta_i \tilde{\theta}(\nu_i))=0$ in $\mathfrak{B}$. Thus~$e-\sum \beta_i \tilde{\theta}(\nu_i)\in\mathfrak{A}$ and $\sum \beta_i \tilde{\theta}(\nu_i)\in\widetilde{\mathfrak{B}}$. This shows that~$\mathfrak{E}=\mathfrak{A}+\widetilde{\mathfrak{B}}$.
Assume~$x\in\mathfrak{A}\cap\widetilde{\mathfrak{B}}$. Then~$\pi(x)=0$, since~$x\in\mathfrak{A}$. And~$x=\sum \beta_i \tilde{\theta}(\nu_i)\in\widetilde{\mathfrak{B}}$, where~$\beta_i\in\kk,\  \nu_i\in\Irr(R)$. It follows that~$\pi(\sum \beta_i \tilde{\theta}(\nu_i))=\sum \beta_i \nu_i+\Id(R)=0$. Since $\Irr(R)$ is a linear basis of $\mathfrak{B}$, every $\beta_i$ is equal to $0$. Thus $x=0$.

Define a linear map
$$
\varphi: E_{(\mathfrak{A}, \mathfrak{B},\lfloor\  \rfloor, \cdot )}=\mathfrak{A}\oplus
\mathfrak{B}\rightarrow  \mathfrak{E}=\mathfrak{A}\oplus \widetilde{\mathfrak{B}},\
 a\mapsto a, \  \nu\mapsto  \tilde{\theta}(\nu),\ \ a \in A,\ \nu\in \Irr(R).
$$
It is clear that $\varphi$ is a surjective map. We claim that $\varphi$ is also injective. If there is an element~$\sum \alpha_i a_i+\sum \beta_j \nu_j\in E_{(\mathfrak{A}, \mathfrak{B},\lfloor\
\rfloor,\cdot)}$, where $\alpha_i,\beta_j\in\kk, a_i\in A, \nu_j\in\Irr(R)$,  such that~$\varphi(\sum \alpha_i a_i+\sum \beta_j \nu_j)=0$, then~$\sum \alpha_i a_i=0$ and~$\sum \beta_j \tilde{\theta}(\nu_j)=0$. It follows that $\pi(\sum \beta_j \tilde{\theta}(\nu_j))=\sum \beta_j \nu_j+\Id(R)=0$. Thus every~$\alpha_i=0$ and~$\beta_j=0$. So $\ker \varphi=\{0\}$.
Since $\tilde{\theta}$ preserves parity, $\varphi$ preserves parity.
 For all terms~$\nu,\nu'\in \Irr(R)$, by Lemma \ref{f=Irr+n-s-polynomials}, we have $
(\nu\nu')=\sum \delta_i\nu_i+\sum \beta_j h_{j, g_j }
$ in $\Lbs(B)$, where $\delta_i, \beta_j\in k$,  $\nu_i\in \Irr(R)$, and each~$h_{j, g_j}$ is normal $R$-polynomial with $g_j\in R$.
Then
$$
(\tilde{\theta}(\nu)\tilde{\theta}(\nu'))=\tilde{\theta}((\nu\nu'))
=\sum \delta_i \tilde{\theta}(\nu_i)+\sum \beta_j \tilde{\theta}(h_{j, g_j }).
$$
Noting that in $ \mathfrak{E}=\mathfrak{A}\oplus
\widetilde{\mathfrak{B}} $,
for all~$a,a'\in A$ and for all~$\nu, \nu'\in \Irr(R)$,
\begin{multline*}
 (a+\tilde{\theta}(\nu))(a'+\tilde{\theta}(\nu'))
 = (aa')+ (\tilde{\theta}(\nu)a')
+(a\tilde{\theta}(\nu'))+\sum \delta_i \tilde{\theta}(\nu_i)+\sum \beta_j \tilde{\theta}(h_{j, g_j })\\
 = \{a\cdot a'\}+\{\nu\cdot a'\} +\{a\cdot \nu'\}+\sum \delta_i \tilde{\theta}(\nu_i)+\sum \beta_j h_{j,\lfloor g_j  \rfloor}.
\end{multline*}

Comparing the multiplication of~$E_{(\mathfrak{A}, \mathfrak{B},\lfloor\  \rfloor,\cdot)}$ in (\ref{5.6multi}), $\varphi$ is an isomorphism.
\end{proof}

By   Lemmas \ref{le2.3} and   \ref{le2.4}, we have the following theorem which gives a complete characterization of   extensions of    $\mathfrak{B}$ by     $\mathfrak{A}$.
\begin{theorem}\label{th2.5}
Let $(\mathfrak{A}=\mathfrak{A}_0\oplus\mathfrak{A}_1, (--))$  be   a Leibniz superalgebra with  a (well-ordered) linear basis
$A=A_0\cup A_1$ and the   multiplication table: $(aa')=\{a\cdot a'\}, a, a'\in A$,
where $A_i$ is a linear basis of $\mathfrak{A}_i$, $i=0,1$.
Let
$\mathfrak{B}=\Lbs(B| R)$ be   a Leibniz superalgebra
generated by a well-ordered set $B=B_0\cup B_1$ with defining relations $R$,
where $R$  is a reduced Gr\"{o}bner-Shirshov basis in
$\Lbs(B)$ and the length of leading monomial of each polynomial in $R$ is greater than 1. Then the following statement holds.

A Leibniz superalgebra  $\mathfrak{E}$ is an extension of
$\mathfrak{B}$ by $\mathfrak{A}$ if and only if
$$
\mathfrak{E}\cong E_{ (\mathfrak{A}, \mathfrak{B},\lfloor\  \rfloor,\cdot)}= \Lbs\left(
 A\cup B
 \left|
 \begin{array}{ll}
 (aa')-\{a\cdot a'\},  & a, a'\in A  \\
 (ab)-\{a\cdot b\},  (\nu a)-\{\nu\cdot a\},   & a \in A,\ b\in B,\ \nu \in \NF(B) \\
 f- \lfloor f \rfloor, & f \in R
 \end{array}
\right.\right)
$$
for some
compatible $\Lbs(B)$-supermodule  structure on   $\mathfrak{A}$ (denote  the supermodule operation by $\cdot$)  and
some  map $ \lfloor\  \rfloor : R \rightarrow \mathfrak{A},\ f\mapsto  \lfloor f  \rfloor$
satisfying~$|f|=  |\lfloor f  \rfloor|$ such that
\begin{enumerate}
\item[(i)]
 for all $a\in A, f\in R$,
 $
a\cdot f=\{ a\cdot\lfloor f\rfloor\}
$
and~$
{f \cdot a=\{\lfloor f\rfloor \cdot a\}}
$ in $\mathfrak{A}$;

\item[(ii)] for all $\mu\in \NF(B), f\in R$, and $(\mu\ff)=\sum \alpha_{i}h_{i,s_i}$ (by Lemma \ref{le2.2}), we have
  $$
\mu\cdot\lfloor f\rfloor=\sum \alpha_i h_{i,\lfloor s_i\rfloor}\ \mbox{holds in}\ \mathfrak{A},
$$
\end{enumerate}
where  for any $x\in\mathfrak{A}$, $y\in\Lbs(B)$, $\{x\cdot y\}$ and $\{y\cdot x\}$ are linear combinations of elements in $A$, and each $h_{i,\lfloor s_i\rfloor}=\{\{\{\lfloor s_i\rfloor\cdot b_{i_1}\}\cdots\} \cdot b_{i_{n_{i}}}\}$ if $h_{i,s_i}=[s_i b_{i_1}\cdots b_{i_{n_{i}}}]_{_L}$.

If this is the case, then $\mathfrak{E}\cong E_{(\mathfrak{A}, \mathfrak{B},\lfloor\  \rfloor,\cdot)}=\mathfrak{A}\oplus\mathfrak{B}$ as a vector space, where $\mathfrak{B}$ is the vector space spanned by $\Irr(R)$
and the multiplication in~$E_{(\mathfrak{A}, \mathfrak{B},\lfloor\  \rfloor,\cdot)}$  is defined as below: for all
$a,a'\in A$ and~$\nu,\nu'\in \Irr(R)$,
$$
((a+\nu)(a'+\nu'))=\{a\cdot a'\}+\{a\cdot \nu'\}+\{\nu\cdot a'\}+\sum \beta_j h_{j,\lfloor g_j\rfloor}+\sum \delta_i\nu_i,
$$
where, in
$\Lbs( B)$, we have unique $\delta_i,\beta_j\in k$, $\nu_i\in \Irr(R)$ and unique normal $R$-polynomials $h_{j, g_j}$, $g_j\in R$, such that
 $(\nu\nu')=\sum \delta_i\nu_i+\sum \beta_j h_{j, g_j}. $
\end{theorem}

The following theorem give a specific characterization of extensions of $\mathfrak{B}$ by $\mathfrak{A}$, when $\mathfrak{B}$ is presented by a linear basis and its multiplication table.

 \begin{theorem} \label{th2.9}
 Let $(\mathfrak{A}=\mathfrak{A}_0\oplus\mathfrak{A}_1, (--))$ (resp. $(\mathfrak{B}=\mathfrak{B}_0\oplus\mathfrak{B}_1, (--))$) be a Leibniz superalgebra with a linear basis $A=A_0\cup A_1$ (resp. $B=B_0\cup B_1$), where $A_i$ (resp. $B_i$) is a linear basis of $\mathfrak{A}_i$ (resp. $\mathfrak{B}_i$), $i= 0,1$  and the
 multiplication table: $(aa')=\{a\cdot a'\}, a, a'\in A$ (resp. $(bb')=\{b\cdot b'\}, b,b'\in B$).
 Then a Leibniz superalgebra $\mathfrak{E}$ is an extension of $\mathfrak{B}$ by $\mathfrak{A}$ if and only if
  $$
\mathfrak{E}\cong  \Lbs\left(
 A\cup B
 \left|
 \begin{array}{ll}
 (aa')-\{a\cdot a'\},  & a, a'\in A  \\
 (ab)-\{a\cdot b\},\ (b a)-\{b\cdot a\}, & a\in A, b\in B\\
 (bb')-\{b\cdot b'\}-\lfloor b,b'\rfloor, & b, b'\in B
 \end{array}
\right. \right)
$$
for some  compatible $\Lbs(B)$-supermodule  structure on  $\mathfrak{A}$ (denote  the supermodule operation by $\cdot$)
 and  some
bilinear map $\lfloor , \rfloor : \mathfrak{B}\times \mathfrak{B} \rightarrow \mathfrak{A}$ with $|\lfloor b,b' \rfloor|=|b|+|b'|$, $b,b'\in B$
such that

\ITEM1 for all $a\in A, b,b'\in B$,
$
a\cdot (bb')-a\cdot \{b\cdot b'\}= \{a\cdot \lfloor b, b'\rfloor\}$ and $(bb')\cdot a-\{b\cdot b'\}\cdot a=\{\lfloor b,b'\rfloor\cdot a\}
$ in $\mathfrak{A}$;

\ITEM2 for all $b,b',b''\in B$, in $\mathfrak{A}$,
\begin{align*}
\{\lfloor b'',b\rfloor\cdot b'\}-(-1)^{|b||b'|}\{\lfloor b'',b'\rfloor\cdot b\}-\{b''\cdot \lfloor b,b'\rfloor\}\\
+\lfloor\{b''\cdot b\},b'\rfloor-(-1)^{|b||b'|}\lfloor\{b''\cdot b'\},b\rfloor-\lfloor b'',\{b\cdot b'\}\rfloor=0,
\end{align*}
where  for any $x\in\mathfrak{A}$, $y\in\Lbs(B)$, $\{x\cdot y\}$ and $\{y\cdot x\}$ are linear combinations of elements in~$A$.

If this is the case, then $\mathfrak{E}\cong E_{(\mathfrak{A}, \mathfrak{B},\lfloor\  \rfloor,\cdot)}=\mathfrak{A}\oplus\mathfrak{B}$ as a vector space
and the multiplication in~$E_{(\mathfrak{A}, \mathfrak{B},\lfloor\  \rfloor,\cdot)}$  is defined as below: for all
$a,a'\in A$ and~$\bb,\bb'\in B$,
\begin{align*}
((a+\bb)(a'+\bb'))=\{a\cdot a'\}+\{a\cdot \bb'\}+\{\bb\cdot a'\}+\lfloor\bb,\bb' \rfloor+\{\bb\cdot \bb'\}.
\end{align*}
\end{theorem}

\begin{proof}
Let $R=\{(bb')-\{b\cdot b'\}\mid b,b'\in B\}$ and $B$ be a well-ordered set with $\deg(b)=1$ for all $b\in B$. It is easy to know that with the deg-length-lex order, $R$ is a reduced \gsb\ in $\Lbs(B)$. Let  $\lfloor \ \rfloor: R\rightarrow \mathfrak{A},\ (bb')-\{b\cdot b'\}\mapsto \lfloor b,b' \rfloor$. Then $\lfloor\  \rfloor $ is a factor set of $\mathfrak{B}$ in $\mathfrak{A}$.
According to Lemmas \ref{lemma4.16} and   \ref{le2.4},  we only need to show that the condition $(ii)'$ in  Lemma \ref{lemma4.16} is now the identity in $(ii)$.

For all $b''\in B,\ f=(bb')-\{b\cdot b'\}\in R$, since $\mathfrak{B}$ is a Leibniz superalgebra, we have $-\{\{b''\cdot b\}\cdot b'\}+(-1)^{|b||b'|}\{\{b''\cdot b'\}\cdot b\}-\{b''\cdot \{b\cdot b'\}\}=0$ in~$\mathfrak{B}$. Thus,
\begin{align*}
(b'' f)&=(b'' f)_{_R}=((b'' b)b')-(-1)^{|b||b'|}((b'' b')b)-(b''\{b\cdot b'\})\\
&=(((b'' b)-\{b''\cdot b\})b')+((\{b''\cdot b\}b')-\{\{b''\cdot b\}\cdot b'\})+\{\{b''\cdot b\}\cdot b'\}\\
&\ \ \ -(-1)^{|b||b'|}((((b'' b')-\{b''\cdot b'\})b)+((\{b''\cdot b'\}b)-\{\{b''\cdot b'\}\cdot b\})+\{\{b''\cdot b'\}\cdot b\}
)\\
&\ \ \ -((b''\{b\cdot b'\})-\{b''\cdot\{b\cdot b'\}\})-\{b''\cdot\{b\cdot b'\}\}\\
&=(((b'' b)-\{b''\cdot b\})b')+((\{b''\cdot b\}b')-\{\{b''\cdot b\}\cdot b'\})\\
&\ \ \ -(-1)^{|b||b'|}((((b'' b')-\{b''\cdot b'\})b)+((\{b''\cdot b'\}b)-\{\{b''\cdot b'\}\cdot b\})
)\\
&\ \ \ -((b''\{b\cdot b'\})-\{b''\cdot\{b\cdot b'\}\}),
\end{align*}
where
$(((b'' b)-\{b''\cdot b\})b')$, $(\{b''\cdot b\}b')-\{\{b''\cdot b\}\cdot b'\},\ (((b'' b')-\{b''\cdot b'\})b),$ ${(\{b''\cdot b'\}b)-\{\{b''\cdot b'\}\cdot b\}}$ and
 $(b''\{b\cdot b'\})-\{b''\cdot\{b\cdot b'\}\}$ are normal $R$-polynomials in~$\Lbs(B)$. It follows that
\begin{align*}
&\ \ \ b''\cdot\lfloor f\rfloor=\{b''\cdot\lfloor b,b'\rfloor\}\\
&=\{\lfloor b'', b\rfloor\cdot b'\}-(-1)^{|b||b'|}\{\lfloor b'', b'\rfloor\cdot b\}
+\lfloor\{b''\cdot b\},b'\rfloor-(-1)^{|b||b'|}\lfloor\{b''\cdot b'\},b\rfloor  -\lfloor b'',\{b\cdot b'\}\rfloor.
\end{align*}

The proof is completed.
\end{proof}

 Now, we consider a special case of Theorems \ref{th2.5} and \ref{th2.9} when  $\mathfrak{A}^2=0$. Furthermore, the extension $\mathfrak{E}$ obtained in following corollary is isomorphic to the second cohomology group. 

\begin{coro}\label{th2.8} Let $(\mathfrak{A}=\mathfrak{A}_0\oplus\mathfrak{A}_1, (--))$  be  an abelian Leibniz superalgebra with  a (well-ordered) linear basis
$A=A_0\cup A_1$,
where $A_i$ is a linear basis of $\mathfrak{A}_i$, $i=0,1$.
Let
$\mathfrak{B}=\Lbs(B| R)=\Lbs(B)/\Id(R)$ be a Leibniz superalgebra
generated by a well-ordered set $B=B_0\cup B_1$ with defining relations $R$,
where $R$  is a reduced Gr\"{o}bner-Shirshov basis in
$\Lbs(B)$ and the length of leading monomial of each polynomial in $R$ is greater than 1.
Then a Leibniz superalgebra $\mathfrak{E}$ is an extension of $\mathfrak{B}$ by $\mathfrak{A}$ if and only if
  $$
\mathfrak{E}\cong E_{(\mathfrak{A}, \mathfrak{B},\lfloor\  \rfloor,\cdot)}= \Lbs\left(
 A\cup B
 \left|
 \begin{array}{ll}
 (aa'),  & a, a'\in A  \\
 (ab)-\{a\cdot b\},  & a \in A,  \bb \in B \\
 (\nu a)-\{\nu\cdot a\},  & a \in A,  \nu \in \NF(B) \\
f-\lfloor f\rfloor,  & f \in R
 \end{array}
\right. \right)
$$
for some
  $\mathfrak{B}$-supermodule structure on $\mathfrak{A}$  (denote the supermodule operation by $\cdot$) and    some  map $ \lfloor\  \rfloor : R \rightarrow \mathfrak{A},\ f\mapsto  \lfloor f  \rfloor$
satisfying~$|f|=  |\lfloor f  \rfloor|$ such that
for all $\mu\in \NF(B), f\in R$, and $(\mu\ff)=\sum \alpha_{i}h_{i,s_i}$ (by Lemma \ref{le2.2}), we have
  $$
\mu\cdot\lfloor f\rfloor=\sum \alpha_i h_{i,\lfloor s_i\rfloor} \ \mbox{holds in}\ \mathfrak{A},
$$
where for any $x\in\mathfrak{A}$, $y\in\NF(B)$, $x\cdot y:=x\cdot (y+\Id(R))$, $y\cdot x:=(y+\Id(R))\cdot x$; and $\{x\cdot y\}$, $\{y\cdot x\}$ are linear combinations of elements in $A$,
and each $h_{i,\lfloor s_i\rfloor}=\{\{\{\lfloor s_i\rfloor\cdot b_{i_1}\} \cdots\} \cdot b_{i_{n_{i}}}\}$ if $h_{i,s_i}=[s_i b_{i_1}\cdots b_{i_{n_{i}}}]_{_L}$.

If this is the case, then $\mathfrak{E}\cong E_{(\mathfrak{A}, \mathfrak{B},\lfloor\  \rfloor,\cdot)}=\mathfrak{A}\oplus\mathfrak{B}$ as a vector space, where $\mathfrak{B}$ is the vector space spanned by $\Irr(R)$
and the multiplication in~$E_{(\mathfrak{A}, \mathfrak{B},\lfloor\  \rfloor,\cdot)}$  is defined as below: for all
$a,a'\in A$ and~$\nu,\nu'\in \Irr(R)$,
$$
((a+\nu)(a'+\nu'))=\{a\cdot \nu'\}+\{\nu\cdot a'\}+\sum \beta_j h_{j,\lfloor g_j\rfloor}+\sum \delta_i\nu_i,
$$
where, in
$\Lbs( B)$, we have unique $\delta_i,\beta_j\in k$, $\nu_i\in \Irr(R)$ and unique normal $R$-polynomials $h_{j, g_j}$, $g_j\in R$, such that
 $(\nu\nu')=\sum \delta_i\nu_i+\sum \beta_j h_{j, g_j}. $
\end{coro}
\begin{proof}
Suppose that $\mathfrak{A}$ is a $\mathfrak{B}$-supermodule with the given property. Then   $\mathfrak{A}$ is a $\Lbs(B)$-supermodule if we define
$$
a\cdot g:=a\cdot (g+Id(R)), \ g\cdot a:= (g+Id(R))\cdot a, \ \mbox{for\ all}\ a\in A, g\in \Lbs(B).
$$
Now, since $\mathfrak{A}^2=0$, for all $a\in A, f\in R$, we have
$$
a\cdot f=a\cdot (f+Id(R))=0=\{ a\cdot\lfloor f\rfloor\},\ \mbox{and}\ f \cdot a= (f+Id(R))\cdot a=0=\{\lfloor f\rfloor \cdot a\}
$$
in $\mathfrak{A}$. Thus, the conditions (i) and (ii) in Theorem \ref{th2.5} hold. It follows that $\mathfrak{E}$ is an extension of $\mathfrak{B}$ by $\mathfrak{A}$.

Conversely, suppose $\mathfrak{E}$ is an extension of $\mathfrak{B}$ by $\mathfrak{A}$. Then by Theorem \ref{th2.5}, there is a $\Lbs(B)$-supermodule on $\mathfrak{A}$ and a factor set such that (i) and (ii) in Theorem \ref{th2.5} hold. Define
$$
a\cdot (g+Id(R)):=a\cdot g, \ (g+Id(R))\cdot a:= g\cdot a , \ \mbox{for\ all}\ a\in A, g\in \Lbs(B).
$$
Then $\mathfrak{A}$ is a $\mathfrak{B}$-supermodule and $(\mu+\Id(R))\cdot\lfloor f\rfloor=\sum \alpha_i h_{i,\lfloor s_i\rfloor} \ \mbox{holds in}\ \mathfrak{A}$.
Now the result follows by Lemma \ref{le2.4}.
\end{proof}

 By Theorem \ref{th2.9} and a similar proof of Corollary \ref{th2.8}, we have the following corollary.

\begin{coro}\label{th2.8'}
 Let $(\mathfrak{A}=\mathfrak{A}_0\oplus\mathfrak{A}_1, (--))$ (resp. $(\mathfrak{B}=\mathfrak{B}_0\oplus\mathfrak{B}_1, (--))$) be a Leibniz superalgebra with a linear basis $A=A_0\cup A_1$ (resp. $B=B_0\cup B_1$), where $A_i$ (resp. $B_i$) is a linear basis of $\mathfrak{A}_i$ (resp. $\mathfrak{B}_i$), $i= 0,1$  and the
 multiplication table: $(aa')=0, a, a'\in A$ (resp. $(bb')=\{b\cdot b'\}, b,b'\in B$).
Then
a Leibniz superalgebra $\mathfrak{E}$  is an extension  of $\mathfrak{B}$ by $\mathfrak{A}$  if and only if
$$
\mathfrak{E}\cong  \Lbs\left(
 A\cup B
 \left|
 \begin{array}{ll}
 (aa'),  & a, a'\in A  \\
 (ab)-\{a\cdot b\},\ (b a)-\{b\cdot a\}, & a\in A, b\in B\\
 (bb')-\{b\cdot b'\}-\lfloor b,b'\rfloor, & b, b'\in B
 \end{array}
\right. \right)
$$
for some
bilinear map $\lfloor , \rfloor : \mathfrak{B}\times \mathfrak{B} \rightarrow \mathfrak{A}$ with $|\lfloor b,b' \rfloor|=|b|+|b'|$, $b,b'\in B$
such that for all $b,b',b''\in B$,
\begin{align*}
\{\lfloor b'',b\rfloor\cdot b'\}-(-1)^{|b||b'|}\{\lfloor b'',b'\rfloor\cdot b\}-\{b''\cdot \lfloor b,b'\rfloor\}\\
+\lfloor\{b''\cdot b\},b'\rfloor-(-1)^{|b||b'|}\lfloor\{b''\cdot b'\},b\rfloor-\lfloor b'',\{b\cdot b'\}\rfloor=0,
\end{align*}
where  for any $x\in\mathfrak{A}$, $y\in\mathfrak{B}$, $\{x\cdot y\}$ and $\{y\cdot x\}$ are linear combinations of elements in~$A$.

If this is the case, then $\mathfrak{E}\cong E_{(\mathfrak{A}, \mathfrak{B},\lfloor\  \rfloor,\cdot)}=\mathfrak{A}\oplus\mathfrak{B}$ as a vector space
and the multiplication in~$E_{(\mathfrak{A}, \mathfrak{B},\lfloor\  \rfloor,\cdot)}$  is defined as below: for all
$a,a'\in A$ and~$\bb,\bb'\in B$,
\begin{align*}
((a+\bb)(a'+\bb'))=\{a\cdot \bb'\}+\{\bb\cdot a'\}+\lfloor\bb,\bb' \rfloor+\{\bb\cdot \bb'\}.
\end{align*}
 \end{coro}

\end{document}